\documentclass[10pt]{amsart}

\usepackage{epsfig,amssymb}
\usepackage{amsmath,amsthm}
\usepackage{tikz,mathrsfs}
\usepackage[colorlinks]{hyperref}
\usepackage{hypcap}

\newcommand{\iso}{\cong}

\newcommand{\then}{\ensuremath{\Longrightarrow}}
\renewcommand{\iff}{\ensuremath{\Longleftrightarrow}}

\DeclareMathAlphabet{\mathpzc}{OT1}{pzc}{m}{it}

\DeclareMathOperator{\modules}{mod} \renewcommand{\mod}{\modules}

\DeclareMathOperator{\proj}{proj}

\DeclareMathOperator{\add}{add}

\DeclareMathOperator{\stabmod}{\underline{mod}}

\DeclareMathOperator{\costmod}{\overline{mod}}

\DeclareMathOperator{\End}{End}
\DeclareMathOperator{\Hom}{Hom}
\DeclareMathOperator{\Ext}{Ext}

\DeclareMathOperator{\Rad}{Rad}
\DeclareMathOperator{\Soc}{Soc}

\DeclareMathOperator{\Cok}{Cok}

\DeclareMathOperator{\Cone}{Cone}

\DeclareMathOperator{\gld}{gl.\!dim}

\newenvironment{smallpmatrix}
	       {\left( \! \begin{smallmatrix}}
	       {\end{smallmatrix} \! \right)}

\def\mathclap{\mathpalette\mathclapinternal}
\def\mathclapinternal#1#2{\clap{$\mathsurround=0pt#1{#2}$}}
\def\clap#1{\hbox to 0pt{\hss#1\hss}}

\newcommand{\leftsub}[2]{{\vphantom{#2}}_{#1}{#2}}
\newcommand{\leftsup}[2]{{\vphantom{#2}}^{#1}{#2}}

\usetikzlibrary{arrows,decorations.pathmorphing,decorations.pathreplacing,positioning}

\tikzset{>=stealth',
         vertex/.style={circle,draw=black,inner sep=1.5pt,outer sep=2pt},
         tvertex/.style={inner sep=1pt,font=\scriptsize},
         gap/.style={fill=white,inner sep=1pt}}

\newcommand{\arrow}[2][20]
 {
  \hspace{-5pt}
  \begin{tikzpicture}
   \node (A) at (0,0) {};
   \node (B) at (#1pt,0) {};
   \draw [#2] (A) -- (B);
  \end{tikzpicture}
  \hspace{-5pt}
 }

\newcommand{\arrowl}[3][20]
 {
  \hspace{-5pt}
  \begin{tikzpicture}
   \node (A) at (0,0) {};
   \node (B) at (#1pt,0) {};
   \draw [#2] (A) -- node [above] {$#3$} (B);
  \end{tikzpicture}
  \hspace{-5pt}
 }

\renewcommand{\to}[1][20]{\arrow[#1]{->}}
\newcommand{\tol}[2][20]{\arrowl[#1]{->}{#2}}

\newcommand{\epi}[1][20]{\arrow[#1]{->>}}

\newcommand{\mono}[1][20]{\arrow[#1]{>->}}

\newcommand{\sub}[1][20]{\arrow[#1]{right hook->}}

\newtheorem{theorem}{Theorem}[section]

\newtheorem{corollary}[theorem]{Corollary}
\newtheorem{lemma}[theorem]{Lemma}
\newtheorem{proposition}[theorem]{Proposition}

\theoremstyle{definition}
\newtheorem{definition}[theorem]{Definition}
\newtheorem{remark}[theorem]{Remark}
\newtheorem{example}[theorem]{Example}
\newtheorem{construction}[theorem]{Construction}
\newtheorem{observation}[theorem]{Observation}
\newtheorem{notation}[theorem]{Notation}

\newcommand{\Index}[2][d]{\mathbf{I}^{#1}_{#2}}
\newcommand{\IndexC}[2][d]{\leftsup{\circlearrowleft}{\mathbf{I}}^{#1}_{#2}}

\newtheorem{theoconst}[theorem]{Theorem / Construction}

\renewcommand{\L}{{\mathcal L}}
\newcommand{\sort}{{\operatorname{sort}}}
\newcommand{\nwr}{\mbox{$\not\wr\,$}}

\title[Higher dimensional clusters]{Higher dimensional cluster combinatorics and representation theory}

\author{Steffen Oppermann}
\address{Steffen Oppermann, Institutt for matematiske fag, NTNU, 7491 Trondheim, Norway}
\email{steffen.oppermann@math.ntnu.no}

\author{Hugh Thomas}
\address{Hugh Thomas, Department of Mathematics and Statisics, Univeristy of New Brunswick, Fredericton NB, E3B~1J4 Canada}
\email{hthomas@unb.ca}

\begin{document}

\begin{abstract}
Higher Auslander algebras were introduced by Iyama generalizing classical concepts from representation theory of finite dimensional algebras. Recently these higher analogues of classical representation theory have been increasingly studied.
Cyclic polytopes are classical objects of study in convex geometry.  In particular, their triangulations have been studied with a view towards generalizing the rich combinatorial structure of triangulations of polygons. In this paper, we demonstrate a connection between these two seemingly unrelated subjects.

We study triangulations of even-dimensional cyclic polytopes and tilting modules for higher Auslander algebras of linearly oriented type $A$ which are summands of the cluster tilting module. We show that such tilting modules correspond bijectively to triangulations. Moreover mutations of tilting modules correspond to bistellar flips of triangulations.

For any $d$-representation finite algebra we introduce a certain $d$-dimensional cluster category and study its cluster tilting objects. For higher Auslander algebras of linearly oriented type $A$ we obtain a similar correspondence between cluster tilting objects and triangulations of a certain cyclic polytope.

Finally we study certain functions on generalized laminations in cyclic polytopes, and show that they satisfy analogues of tropical cluster exchange relations. Moreover we observe that the terms of these exchange relations are closely related to the terms occuring in the mutation of cluster tilting objects.
\end{abstract}

\maketitle

\section{Introduction}

Cluster algebras have been the subject of intensive research since their introduction some ten years ago by Fomin and Zelevinsky \cite{FZ}.  The two best-understood families of cluster algebras are those which admit a categorification (as in \cite{BMRRT, GLS}, or, generalizing both, \cite{CC,C_PhD}), and those which arise from a surface with boundary (\cite{FST}, building on \cite{GSV,FG}).  Both these families of cluster algebras have a significant ``2-dimensional'' quality.  In the case of the cluster algebras with a categorification, this is present in a certain 2-Calabi-Yau property in the associated category.  In the case of cluster algebras arising from surfaces, there is the 2-dimensionality of the surface.  It is natural to ask if similar constructions exist in higher dimensions.

In order to pursue this question, it is natural to begin by focussing attention on the cluster algebras of type $A_n$, which fit within both the families of cluster algebras mentioned above.  Their categorification is based on the representation theory of the path algebra of an $A_n$ quiver; the surface to which they correspond is a disk with $n+3$ marked points on the boundary.  In seeking a higher dimensional cluster theory, one would expect to replace the path algebra by an algebra of higher global dimension, and the disk by some higher dimensional space.  The difficulty is to determine appropriate candidates for these roles.  A reasonable criterion by which to justify such choices would be evidence of similar structures arising in the higher dimensional algebra and the higher dimensional geometry.  
 
In the present paper, we exhibit links along these lines.  Our replacement for the path algebra of $A_n$ is the $(d-1)$-fold higher Auslander algebra of the (linearly oriented) path algebra of type $A_n$. (These higher Auslander algebras 
were introduced in \cite{Iy_n-Auslander}).  This is a prototypical $d$-representation finite algebra. Our work can therefore be viewed as a part of the recent development of the theory of $d$-representation finite algebras. This class of algebras, which has been introduced in \cite{IO}, is a natural generalization of representation finite hereditary algebras. Many higher dimensional analogues of classical results for representation finite hereditary algebras have been shown to hold for this class of algebras (see \cite{Iy_n-Auslander,IO,IO2}).  Indeed, part of the present paper (Section 5) is carried out for general $d$-representation finite algebras.  There we introduce for such algebras a $d$-dimensional analogue of the classical cluster categories.  It shares many properties of the classical cluster categories, and in particular, there is a well behaved notion of cluster tilting. 

Our replacement for the disk is a $2d$-dimensional cyclic polytope. 
The cyclic polytope $C(m,\delta)$ is a certain polytope in 
$\mathbb R^\delta$ with $m$ vertices.  
Cyclic 
polytopes have been extensively studied in convex geometry, going back 
to \cite{Car} in 1911.  For an
introduction, see \cite[Chapter VI]{bar}.
A \emph{triangulation} of $C(m,\delta)$ is a subdivision
of $C(m,\delta)$ into $\delta$-dimensional 
simplices, each of whose vertices is a vertex of the
polytope. Triangulations of cyclic polytopes have been studied with a
view to extending to that setting some of the rich structure of triangulations of 
convex polygons (see \cite{KV,ER}) and as a testing-ground for more general
convex-geometric conjectures (see \cite{ERR,RSa}).
We shall be concerned only with the case $\delta=2d$ is even.  
We refer to the $d$-dimensional simplices of $C(m,2d)$ which do not 
lie on the boundary of $C(m,2d)$ as the {\em internal $d$-simplices} of
$C(m,2d)$.  We will give a new combinatorial description of
the set of triangulations of $C(m,2d)$ by characterizing the sets of 
internal $d$-simplices of $C(m,2d)$ which can arise as the set of
internal $d$-faces of a triangulation.  By a result of Dey \cite{dey}, the
$d$-dimensional faces of the triangulation uniquely determine the triangulation.

\subsection{\texorpdfstring{$d$}{d}-representation finite algebras}
Iyama has introduced higher dimesional analogues of Auslander-Reiten theory and Auslander algebras (see for instance \cite{Iy_n-Auslander}), generalizing these classical concepts from the representation theory of finite dimensional algebras. In \cite{IO} the notion of $d$-representation finiteness was introduced as an ideal setup for studying these concepts. (See also \cite{Iy_n-Auslander, IO2, HI} for further results illustrating this philosophy.)
In this paper we will mostly focus on the currently best-understood class of $d$-representation finite algebras: the $(d-1)$-st higher Auslander algebras of linearly oriented $A_n$ (see \cite{Iy_n-Auslander,IO}); these algebras will be called $A_n^d$.

Inside the module category of $A_n^d$, there is a unique 
$d$-cluster tilting module, $\leftsub{A_n^d}{M}$.  We have the following result:
\begin{theorem}[see Theorems~\ref{theo.indexing} and \ref{central}] \label{thm.intro_corr_mod}
There is a bijection
\[ \left\{ \begin{matrix} \text{internal } d\text{-simplices} \\ \text{of } C(n+2d,2d) \end{matrix} \right\} \arrow[30]{<->} \left\{ \begin{matrix} \text{indecomposable} \\ \text{non-projective-injective} \\ \text{direct summands of } \leftsub{A^d_n}{M} \end{matrix} \right\} \]
which induces a bijection 
\[\left\{ \begin{matrix}\text{triangulations of}\\C(n+2d,2d) \end{matrix}\right\} 
\arrow[30]{<->}
\left\{\begin{matrix}\text{basic tilting modules for $A_n^d$}\\ \text{contained in }\add \leftsub{A_n^d}M\end{matrix}\right\} \]
\end{theorem}

The case when $d=1$ was already understood.  
Here, the algebra $A_n^1$ is the path algebra of linearly ordered $A_n$.  
The $1$-cluster tilting module $\leftsub{A^1_n}{M}$ is an additive generator of the whole module category $\mod A_n^1$.  
So the $d=1$ case of the previous result is the
fact that the non-projective-injective indecomposable modules 
of $A_n^1$ can be identified
with the diagonals of an $(n+2)$-gon in such a way that tilting modules for
$A_n^1$ correspond to triangulations.

\subsection{A cluster category} 
We consider another representation-theoretic setup which is similar but in some ways preferable to the one discussed above. We introduce a higher dimensional cluster category called $\mathscr{O}_{\Lambda}$ for any $d$-representation finite algebra $\Lambda$. 

Note that a different notion of higher cluster category, the $d$-cluster category $\mathscr{C}_H^d$, for $H$ a hereditary algebra, has been studied (see \cite{BM} and references therein). However $\mathscr{C}_H^d$ is not higher dimensional in the sense of this paper: $\mathscr{C}_H^d$ is defined for hereditary algebras, it has a two-dimensional Auslander-Reiten quiver (see \cite{Iy_n-Auslander})
and the combinatorics of its cluster tilting objects is modelled by subdivisions of polygons in the plane.

Our cluster category $\mathscr{O}_{\Lambda}$ is constructed as a subcategory of the $2d$-Amiot cluster category $\mathscr{C}_{\Lambda}^{2d}$. $2d$-Amiot cluster categories are a generalization of classical $2d$-cluster categories to not necessarily hereditary algebras $\Lambda$; in particular the categories $\mathscr{C}_{\Lambda}^{2d}$ are $2d$-Calabi-Yau and triangulated. These properties of $\mathscr{C}_{\Lambda}^{2d}$ will be used to show that $\mathscr{O}_{\Lambda}$ is $(d+2)$-angulated, and also satisfies a certain Calabi-Yau property. It should be noted that for $d=1$ the two categories $\mathscr{O}_{\Lambda}$ and $\mathscr{C}_{\Lambda}^2$ coincide, and also coincide with the classical cluster category of the hereditary representation finite algebra $\Lambda$. For all $d$-representation finite algebras $\Lambda$, the category $\mathscr{O}_{\Lambda}$ contains only finitely many indecomposable objects, which can be arranged in a $d$-dimensional analogue of an Auslander-Reiten quiver.

For the case $\Lambda = A_n^d$ we obtain the following analogue of Theorem~\ref{thm.intro_corr_mod}:
\begin{theorem}[see Proposition~\ref{prop.indexing_O}(1) and Theorem~\ref{central_cluster}] \label{thm.intro_corr_cluster}
There is a bijection
\[ \left\{ \begin{matrix} \text{internal } d\text{-simplices} \\ \text{of } C(n+2d+1,2d) \end{matrix} \right\} \arrow[30]{<->} \left\{ \begin{matrix} \text{indecomposable} \\ \text{objects in }\mathscr{O}_{A_n^d} \end{matrix} \right\} \]
which induces a bijection
\[\left\{\begin{matrix} \text{triangulations of} \\
C(n+2d+1,2d)\end{matrix}\right\} \arrow[30]{<->} 
\left\{\begin{matrix} \text{basic cluster tilting}\\ \text{objects in }\mathscr{O}_{A^d_n}
\end{matrix}\right\}\]
\end{theorem}

\subsection{Local moves}
There are notions of local moves for 
all of the setups above, that is for triangulations, tilting modules, and cluster tilting objects.

For triangulations, the local move
is called a \emph{bistellar flip}.  It generalizes the operation
on triangulations of polygons which removes one edge of the triangulation and replaces
it by the other diagonal of the resulting quadrilateral.  
We will show that two triangulations $S$ and $T$ of
a cyclic polytope $C(n+2d,2d)$ are related by a bistellar flip if and
only if the collections of $d$-faces of $S$ and $T$ coincide except that
there is one $d$-face present in $S$ which is not in $T$ and vice versa.

The local move for cluster tilting objects is called \emph{mutation}; if $A$ and $B$ are distinct indecomposable objects in $\mathscr{O}_{A_n^d}$, such that $A\oplus \overline T$ and $B\oplus \overline T$ are basic cluster tilting objects in $\mathscr{O}_{A_n^d}$, then there are sequences
of objects and morphisms

\begin{equation}\label{ex1}
A \to[30] E_d \to[30] \cdots \to[30] E_1 \to[30] B
\end{equation}
and
\begin{equation}\label{ex2}
B \to[30] F_1 \to[30] \cdots \to[30] F_d \to[30] A
\end{equation}
\eqref{ex1} and \eqref{ex2} are called \emph{exchange $(d+2)$-angles}; they are distinguished $(d+2)$-angles in the $(d+2)$-angulated structure of $\mathscr{O}_{A_n^d}$. In the $d=1$ case these sequences are the usual exchange triangles for classical cluster categories.

Similarly there is the notion of mutation of tilting modules 
contained in $\add \leftsub{A_n^d}{M}$. In that case one has an exact exchange sequence similar to one of the sequences \eqref{ex1} or \eqref{ex2} (but one has only one sequence, not two).

One main result of this paper is that all these notions of local moves coincide in the following way:
\begin{theorem}[see Theorems~\ref{central_cluster} and \ref{central}]
Under the bijections of Theorems~\ref{thm.intro_corr_cluster} and \ref{thm.intro_corr_mod} bistellar flips of triangulations correspond to mutations of cluster tilting objects and mutations of tilting modules, respectively.
\end{theorem}

\subsection{Tropical cluster exchange relations}
One motivation for this paper is
the fact that, in the $d=1$ 
case, cluster tilting objects of $\mathscr{O}_{A_n^1}$ ($ = \mathscr{C}_{A_n^1}^2$), or equivalently 
triangulations of an $(n+3)$-gon, form a model for the combinatorics of
the $A_n$ cluster algebra in the sense that diagonals of the $(n+3)$-gon,
or equivalently the indecomposamble objects of the cluster category, are in bijection
with the cluster variables in the $A_n$ cluster algebra.  

We might hope that the internal $d$-simplices of $C(n+2d+1,2d)$, which are in bijection with the indecomposable objects in the cluster category $\mathscr{O}_{A_n^d}$,
also correspond to ``cluster variables'' 
in some analogue of a cluster algebra.  At present,
we do not know how this should be interpreted.  
However, we are able to exhibit an analogue of the
tropical cluster exchange relations of \cite{GSV,FT} in our setting.

Let us very briefly recall the tropical cluster algebra of functions on laminations, 
in the rather
special case which is of interest to us.  Fix an $(n+3)$-gon.  A \emph{lamination}
is a collection of lines in the polygon, which do not intersect, and 
which begin and end on the 
boundary of the polygon, and not on any vertex.  Let $\L$ be the set of
laminations.  For any lamination $L\in \L$, and $E$ 
any boundary edge or diagonal of
the polygon, 
there is a well-defined 
number of points of intersection between $L$ and $E$.  

Encode this information by associating to each edge or diagonal $A$ of the
polygon, a function $I_A \colon \L \to \mathbb N$, where $I_A(L)$ is 
the number of intersections between $A$ and $L$.  

These functions satisfy a certain \emph{tropical exchange relation}, 
namely, if $E,F,G,H$ are four sides of a 
quadrilateral in cyclic order, and $A,B$ are the two diagonals, then the
relation is:

$$I_A+I_B=\max(I_E+I_G,I_F+I_H)$$
This relation is the tropicalization of the usual cluster relation in type
$A$ (in the sense that $(\times,+)$ have been replaced by $(+,\max)$).  
Using this relation, and supposing that the functions corresponding to the
edges of a given starting triangulation (including the boundary edges) are
known, one can determine the function corresponding to an arbitrary diagonal
of the polygon.

For general $d$, we define a similar collection of laminations, again denoted
$\L$, and define functions $I_A \colon \L \to \mathbb N$ for each $A$ a 
$d$-simplex of $C(n+2d+1,2d)$ (including boundary $d$-simplices).  
These functions satisfy an exchange relation similar to the tropical
exchange relation above. 
The exchange relation is closely related to the representation-theoretic sequences \eqref{ex1} and \eqref{ex2}:

\begin{theorem}[see Theorem~\ref{th4}]
Let $A$ and $B$ be internal $d$-simplices of $C(n+2d+1,2d)$ such that there exist two triangulations whose $d$-simplices consists of $\overline T \cup \{A\}$ and $\overline T\cup \{B\}$, respectively, for some set $\overline T$. 

Then, if we write $I_{E_i}$ for the sum of the $I_X$ with $X$ a summand of $E_i$ (and similar for $F_i$),
\[ (-1)^{d+1} I_A + I_B = \max \Big(\sum (-1)^{i+1} I_{E_i} + \mbox{boundary terms}, \sum (-1)^{i+1} I_{F_i} + \mbox{boundary terms} \Big). \]
\end{theorem}

Here, ``boundary terms'' refers to a sum of terms $\pm I_X$ with $X$ a boundary $d$-simplex.  Such terms $I_X$ should be thought of as coefficients; 
they are $d$-faces of every triangulation and, as in
the $d=1$ case, there are no corresponding objects in the
cluster category, so they cannot be seen in that setup.

\subsection{Outline} In Section~\ref{cyclic}, we discuss cyclic polytopes and provide
a new combinatorial description of their triangulations in the even-dimensional
case.  In Section~\ref{section.higher_Aus_An}, we discuss the higher Auslander algebras of linearly
oriented $A_n$ and their tilting modules.  In Section~\ref{local}, we compare the 
local moves for triangulations (bistellar flip) and tilting modules (tilting mutation), and show that they
agree.  In Section~\ref{section.cluster_cat}, a cluster category is constructed for any
$d$-representation finite algebra.  We apply the 
construction from Section~\ref{section.cluster_cat} to the higher Auslander algebras of linearly
oriented $A_n$ in Section~\ref{section.cluster_cat_A}.  In Section~\ref{sect.tropical}, we exhibit higher-dimensional 
tropical exchange relations.  In Section~\ref{higherd}, we discuss certain classical
($d=1$) phenomena which do not persist in higher dimensions.  

The initial subsection of each section contains the statements of 
the main results of that section.  Readers who are not interested in
the details of the proofs in a particular section can safely skip all
subsequent subsections.

\subsection*{Acknowledgements}
The authors would like to thank Jesus De Loera,
Osamu Iyama, Vic Reiner, and Francisco Santos for helpful conversations. S.\ O.\ was supported by NFR Storforsk grant no.\ 167130. H.\ T.\ was supported by an NSERC Discovery Grant.  Computer
calculations were performed using the facilities of ACEnet, using
Sage \cite{sage} and the NetworkX package \cite{hss}.   Much of this
work was done during two visits by H.\ T.\ to NTNU; he thanks
the Institutt for Matematiske Fag, and his hosts Idun Reiten and Aslak Bakke Buan.  The paper was completed during a visit by S.\ O.\ to UNB funded by a Harrison McCain Young Scholar award; he thanks the math department there for their hospitality.
%colin ingalls?  robert marsh?  david bremner?

\section{Cyclic polytopes}\label{cyclic}

The moment curve is the curve defined by $p_t=(t,t^2,\dots,t^{\delta})
\subset \mathbb R^{\delta}$, for $t\in\mathbb R$.  Choose $m$ distinct real values,
$t_1<t_2<\dots<t_m$.  The convex hull of $p_{t_1},\dots,p_{t_m}$ is 
a cyclic polytope.  (We will take this as our definition of cyclic polytope,
though sometimes a somewhat more general definition is used.)

We will be interested in triangulations of $C(m,\delta)$.  A triangulation of 
$C(m,\delta)$ is a subdivision of $C(m,\delta)$ into $\delta$-dimensional simplices
whose vertices are vertices of $C(m,\delta)$.  We write $S(m,\delta)$ for the set
of all triangulations of $C(m,\delta)$.  
A triangulation can be specified
by giving the collection of $(\delta+1)$-subsets of $\{1,\dots,m\}$ corresponding to
the $\delta$-simplices of the triangulation.  It turns out that whether or 
not a collection of $(\delta+1)$-subsets of $\{1,\dots,m\}$ forms a triangulation
is independent of the values $t_1<\dots<t_m$ chosen, so, for convenience,
we set $t_i=i$. 
Combinatorial descriptions of
the set of triangulations of $C(m,\delta)$ appear in the literature 
\cite{R,T}, but for our purposes a new description is required.

We will mainly be interested in the case where $\delta=2d$ is even.  
In $\mathbb R^{2d}$, we will refer to \emph{upper} and \emph{lower} 
with respect to the $2d$-th coordinate.  The upper facets of $C(m,2d)$ are
those which divide $C(m,2d)$ from points above it, while the lower facets
of $C(m,2d)$ are those which divide it from points below it.  Each facet
of $C(m,2d)$ is either upper or lower.  

We will be particularly interested in $d$-dimensional simplices whose
vertices are vertices of $C(m,2d)$.  We refer to such $d$-dimensional
simplices as $d$-dimensional simplices in $C(m,2d)$ (leaving unstated
the assumption that their vertices are vertices of $C(m,2d)$).   
By convention, we record such simplices as increasing $(d+1)$-tuples from
$[1,m]= \{1,2,\dots,m\}$.

\begin{lemma}\label{dfaceclass}
 Let $A=(a_0,\dots,a_d)$ be a $d$-simplex in $C(m,2d)$.  
\begin{enumerate}
\item $A$ lies within a lower boundary facet of $C(m,2d)$ iff 
$A$ contains $i$ and $i+1$ for some $i$.  
\item $A$ lies within an upper boundary facet of $C(m,2d)$ and not within
any lower boundary facet iff
$A$ does not contain $i$ and $i+1$ for any $i$,  and contains both 1 and $m$.  
\item Otherwise, the relative interior of $A$ lies in the interior of
$C(m,2d)$.  We refer to such $d$-faces as {\em internal}.
\end{enumerate}
\end{lemma}

We define index sets as follows:

\begin{definition}\label{def:index}
\begin{align*}
\Index{m} &= \{ (i_0, \ldots, i_d) \in \{1, \ldots, m\}^{d+1} \mid \forall x 
\in \{0,1,\dots,d-1\} \colon i_x + 2 \leq i_{x+1}\} \\
\IndexC{m} & = \{ (i_0, \ldots, i_d) \in \Index{m} \mid  i_d + 2 \leq i_0 + m\}.
\end{align*} 
\end{definition}

Now Lemma~\ref{dfaceclass} can be rephrased as saying that 
$\IndexC{m}$ indexes the internal $d$-simplices of $C(m,2d)$, while 
$\Index{m}$ indexes the $d$-simplices in $C(m,2d)$ which do not lie on a 
lower boundary facet.  

Let $S$ be a triangulation of $C(m,2d)$.  
Denote by $e(S)$ the set of $d$-simplices in $C(m,2d)$ which appear as 
a face of some simplex in $S$, and which do not lie on any lower 
boundary facet of $C(m,2d)$.  

%Any $(d+1)$-tuple containing both
%$i$ and $i+1$ lies on the boundary of $C(m,2d)$, so it necessarily is a 
%face of $S$.  We shall say that a tuple of integers is {\it separated} if 
%it does not contain any consecutive pair.  The set of all increasing,
%separated $(d+1)$-tuples from $\{1,2,\dots,m\}$ will be denoted 
%$\Index{m}$.

%For book-keeping purposes, if we wish to record a $k$-dimensional simplex
%in $C(m,2d)$, instead of thinking of it as a $(k+1)$-subset of $\{1,\dots,m\}$,
%we will think of it as a (strictly) increasing $(k+1)$-tuple.  
%We will write 
%$e(S)$ for the set of $d$-dimensional faces of $S$ which are separated.  

Let $X$ and $Y$ be increasing 
$(d+1)$-tuples of real numbers.  We say that $X=(x_0,\dots,x_{d})$ intertwines 
$Y=(y_0,\dots,y_{d})$  if $x_0<y_0<x_1<y_1\dots<x_{d}<y_{d}$.  
We write $X \wr Y$ for this relation.  
A collection of increasing $(d+1)$-tuples
is called \emph{non-intertwining} if
no pair of the elements intertwine (in either order).  

\begin{theorem}\label{th1}
For any $S \in S(m,2d)$ the set $e(S)$ consists of exactly $m-d-1\choose d$ 
elements of $\Index{m}$, and  is non-intertwining. % and maximal among non-intertwining sets.  
\end{theorem} 

We also have a converse result:

\begin{theorem}\label{th2} Any non-intertwining collection of
$m-d-1\choose d$  
elements of $\Index{m}$ is $e(S)$ for a unique $S\in S(m,2d)$.
\end{theorem}

\begin{example} We consider the above theorems in the case $d=1$. 
If $S$ is a triangulation of 
$C(m,2)$, then $e(S)$ consists of the internal edges of the triangulation
together with the edge $1m$.  The theorems are clear in this case.
\end{example}

\subsection{Proof of Theorem~\texorpdfstring{\ref{th1}}{2.3}}
We recall Radon's Theorem, which can be found, for example, as \cite[Theorem I.4.1]{bar}:

\begin{theorem}
Given $e+2$ points in $\mathbb R^{e}$, they can
be partitioned into two disjoint sets $C$ and $D$ such that the convex
hulls of $C$ and $D$ intersect.\end{theorem}

An \emph{affine dependency} among vectors $\{v_1,\dots,v_{r}\}$ 
in $\mathbb R^{e}$ is a relation
of the form $\sum a_iv_i=0$ where $\sum a_i=0$, but the coefficients are
not all zero.  

We can make Radon's Theorem more specific if we begin with $2d+2$
distinct points on the moment curve in $\mathbb R^{2d}$.  
The result below is essential for us, so we provide a proof; a different proof can be found in \cite{ER}.

\begin{lemma}\label{affdep} Let $a_1<\dots<a_{2d+2}$. Among the points 
$p_{a_1},\dots,p_{a_{2d+2}}$ there is a unique affine dependency, which 
can be expressed in the form 
$$\sum_{i \textrm { even}} c_ip_{a_i} = \sum_{i \textrm { odd}} c_ip_{a_i}$$
where the $c_i$ are all positive and 
$$\sum_{ i \textrm { even}} c_i=1=\sum_{i \textrm { odd}} c_i$$
\end{lemma}

\begin{proof}
Because the moment curve is degree
$2d$, it can have at most $2d$ intersections with any (affine) hyperplane. 
Thus, the $2d+2$ points which we consider do not all lie on any 
hyperplane, so there must be exactly one affine dependency among them.  

Express the affine dependency as 
$\sum_{i\in I} c_i p_{a_i} = \sum_{i\not\in I} c_i p_{a_i}$, with 
$c_i\geq 0$ and $\sum_{i \in I} c_i = \sum_{i \not\in I} c_i$.  
Since no
$2d+1$ points lie in a hyperplane, any proper subset of  the $\{p_{a_i}\}$
is affinely independent, and thus we must have all $c_i>0$.  If the
affine dependency is not of the form given in the statement of the lemma, 
we must have that either $I$
or $I^c=\{1,\dots,2d+2\}\setminus I$ contains two consecutive integers, so without
loss of generality suppose that $\{i,i+1\}\subset I$.  

More geometrically, the affine dependency implies that the convex hull of $\{p_{a_i}\}_{i\in I}$
intersects the convex hull of $\{p_{a_i}\}_{i \in I^c}$.  Deform this
configuration by moving $a_{i+1}$ towards $a_i$. 
The point of intersection necessarily moves continuously as $a_{i+1}$ is 
deformed.  As $a_{i+1}$ moves, 
the point of intersection cannot hit the boundary of 
the convex hull of $\{p_{a_i}\}_{i\in I^c}$, because, if it did, that
would amount to an affine dependency omitting some $p_{a_j}$, which
we have already said is impossible.  

Thus, by continuity, we will still have an affine dependency when 
$a_{i+1}$ reaches $a_i$.  But that is impossible, since now we would
have an affine dependency among $2d+1$ points not all on a hyperplane.  
\end{proof}

%(NOTE: I need to check what exactly is in \cite{ER}; we should probably
%give the above argument anyway, but I should be sure how much credit to
%give \cite{ER} for it. --- It's actually totally different from what's in
%\cite{ER}.

The previous lemma 
can also be expressed as saying that if $X$ and $Y$ are intertwining
$(d+1)$-tuples, then the corresponding $d$-simplices intersect in a
single interior point
of both, while if $X$ and $Y$ are distinct $(d+1)$-tuples which do
not intertwine, the relative interiors of their corresponding
simplices are disjoint.

\begin{lemma} If $S\in S(m,2d)$, then $e(S)$ is non-intertwining.
\end{lemma}

\begin{proof}
The elements of $e(S)$ are faces of simplices
in the triangulation.  Thus, they cannot intersect in a single point 
in both their interiors.  It follows that $e(S)$ is non-intertwining.  
\end{proof}

\begin{proof}[Proof of Lemma~\ref{dfaceclass}]
This follows immediately from the description of the upper and lower
boundary facets of $C(m,2d)$ given in \cite[Lemma 2.3]{ER}: the lower
boundary facets of $C(m,2d)$ are precisely those simplices whose vertices are $2d$-subsets consisting of
a union of $d$ pairs of the form $\{i,i+1\}$, while the upper boundary
facets are precisely those simplices whose vertices are $2d$-subsets consisting of a union of 
$d-1$ pairs of the form $\{i,i+1\}$ together with $\{1,m\}$.  
\end{proof}

%\begin{lemma}\label{lowerbound} If $A$ is an increasing 
%$(d+1)$-tuple which is not
%separated, then it lies on a lower boundary facet of $C(m,2d)$, and 
%conversely.
%\end{lemma}

%\begin{proof} 
%\cite[Lemma 2.3]{ER} describes the lower boundary facets of $C(m,2d)$: they
%are indexed by the $2d$-subsets consisting of a union of $d$ pairs of the
%form $\{i,i+1\}$.  An increasing $(d+1)$-tuple is contained in such a 
%$2d$-subset iff it is not separated.   
%\end{proof}

%\begin{lemma} \tbd[upper boundary]
%\end{lemma}

We next show that if $S\in S(m,2d)$, then $|e(S)|={m-d-1\choose d}$. 
We do this in two steps, first showing that the number of simplices in
$S$ is $m-d-1\choose d$, and then showing that there is a way to 
assign each element of $e(S)$ to a simplex of $S$ in a one-to-one way.

\begin{definition} We say that  $(i_0,i_1,\dots,i_k)$ is 
\emph{separated} if $i_{x+1}\geq i_x +2$ for all $0\leq x \leq k-1$.  
\end{definition}

Using this term, we can rephrase the definition of $\Index{m}$ as 
the set of separated $(d+1)$-tuples from $\{1,2,\dots,m\}$.  

\begin{lemma}\label{lemnum} For $S\in S(m,2d)$, the triangulation 
$S$ contains $m-d-1\choose d$ simplices.  \end{lemma}

\begin{proof} Consider some separated $d$-tuple from $[2,m-1]$, say 
$A=(a_1,\dots,a_d)$.  Collapse together the vertices of $C(m,2d)$ less
than $a_1$, then those greater than $a_1$ but less than $a_2$, etc.  
(That is to say, move the given sets of vertices along the moment curve
until they coincide.)
A triangulation of $C(m,2d)$ will yield a triangulation of the smaller
polytope resulting from this process: deform the triangulation along
with the polytope, and throw away any simplices which degenerate.  In this case, 
the
result is a cyclic polytope with $2d+1$ vertices.  This polytope
is itself a simplex, so it has only one triangulation.  The unique simplex of
this triangulation
must have come from some simplex of $C(m,2d)$.  
Therefore, there must be 
exactly one 
simplex of $S$ of the form $(b_0,a_1,b_1,\dots, b_{d-1},a_d,b_d)$ 
(for the specified values of $a_i$ and some $b_i$, such that the $(2d+1)$-tuple
is increasing as listed).  Clearly, any 
simplex of $S$ satisfies this property for exactly one choice of 
$d$-subset $A$, so there must be as many simplices in $S$ as there are 
separated $d$-tuples in $[2,m-1]$, that is, $m-d-1\choose d$.  
\end{proof}

For $A=(a_0,\dots,a_{2d})$ an increasing $(2d+1)$-tuple from $[1,m]$, define 
the $(d+1)$-tuple $e(A)=(a_0,a_2,\dots a_{2d})$ by taking the even-index
terms from $A$. Similarly we set $o(A) = (a_1, a_3, \ldots, a_{2d-1})$.

Given a simplex $C$ of dimension less than $2d$, 
the points \emph{immediately below} it are those points which
are a small distance below some point
in the relative interior of $C$.

\begin{lemma} \label{below}
If $A$ is a $2d$-simplex of some triangulation in $S(m,2d)$, then 
$A$ contains the points immediately below $e(A)$.  
\end{lemma}

\begin{proof} Consider $2d+2$ points on the moment curve, with
the first $2d+1$ corresponding to the vertices of $A$, in order, and the
last being $p_t$, with $t$ varying. Consider the effect as $t \to
\infty$.  The vector $p_t$ approaches vertical.  Thus, the point in
common between $e(A)$ and $\langle o(A), p_t\rangle$ approaches (as
$t \to \infty$) a point which lies in $e(A)$ and which has a 
point in $o(A)$ vertically below it. It follows that any point between
these two will be in $A$.  
\end{proof}

\begin{lemma} \label{lemma.below_onlyif_e}
If $A$ is a $2d$-simplex of some triangulation in $S(m,2d)$, and $E$ a $d$-face of $A$,
with $e(A)\ne E$, then $A$ does not contain the points immediately
below $E$.
\end{lemma}

\begin{proof}
Let $A=(a_0,\dots,a_{2d})$, which we can think of as a realization of
$C(2d+1,2d)$.
The lower facets of $A$ 
are those obtained by deleting some 
$a_{2j}$ (again by \cite[Lemma 2.3]{ER}).  Knowing that $E\ne e(A)$, we know that $E$ lies inside
at least one lower facet.  Thus the points immediately below $E$ lie outside
$A$.  
\end{proof}

\begin{proposition}\label{propa} For $S\in S(m,2d)$, 

$$e(S)=\{e(A)\mid A\in S\}$$

\end{proposition}

\begin{proof}
Clearly, if $A$ is a simplex of $S$, then $e(A)$ is a face of $S$, and
it is also automatic that it is separated.  It follows that for $A\in S$,
we have that $e(A)\in e(S)$.

Let $E$ be a $d$-face of $S$ which is separated.  
By Lemma~\ref{dfaceclass}, $E$ does not lie in the union of the lower 
facets of $C(m,2d)$.  
This means that there are points
immediately below $E$ which lie inside $C(m,2d)$.  
 These points must lie in some simplex $A$ of $S$. By Lemma~\ref{lemma.below_onlyif_e} this can only happen if the face is $e(A)$.
\end{proof}

\begin{lemma}\label{cor1}
For $S\in S(m,2d)$ and $A,B$ distinct simplices in $S$, $e(A)\ne e(B)$.
\end{lemma}

\begin{proof} If $e(A)=e(B)$, the points immediately below $e(A)=e(B)$ 
must lie in both $A$ and $B$, so $A=B$.  
\end{proof}

Proposition \ref{propa} and Lemma \ref{cor1} together imply that the
number of elements of $e(S)$ equals the number of simplices of 
$S$, which, by Lemma \ref{lemnum}, is $m-d-1\choose d$.  This 
completes the proof of Theorem \ref{th1}.  

\subsection{Proof of Theorem~\texorpdfstring{\ref{th2}}{2.4}}

\begin{lemma} \label{deylem}
For $S\in S(m,2d)$, 
the faces of $S$ of dimension at least $d$ consist of exactly those 
simplices whose $d$-faces are either not separated or are contained in
$e(S)$. \end{lemma}
\begin{proof}
Dey \cite{dey} shows that, for any triangulation $T$ of a point configuration in $\mathbb R^{\delta}$, it is possible to reconstruct $T$ on the basis
of knowing only its $\lfloor \frac {\delta}2 \rfloor$-faces.  We follow Dey's approach,
but specialize to our setting, where it is possible to give a somewhat simpler
description of the reconstructed triangulation.  

Let $S\in S(m,2d)$, and let $A$ be a $k$-dimensional simplex of $S$ with
$k\geq d$.  Clearly, the $d$-dimensional faces of $A$ are also $d$-simplices
of $S$.  The $d$-simplices of $S$ correspond to the $(d+1)$-tuples in
$e(S)$ and those increasing $(d+1)$-tuples from $[1,m]$ 
which are not separated, so one direction of the lemma
is shown.  

For the other direction, suppose that we have a $k$-simplex $A$ in 
$C(m,2d)$, which
does not belong to $S$.  The relative interior of $A$ must intersect
the relative interior of some $j$-simplex $B$ of $S$ with $j\geq k$.    
Dey \cite[Lemma 3.1]{dey} points out that in $\mathbb R^{2d}$, if a $k$-simplex $A$ and a 
$j$-simplex $B$ intersect 
in their relative interiors, with $k+j\geq 2d$, then there must be a
$k'$-face $A'$ of $A$ and a $j'$-face $B'$ of $B$ which intersect in their 
relative
interiors, with $k' + j' \leq 2d$.  That is to say, among the at most $2d+2$ vertices
of $A'$ and $B'$, there must be an affine dependency.  By Lemma~\ref{affdep},
the form of this affine dependency implies that 
$A'\wr B'$ or the reverse.  In particular $A'$ and $B'$ must both 
be $d$-faces.  $B'$ belongs to $S$  
since $B$ does.  $A'$ and $B'$ intersect in
their relative interiors, so $A'$ cannot be a face of $S$.  We have shown
that there is a $d$-face of $A$ which is not a $d$-face of $S$, as desired.
\end{proof}

We now define two operations on triangulations, following \cite[Section 3]{RSa}.

\begin{definition}   Let $S\in S(m,2d)$. \begin{enumerate}  
\item We define $S/1$ to be the triangulation of 
$C([2,m],2d)$ which is obtained from $S$ by moving the vertices
1 and 2 together and throwing away the simplices that degenerate.  

\item We define $S\setminus 1$ to be the triangulation of $C([2,m],2d-1)$ obtained
by taking only the simplices of $S$ that contain 1, and then removing
1 from them.  This clearly defines a triangulation of the vertex figure of 
$C(m,2d)$ at $1$, that is to say, of the $(2d-1)$-dimensional polytope obtained by intersecting
$C(m,2d)$ with a hyperplane which cuts off the vertex 1.  This vertex figure
is not a cyclic polytope according to our definition, but its vertices 
determine
the same oriented matroid as the vertices of a cyclic polytope, 
which is sufficient to 
imply that a triangulation of the vertex figure also determines a triangulation
of $C([2,m],2d-1)$, and conversely (see \cite[Lemma 3.1]{RSa} for details). We write $S\setminus\{1,2\}$ for
$(S\setminus 1) \setminus 2$.
\end{enumerate} \end{definition}

We next define two operations on subsets of $\Index{m}$.
  We will
eventually relate these to the operations we have already defined
on $S(m,2d)$, but for now, they are separate.

For an $e$-tuple $A = (a_1, \ldots, a_e)$ with $a_1 > 1$ we denote by $1 \star A$ the $(e+1)$-tuple $(1, a_1, \ldots, a_e)$. For a set $X$ of $e$-tuples with this property we denote by $1 \star X$ the set $\{1 \star A \mid A \in X\}$. Similarly we define $2 \star A$ and $2 \star X$.

\begin{definition} Let $X\subset \Index{m}$.\begin{enumerate}\item
$X/1$ is obtained from $X$ by 
replacing all 1's by 2's, and removing any resulting tuples which are not
separated.  
%Note that $e(S)/1=e(S/1)$.  

\item
$X\setminus \{1,2\}$ consists of
all $d$-tuples $A$ from $[3,m]$ 
such that 
$1\star A$ is in $X$ and either $2\star A$ is in $X$ or 
$3\in A$.  (These two possibilities are mutually exclusive, since if 
$3\in A$, then $2\star A$ is not separated, and so it cannot be 
in $X$.)  
\end{enumerate}\end{definition}

Note that for $X\subset \Index{m}$, we do not define $X\setminus 1$; instead,
we define $X\setminus\{1,2\}$ in one step.

\begin{lemma} If $X$ is a non-intertwining subset of $\Index{m}$,
so are $X/1$ and $X\setminus \{1,2\}$.  
\end{lemma}

\begin{proof} $X/1$ is separated by definition.  Suppose that $A\wr B$ in
$X/1$.  Write $\widehat A, \widehat B$ for elements of $X$ which witness
the presence of $A,B$ in $X/1$.   
The minimal entry of $B$ is at least 3, so $\widehat B=B$, and thus $\widehat A
\wr \widehat B$, a contradiction.

It is immediate that $X\setminus\{1,2\}$ is separated.  Suppose that
$A\wr B$ in $X\setminus \{1,2\}$.  Then $1\star A\in X$.  Since 
$A\wr B$, the minimal element of $B$ is at least 4, so $2\star B
\in X$.  But $(1\star A) \wr (2\star B)$, a contradiction.  
\end{proof}

\begin{lemma} $|X/1|+|X\setminus\{1,2\}|=|X|$.\end{lemma}

\begin{proof} The difference $|X|-|X/1|$ is accounted for by 
elements $1\star A$ in $X$ such that either $2\star A$ is in $X$ or
 $3\in A$.  These exactly correspond to the 
elements of $X\setminus\{1,2\}$.  
\end{proof}

\begin{lemma}\label{size} The maximal size of a non-intertwining
subset of $\Index{m}$
is $m-d-1\choose d$.  
Also, if $X$ is a set of that size, $|X/1|={m-d-2\choose d}$, and 
$|X\setminus\{1,2\}|={m-d-2\choose d-1}$.  
\end{lemma}

\begin{proof}
We know that there do exist non-intertwining subsets of 
$\Index{m}$ of cardinality $m-d-1\choose d$,
because, by Theorem \ref{th1}, 
$e(S)$ is of this form for any $S\in C(m,2d)$.  

The proof that this is the maximum possible size 
is by induction on $m$ and $d$. Let $X$ be such a set.  
$X/1$ is a set of
non-intertwining separated $(d+1)$-tuples in $[2,m]$, and thus by induction
its size is at most $m-d-2\choose d$.  $X\setminus\{1,2\}$ is a collection
of non-intertwining separated $d$-tuples in $[3,m]$; its size is therefore
at most $m-d-2\choose d-1$.  
Thus 
\begin{align*}
|X|&=|X/1|+|X\setminus \{1,2\}| \\&\leq {m-d-2\choose d}+{m-d-2\choose d-1}
={m-d-1\choose d}.
\end{align*}

Also, if $X$ achieves equality, the corresponding equalities
for $X/1$ and $X\setminus\{1,2\}$ must also hold, 
which establishes the second point.  
\end{proof}

\begin{lemma}\label{equality} If $X$ and $Y$ are non-intertwining
subsets of $\Index{m}$ of cardinality $m-d-1\choose d$
such that $X/1=Y/1$ and $X\setminus \{1,2\}=Y\setminus\{1,2\}$, then
$X=Y$.  \end{lemma}

\begin{proof} The tuples of $X$ and $Y$ in which neither 1 nor 2 appears
must be the same, 
using only the fact that $X/1=Y/1$.  

We next consider the elements of $X$ and $Y$ which contain 1.  
Let $A$ be the lexicographically final element of $X$, containing 1,
which is not contained in $Y$.  Let $A'$ be obtained from $A$ by 
replacing $1$ by $2$.  

If $A'\in X$, then $A\setminus 1 \in X\setminus\{1,2\}$, which then implies
that $A\in Y$. So $A' \not \in X$.  

Also, if $3\in A$, then $A\setminus 1 \in X\setminus\{1,2\}$, which we already
saw is false.  So $3\not\in A$.  

Since $A/1 \in X/1=Y/1$, there must be some element of $Y$ which implies that
$A/1\in Y/1$.  Since we know $A\not\in Y$, it must be that $A'\in Y$.  

So $X$ contains $A$ and not $A'$, while $Y$ contains $A'$ and not $A$.  
Since $A\not\in Y$, there
must be some element $B'$ of $Y$ satisfying $A\wr B'$.  Since we do not
have $A'\wr B'$, the minimum element of $B'$ must be 2.  Since $A\in X$,
we know that $B'$ is not in $X$.  Since $B'/1\in Y/1=X/1$, 
we must have that $B\in X$, where $B$
is obtained from $B'$ by replacing 2 by 1.  Since $B'\not\in X$, we have 
$B\setminus 1 = B' \setminus 2 \not\in X\setminus\{1,2\}=Y\setminus\{1,2\}$, which implies
that $B\not\in Y$.  Thus $B$ is in $X$ but not $Y$.  

Since $A\wr
B'$, we have that $B$ lexicographically follows $A$, contradicting our
choice of $A$.  
Thus $X$ and $Y$ have the same elements which contain 1.  

Now consider the elements of $X$ and $Y$ which contain 2.  Let $A'$ be 
an element of $X$ containing 2, and let $A$ be the same element with
2 replaced by 1.  If $A\in X$ as well, then $A\setminus 1\in X\setminus\{1,2\}=
Y\setminus\{1,2\}$, and thus $A'\in Y$.  If $A\not\in X$, then $A\not\in Y$
(since $A$ contains 1, it falls in the case already considered),
but $A'\in X/1=Y/1$, which forces $A'\in Y$.  
\end{proof}

For a triangulation $Q$ of $C(p,\delta)$ and a 
triangulation $P$ of $C(p,\delta-1)$, we write that 
$P\prec Q$ if each simplex of $P$ is a facet of 
at least one simplex of $Q$.  In this case, the simplices of $Q$
are divided into two classes, those above $P$ and those below $P$.

We have the following proposition, which we cite in a convenient form,
restricted to the case which is of interest to us.  (As it appears in
\cite{RSa}, it treats subdivisions of cyclic polytopes which are more 
general than triangulations.)

\begin{proposition}[{\cite[Lemma 4.7(1)]{RSa}}]\label{err}
Let $T$ be a triangulation of $C([2,m],2d)$, and let $W$ be a 
triangulation of $C([3,m],2d-2)$.  Then there exists a triangulation
$S$ of $C(m,2d)$ with $S/1=T$ and $S\setminus\{1,2\}=W$ iff 
$W\prec T\setminus 2$.

In this case, the triangulation $S$ is unique, and can be described as follows.
Let $T^\circ$ denote those simplices of $T$ which
do not contain $2$. Let $(T\setminus 2)^+$ denote the simplices of 
$T\setminus 2$ which
lie above $W$, and $(T\setminus 2)^-$ denote the simplices of $T\setminus 2$ 
which lie
below $W$. 

Then $$S=T^\circ \cup (1 \star 2 \star W) \cup
1 \star (T\setminus 2)^+ \cup 2 \star (T\setminus 2)^-.$$
\end{proposition}

\begin{lemma}
Let $S \in S(m, 2d)$. Then
$e(S/1)=e(S)/1$.
\end{lemma}

\begin{proof}
If $E\in e(S/1)$, then it is $e(C)$ for some simplex
$C$ in $S/1$, and that simplex comes from some simplex $\widehat C$ in
$S$.  Now $e(\widehat C)/1=E$.  

On the other hand, suppose $E\in e(S)/1$.  %Let $T=S/1$, $W=S\setminus\{1,2\}$,
%and apply Proposition~\ref{err} to get an explicit description of $S$.  
Let $\widehat C$ be a simplex
from $S$ such that $e(\widehat C)/1=E$.  The image $C$ of $\widehat C$ in $S / 1$ is a face of $S / 1$ (which is $2d$- or $(2d-1)$-dimensional). Now $E$ is a face of $C$, and hence of $S / 1$.
\end{proof}

\begin{lemma} 
Let $S \in S(m, 2d)$. Then
$e(S\setminus\{1,2\})=e(S)\setminus \{1,2\}$.
\end{lemma}

\begin{proof} 
Let $E\in e(S\setminus\{1,2\})$.  So $S$ contains a simplex 
$1\star 2 \star A$ with $e(A)=E$.  Now $1\star E$ is a face of $S$, so
$1\star E\in e(S)$.  We also have that $2\star E$ is a face of $S$, so 
either $2 \star E\in e(S)$, or $2\star E$ lies on a lower boundary facet.  
In the former case, we have that $E\in e(S)\setminus\{1,2\}$, and we are
done.  In the latter case, by Lemma \ref{dfaceclass}, $2\star E$ is
not separated, which must be because $3\in E$, so, again, we conclude
that $E\in e(S)\setminus\{1,2\}$, as desired.

On the other hand, let $E\in e(S)\setminus\{1,2\}$. 
Let $T=S/1$, $W=S\setminus\{1,2\}$,
and apply Proposition~\ref{err}.

We know that $1 \star E\in e(S)$.  Say it is $e(\widehat C)$ for some 
simplex $\widehat C$ of $S$.  If $\widehat C$ is an element of $1\star 2\star W$, then we are done,
because $e(\widehat C\setminus \{1,2\})=E$.  Suppose otherwise, so 
$\widehat C$ is of the form $1\star X$ with $X\in (T\setminus 2)^+$.  
Since the points immediately below $1\star E$ lie in $\widehat C$, we know that 
$E$ lies above 
$W$ inside $T\setminus 2$.  But this means that $2\star E$ is not a face 
of $S$, so it 
cannot be 
$e(\widehat D)$ for $\widehat D$ a simplex of $S$ and it  cannot be
non-separated, by Lemma \ref{dfaceclass}.  Thus $E$ does not
lie in $e(S)\setminus\{1,2\}$, contrary to our assumption.    
\end{proof}

\begin{proof}[Proof of Theorem \ref{th2}]
Suppose that we have a non-intertwining set $X\subset \Index{m}$ of
cardinality $m-d-1\choose d$. 
We want to show that it defines a unique
triangulation.  The proof is by
induction on $d$ and $m$.

By Lemma \ref{size}, $|X/1| ={m-d-2\choose d}$ and $|X\setminus\{1,2\}|= 
{m-d-2\choose d-1}$.  
It follows by induction that $X/1$ and $X\setminus\{1,2\}$ define unique
triangulations, of $C([2,m],2d)$ and $C([3,m],2d-2)$, respectively, which
we can denote $T$ and $W$.  

\begin{lemma} $W\prec T\setminus 2$. \end{lemma}

\begin{proof}
Let $A$ be a $(2d-2)$-simplex of $W$.  We wish to show that $A$ is a face
of $T\setminus 2$, or in other words, that $2 \star A$ is a face of $T$.
By Lemma \ref{deylem}, it 
suffices to show that any $d$-face of $2 \star W$ is a face of $T$.  

Such faces are of two kinds: first, faces of the form $2 \star E$, with $E$ a $(d-1)$-face of $W$,
and, second, $d$-faces of $W$. 

If $E\in e(W)$, then by definition $E\in 
X\setminus\{1,2\}$, so $1\star E$ is in $X$, and $2\star E$ is either
on the boundary of $C([2,m],d)$ or it is in $X/1$; either way, we are done.

If $E$ lies on a lower boundary facet of $W$, then it is not separated, so
neither is $2\star E$, which therefore lies on a lower boundary facet of $T$.

Now consider $F=(a_0,\dots,a_d)$ a $d$-face of $W$.  We know that the other $d$-faces of
$2 \star F$ are in $T$.  If $F$ is not, then, since $e(T)$ is maximal, there
is some $B=(b_0,\dots,b_d)$ in $e(T)$ such that $B\wr F$ or $F\wr B$.  

Suppose $F\wr B$.    Then $B$ also intertwines $(2,a_1,\dots,a_d)$, but
that contradicts the assertion that $(2,a_1,\dots,a_d)$ is in $e(T)$.  

So suppose $B\wr F$.  If $b_0>2$, then $(2,a_0,\dots,a_{d-1})\wr B$,
which is again a contradiction.  So $b_0=2$.  

Since $B$ is in $X/1=e(T)$, it lifts to (at least) one element of $X$, say 
$\widetilde B=(\widetilde b_0,b_1,\dots,b_d)$.  Suppose first that
$\widetilde b_0=2$.  Since $(a_0,\dots,a_{d-1})$ is a $(d-1)$-face of $W$,
it lifts to $\widetilde A=(1,a_0,\dots,a_{d-1})$ in $X$.  But 
$\widetilde A\wr \widetilde B$, a contradiction.  

Now suppose that $\widetilde b_0=1$.  Since $a_1\geq 4$, $(a_1,\dots,a_d)$ 
lifts to 
$\widetilde A'=(2,a_1,\dots,a_d)$.
Thus $\widetilde B\wr \widetilde A'$, a contradiction.  
\end{proof}

Proposition \ref{err} implies 
that there is a unique triangulation $S$ such that 
$S/1=T$ and $S\setminus\{1,2\}=W$.  We know that $e(S)/1=X/1$ and 
$e(S)\setminus\{1,2\}=X\setminus\{1,2\}$.  By Lemma ~\ref{equality},
$X=e(S)$. By Lemma \ref{deylem}, $S$ is the only 
triangulation with $e(S)=X$.    
\end{proof}

\section{Higher Auslander algebras of linearly oriented \texorpdfstring{$A_n$}{A\_n} and their cluster tilting modules}
\label{section.higher_Aus_An}

We begin this section by recalling some background on higher Auslander algebras (see \cite{Iy_n-Auslander}). We always assume $\Lambda$ to be a finite dimensional algebra over some field $k$.

\begin{definition} \phantomsection \label{def.d-repfin}
\begin{enumerate}
\item A \emph{$d$-cluster tilting module} (\cite{Iy_n-Auslander}) is a module $M \in \mod \Lambda$, such that
\begin{align*}
\add M & = \{ X \in \mod M \mid \Ext_{\Lambda}^i(X,M) = 0 \, \forall i \in \{1, \ldots, d-1\} \} \\
& = \{ X \in \mod M \mid \Ext_{\Lambda}^i(M,X) = 0 \, \forall i \in \{1, \ldots, d-1\} \}.
\end{align*}
\item If $\Lambda$ has a $d$-cluster tilting module, and moreover $\gld \Lambda \leq d$, then $\Lambda$ is called \emph{$d$-representation finite} (\cite{IO}).
\end{enumerate}
\end{definition}

Iyama \cite{Iy_n-Auslander} has shown that for a $d$-representation finite algebra $\Lambda$, the $d$-cluster tilting module $M$ as in the definition above is unique up to multiplicity (see Theorem~\ref{theo.ncto_from_proj} below).

\begin{definition}
Let $\Lambda$ be $d$-representation finite, and $M$ be the basic $d$-cluster tilting module. Then the \emph{$d$-Auslander algebra} of $\Lambda$ is $\End_{\Lambda}(M)$.
\end{definition}

We now focus on the case of higher Auslander algebras of linear oriented $A_n$. We denote by $A_n^1 = k[1 \to 2 \to \cdots \to n]$ the quiver algebra of a linearly oriented $A_n$ quiver. We denote by $\leftsub{A_n^1}{S}_i$ the simple module concentrated in vertex $i$, by $\leftsub{A_n^1}{P}_i$ its projective cover, and by $\leftsub{A_n^1}{I}_i$ its injective envelope. Then the numbering of the vertices is chosen in such a way that
\[ \Hom_{A_n^1}(\leftsub{A_n^1}{P}_i, \leftsub{A_n^1}{P}_j) \neq 0 \, \iff \, i \leq j. \]

\begin{theoconst}[\cite{Iy_n-Auslander}] \label{const.AnMn}
\begin{itemize}
\item $A_n^1$ is $1$-representation finite. We denote its basic $1$-cluster tilting module by $\leftsub{A_n^1}{M}$, and the $1$-Auslander algebra of $A_n^1$ by $A_n^2 = \End_{A_n^1}(\leftsub{A_n^1}{M})$.
\item $A_n^2$ is $2$-representation finite. We denote its basic $2$-cluster tilting module by $\leftsub{A_n^2}{M}$, and the $2$-Auslander algebra of $A_n^2$ by $A_n^3 = \End_{A_n^2}(\leftsub{A_n^2}{M})$. \\
\textcolor{white}{space} $\vdots$
\item This iterates. That is, $A_n^d$ is $d$-representation finite. We denote its basic $d$-cluster tilting module by $\leftsub{A_n^d}{M}$, and the $d$-Auslander algebra of $A_n^d$ by $A_n^{d+1} = \End_{A_n^d}(\leftsub{A_n^d}{M})$.
\end{itemize}
\end{theoconst}

We remark that the case $d=1$ of the above theorem is classical.  An algebra is $1$-representation finite if and only if it is representation finite and hereditary.  A $1$-cluster tilting module is by definition an additive generator of the module category.  The $1$-Auslander algebra of $A_n^1$ is then the classical Auslander algebra of $A_n^1$.

Iyama \cite{Iy_n-Auslander} gives descriptions of these algebras $A_n^d$ by giving quivers and relations (see also Table~\ref{table.examples.quivers} for an idea of how these quivers look). Here we follow a slightly different approach in indexing the indecomposable summands of $\leftsub{A_n^d}{M}$, which will allow us to immediately read off when there are non-zero homomorphisms or non-zero extensions between two such summands (see Theorem~\ref{theo.indexing.props}(5 and 6)). Since our aim here is to study tilting modules (in Theorem~\ref{theo.tilt-exchange} and in Section~\ref{local}), it is particularly important to us to be able to decide whether extension groups vanish.

\begin{table}
\begin{tikzpicture}
% first quiver
 \node at (0,3) {$A_4^1$:};
 \node (1) at (0,0) {$1$};
 \node (2) at (1,1) {$2$};
 \node (3) at (2,2) {$3$};
 \node (4) at (3,3) {$4$};
 \draw [->] (1) -- (2);
 \draw [->] (2) -- (3);
 \draw [->] (3) -- (4);
% second quiver
 \pgftransformshift{\pgfpoint{-1cm}{-5cm}}
 \node at (0,3) {$A_4^2$:};
 \node (11) at (0,0) {$13$};
 \node (12) at (1,1) {$14$};
 \node (13) at (2,2) {$15$};
 \node (14) at (3,3) {$16$};
 \node (22) at (2,0) {$24$};
 \node (23) at (3,1) {$25$};
 \node (24) at (4,2) {$26$};
 \node (33) at (4,0) {$35$};
 \node (34) at (5,1) {$36$};
 \node (44) at (6,0) {$46$};
 \draw [->] (11) -- (12);
 \draw [->] (12) -- (13);
 \draw [->] (13) -- (14);
 \draw [->] (12) -- (22);
 \draw [->] (13) -- (23);
 \draw [->] (14) -- (24);
 \draw [->] (22) -- (23);
 \draw [->] (23) -- (24);
 \draw [->] (23) -- (33);
 \draw [->] (24) -- (34);
 \draw [->] (33) -- (34);
 \draw [->] (34) -- (44);
% third quiver
 \pgftransformshift{\pgfpoint{8cm}{5cm}}
 \node at (0,3) {$A_4^3$:};
 \node (111) at (0,0) {$135$};
 \node (112) at (1,1) {$136$};
 \node (113) at (2,2) {$137$};
 \node (114) at (3,3) {$138$};
 \node (122) at (2,0) {$146$};
 \node (123) at (3,1) {$147$};
 \node (124) at (4,2) {$148$};
 \node (133) at (4,0) {$157$};
 \node (134) at (5,1) {$158$};
 \node (144) at (6,0) {$168$};
 \node (222) at (.95,-3.5) {$246$};
 \node (223) at (1.95,-2.5) {$247$};
 \node (224) at (2.95,-1.5) {$248$};
 \node (233) at (2.95,-3.5) {$257$};
 \node (234) at (3.95,-2.5) {$258$};
 \node (244) at (4.95,-3.5) {$268$};
 \node (333) at (2.2,-6) {$357$};
 \node (334) at (3.2,-5) {$358$};
 \node (344) at (4.2,-6) {$368$};
 \node (444) at (3.75,-7.5) {$468$};
 \draw [->] (111) -- (112);
 \draw [->] (112) -- (113);
 \draw [->] (113) -- (114);
 \draw [->] (112) -- (122);
 \draw [->] (113) -- (123);
 \draw [->] (114) -- (124);
 \draw [->] (122) -- (123);
 \draw [->] (123) -- (124);
 \draw [->] (123) -- (133);
 \draw [->] (124) -- (134);
 \draw [->] (133) -- (134);
 \draw [->] (134) -- (144);
 \draw [->] (122) -- (222);
 \draw [->] (123) -- (223);
 \draw [->] (124) -- (224);
 \draw [->] (133) -- (233);
 \draw [->] (134) -- (234);
 \draw [->] (144) -- (244);
 \draw [->] (222) -- (223);
 \draw [->] (223) -- (224);
 \draw [->] (223) -- (233);
 \draw [->] (224) -- (234);
 \draw [->] (233) -- (234);
 \draw [->] (234) -- (244);
 \draw [->] (233) -- (333);
 \draw [->] (234) -- (334);
 \draw [->] (244) -- (344);
 \draw [->] (333) -- (334);
 \draw [->] (334) -- (344);
 \draw [->] (344) -- (444);
% fourth quiver
 \pgftransformshift{\pgfpoint{-8cm}{-10cm}}
 \node at (0,3) {$A_4^4$:};
 \node (1111) at (0,0) {$1357$};
 \node (1112) at (1,1) {$1358$};
 \node (1113) at (2,2) {$1359$};
 \node (1114) at (3,3) {$135$X};
 \node (1122) at (2,0) {$1368$};
 \node (1123) at (3,1) {$1369$};
 \node (1124) at (4,2) {$136$X};
 \node (1133) at (4,0) {$1379$};
 \node (1134) at (5,1) {$137$X};
 \node (1144) at (6,0) {$138$X};
 \node (1222) at (.95,-3.5) {$1468$};
 \node (1223) at (1.95,-2.5) {$1469$};
 \node (1224) at (2.95,-1.5) {$146$X};
 \node (1233) at (2.95,-3.5) {$1479$};
 \node (1234) at (3.95,-2.5) {$147$X};
 \node (1244) at (4.95,-3.5) {$148$X};
 \node (1333) at (2.2,-6) {$1579$};
 \node (1334) at (3.2,-5) {$157$X};
 \node (1344) at (4.2,-6) {$158$X};
 \node (1444) at (3.75,-7.5) {$168$X};
 \node (2222) at (6.95,-2) {$2468$};
 \node (2223) at (7.95,-1) {$2469$};
 \node (2224) at (8.95,0) {$246$X};
 \node (2233) at (8.95,-2) {$2479$};
 \node (2234) at (9.95,-1) {$247$X};
 \node (2244) at (10.95,-2) {$248$X};
 \node (2333) at (8.2,-4.5) {$2579$};
 \node (2334) at (9.2,-3.5) {$257$X};
 \node (2344) at (10.2,-4.5) {$258$X};
 \node (2444) at (9.75,-6) {$268$X};
 \node (3333) at (12.2,-3.5) {$3579$};
 \node (3334) at (13.2,-2.5) {$357$X};
 \node (3344) at (14.2,-3.5) {$358$X};
 \node (3444) at (13.75,-5) {$368$X};
 \node (4444) at (15.75,-4.5) {$468$X};
 \draw [->] (1111) -- (1112);
 \draw [->] (1112) -- (1113);
 \draw [->] (1113) -- (1114);
 \draw [->] (1112) -- (1122);
 \draw [->] (1113) -- (1123);
 \draw [->] (1114) -- (1124);
 \draw [->] (1122) -- (1123);
 \draw [->] (1123) -- (1124);
 \draw [->] (1123) -- (1133);
 \draw [->] (1124) -- (1134);
 \draw [->] (1133) -- (1134);
 \draw [->] (1134) -- (1144);
 \draw [->] (1122) -- (1222);
 \draw [->] (1123) -- (1223);
 \draw [->] (1124) -- (1224);
 \draw [->] (1133) -- (1233);
 \draw [->] (1134) -- (1234);
 \draw [->] (1144) -- (1244);
 \draw [->] (1222) -- (1223);
 \draw [->] (1223) -- (1224);
 \draw [->] (1223) -- (1233);
 \draw [->] (1224) -- (1234);
 \draw [->] (1233) -- (1234);
 \draw [->] (1234) -- (1244);
 \draw [->] (1233) -- (1333);
 \draw [->] (1234) -- (1334);
 \draw [->] (1244) -- (1344);
 \draw [->] (1333) -- (1334);
 \draw [->] (1334) -- (1344);
 \draw [->] (1344) -- (1444);
 \draw [->] (1222) -- (2222);
 \draw [->] (1223) -- (2223);
 \draw [->] (1224) -- (2224);
 \draw [->] (1233) -- (2233);
 \draw [->] (1234) -- (2234);
 \draw [->] (1244) -- (2244);
 \draw [->] (1333) -- (2333);
 \draw [->] (1334) -- (2334);
 \draw [->] (1344) -- (2344);
 \draw [->] (1444) -- (2444);
 \draw [->] (2222) -- (2223);
 \draw [->] (2223) -- (2224);
 \draw [->] (2223) -- (2233);
 \draw [->] (2224) -- (2234);
 \draw [->] (2233) -- (2234);
 \draw [->] (2234) -- (2244);
 \draw [->] (2233) -- (2333);
 \draw [->] (2234) -- (2334);
 \draw [->] (2244) -- (2344);
 \draw [->] (2333) -- (2334);
 \draw [->] (2334) -- (2344);
 \draw [->] (2344) -- (2444);
 \draw [->] (2333) -- (3333);
 \draw [->] (2334) -- (3334);
 \draw [->] (2344) -- (3344);
 \draw [->] (2444) -- (3444);
 \draw [->] (3333) -- (3334);
 \draw [->] (3334) -- (3344);
 \draw [->] (3344) -- (3444);
 \draw [->] (3444) -- (4444);
\end{tikzpicture}
\caption{Quivers of the algebras $A_4^1$, $A_4^2$, $A_4^3$, and $A_4^4$. The numbers in the vertices come from the labeling introduced in Theorem~\ref{theo.indexing}. (In the quiver of $A_4^4$ the number ``$10$'' is written ``X'', to avoid having to use commas between the indices.)} \label{table.examples.quivers}
\end{table}

In the theorem below, we use the indexing set $\Index{n+2d}$, as in
Definition \ref{def:index}.  It consists of the separated $(d+1)$-tuples
from $\{1,\dots,n+2d\}$.

Assigning labels to summands of $\leftsub{A_n^d}{M}$ we follow an inductive construction. Recall that the projective, injective, and simple $A_n^1$-modules are already labeled by elements of $\Index[0]{n}$.

\begin{theoconst} \phantomsection \label{theo.indexing}
\begin{itemize}
\item For $(i_0, i_1) \in \Index[1]{n+2}$, the module $\leftsub{A_n^1}{M}$ has a unique indecomposable summand which has composition factors
\[ \leftsub{A_n^1}{S}_j, \quad i_0-1 < j < i_1-1. \]
We denote this summand of $\leftsub{A_n^1}{M}$ by $\leftsub{A_n^1}{M}_{i_0, i_1}$. Moreover, all summands of $\leftsub{A_n^1}{M}$ are of this shape, that is
\[ \leftsub{A_n^1}{M} = \bigoplus_{(i_0, i_1) \in \Index[1]{n+2}} \leftsub{A_n^1}{M}_{i_0, i_1}. \]
We denote by $\leftsub{A_n^2}{P}_{i_0, i_1} = \Hom_{A_n^1}(\leftsub{A_n^1}{M}, \leftsub{A_n^1}{M}_{i_0,i_1})$ the corresponding indecomposable projective $A_n^2$-module, and by $\leftsub{A_n^2}{S}_{i_0, i_1}$ and $\leftsub{A_n^2}{I}_{i_0, i_1}$ the corresponding simple and indecomposable injective $A_n^2$-modules.
\item For $(i_0, i_1, i_2) \in \Index[2]{n+4}$, the module $\leftsub{A_n^2}{M}$ has a unique indecomposable summand which has composition factors
\[ \leftsub{A_n^2}{S}_{j_0, j_1}, \quad i_0-1 < j_0 < i_1-1 < j_1 < i_2-1. \]
We denote this summand of $\leftsub{A_n^2}{M}$ by $\leftsub{A_n^2}{M}_{i_0, i_1, i_2}$. All summands of $\leftsub{A_n^2}{M}$ are of this form, that is
\[ \leftsub{A_n^2}{M} = \bigoplus_{(i_0, i_1, i_2) \in \Index[2]{n+4}} \leftsub{A_n^2}{M}_{i_0, i_1, i_2}. \]
We denote by $\leftsub{A_n^3}{P}_{i_0, i_1, i_2} = \Hom_{A_n^2}(\leftsub{A_n^2}{M}, \leftsub{A_n^2}{M}_{i_0,i_1,i_2})$, $\leftsub{A_n^3}{S}_{i_0, i_1, i_2}$, and $\leftsub{A_n^3}{I}_{i_0, i_1, i_2}$ the corresponding indecomposable projective, simple, and indecomposable injective $A_n^3$-modules, respectively. \\ \textcolor{white}{spave} $\vdots$
\item This iterates. That is, for $(i_0, \ldots, i_d) \in \Index{n+2d}$, the module $\leftsub{A_n^d}{M}$ has a unique indecomposable summand which has composition factors
\[ \leftsub{A_n^d}{S}_{j_0, \ldots, j_{d-1}}, \quad i_0-1 < j_0 < i_1-1 < j_1 < \cdots < i_{d-1}-1 < j_{d-1} < i_d-1. \]
We denote this summand of $\leftsub{A_n^d}{M}$ by $\leftsub{A_n^d}{M}_{i_0, \ldots i_d}$. We have
\[ \leftsub{A_n^d}{M} = \bigoplus_{(i_0, \ldots, i_d) \in \Index{n+2d}} \leftsub{A_n^d}{M}_{i_0, \ldots, i_d}. \]
We denote by $\leftsub{A_n^{d+1}}{P}_{i_0, \ldots, i_d} = \Hom_{A_n^d}(\leftsub{A_n^d}{M}, \leftsub{A_n^d}{M}_{i_0, \ldots, i_d})$, $\leftsub{A_n^{d+1}}{S}_{i_0, \ldots, i_d}$, and $\leftsub{A_n^{d+1}}{I}_{i_0, \ldots, i_d}$ the corresponding indecomposable projective, simple, and indecomposable injective $A_n^{d+1}$-modules, respectively.
\end{itemize}
\end{theoconst}

\begin{remark}
Thinking of quivers (see Table~\ref{table.examples.quivers}) the statement on the composition factors of the $M_{i_0, \ldots i_d}$ just means that $M_{i_0, \ldots, i_d}$ has a ``box-shaped'' support, which lies properly between the vertices $(i_0-1, \ldots, i_{d-1}-1)$ and $(i_1-1, \ldots, i_d-1)$. (Here the ``$-1$''s are just a result of normalizing the indexing sets, so that they start at $1$
for all $d$.)
\end{remark}

With the indexing of Theorem~\ref{theo.indexing}, we have the following result:

\begin{theorem} \label{theo.indexing.props}
For given $n$ and $d$ we have
\begin{enumerate}
\item $\leftsub{A_n^d}P_{i_0, \ldots, i_{d-1}} = \leftsub{A_n^d}{M}_{1, i_0+2, \ldots, i_{d-1}+2}$.
\item $\leftsub{A_n^d}I_{i_0, \ldots, i_{d-1}} = \leftsub{A_n^d}{M}_{i_0, \ldots, i_{d-1}, n+2d}$.
\item $\Hom_{A_n^d}(\leftsub{A_n^d}{M}_{i_0, \ldots, i_d}, \leftsub{A_n^d}{M}_{j_0, \ldots, j_d}) \neq 0 \iff i_0-1 < j_0 < i_1-1 < j_1 < \cdots < i_d-1 < j_d$, and in this case the $\Hom$-space is one-dimensional.
\item $\Ext_{A_n^d}^d(\leftsub{A_n^d}{M}_{i_0, \ldots, i_d}, \leftsub{A_n^d}{M}_{j_0, \ldots, j_d}) \neq 0 \iff (j_0, \ldots, j_d) \wr (i_0, \ldots, i_d)$ (see 
Section \ref{cyclic}), and in this case the $\Ext$-space is one-dimensional.
\end{enumerate}
\end{theorem}

We note that, by the previous theorem, the projective-injective 
indecomposables of $A^d_n$ are indexed by the elements of 
$\Index{n+2d}\setminus \IndexC{n+2d}$. By Lemma~\ref{dfaceclass}, these
correspond to the $d$-simplices of $C(n+2d,2d)$ which lie in an upper
boundary facet and in no lower boundary facet.  Non-projective-injective
summands correspond to internal $d$-simplices of $C(n+2d,2d)$.  

An $A^d_n$-module $X$ is called {\em rigid} if $\Ext^i(X,X)=0$ for all
$i>0$.  Note that if $X$ is a summand of $\leftsub{A_n^d}{M}$, then
$X$ is rigid iff $\Ext^d(X,X)=0$.  

The previous theorem combines with Theorems~\ref{th1} and \ref{th2} to yield the
following statement, the $d=1$ case of which is essentially contained in
\cite{BK}.  

\begin{corollary}\label{sumcor} 
For fixed $n$ and $d$, there are natural bijections between the following sets:
\begin{enumerate}
\item \label{sumone}
$\{$Sets of $n+d-1 \choose d$ non-intertwining $(d+1)$-tuples in $\Index{n+2d}\}$.
\item \label{sumtwo}
$\{$triangulations in $S(n+2d,2d)\}$.
\item \label{sumthree}
$\{$isomorphism classes of basic summands of $\leftsub{A^d_n}M$ with $n+d-1 \choose d$ indecomposable summands, which are rigid$\}$. 
\end{enumerate}
\end{corollary}

This is not as strong a statement as we would like to make: specifically,
we would like to replace (\ref{sumthree}) with 
\begin{enumerate}\item[($3'$)]
$\{$isomorphism classes of summands of $\leftsub{A^d_n}{M}$ which are tilting modules $\}$.
\end{enumerate}
In order to so, we need to study how to mutate tilting modules, that is, how to replace one summand of a tilting module by something else.

We need the following piece of notation: For $(i_0, \ldots, i_d)$ and $(j_0, \ldots, j_d)$ in $\mathbb{Z}^{d+1}$ and $X \subseteq \{0, \ldots, d\}$ we write $m_X((i_0, \ldots, i_d),(j_0, \ldots, j_d)) = (\ell_0, \ldots, \ell_d)$ with $\ell_x = i_x$ if $x \in X$, and $\ell_x = j_x$ if $x \not\in X$.

\begin{theorem} \label{theo.tilt-exchange}
Let $T \oplus \leftsub{A_n^d}{M}_{i_0, \ldots, i_d}$ be a tilting $A_n^d$-module, with $T \in \add \leftsub{A_n^d}{M}$. Assume $M_{j_0, \ldots, j_d} \not\in \add T \oplus \leftsub{A_n^d}{M}_{i_0, \ldots, i_d}$ such that $T \oplus \leftsub{A_n^d}{M}_{j_0, \ldots, j_d}$ is rigid. Then
\begin{enumerate}
\item  $T \oplus \leftsub{A_n^d}{M}_{j_0, \ldots, j_d}$ is a tilting $A_n^d$-module.
\item Either $\Ext_{A_n^d}^d(\leftsub{A_n^d}{M}_{i_0, \ldots, i_d}, \leftsub{A_n^d}{M}_{j_0, \ldots, j_d}) = 0$ and $\Ext_{A_n^d}^d(\leftsub{A_n^d}{M}_{j_0, \ldots, j_d}, \leftsub{A_n^d}{M}_{i_0, \ldots, i_d}) \neq 0$, or the other way around.
\item We have
\begin{align*}
& \{ M_{m_X((i_0, \ldots, i_d), (j_0, \ldots, j_d))} \mid X \subseteq \{0, \ldots, d\} \colon m_X((i_0, \ldots, i_d), (j_0, \ldots, j_d)) \in \Index{n+2d} \} \\ & \qquad \subseteq \add (T \oplus M_{i_0, \ldots, i_d} \oplus M_{j_0, \ldots, j_d}).
\end{align*}
That is, for $(\ell_0,\dots,\ell_d) \in \Index{n+2d}$ with $\ell_k$ equalling either $i_k$ or $j_k$, we have that $(\ell_0, \ldots, \ell_d)$ is the index of a summand of $T$, unless $(\ell_0, \ldots, \ell_d)$ equals $(i_0, \ldots, i_d)$ or $(j_0, \ldots, j_d)$.
\item Assume $\Ext_{A_n^d}^d(\leftsub{A_n^d}{M}_{j_0, \ldots, j_d}, \leftsub{A_n^d}{M}_{i_0, \ldots, i_d}) \neq 0$. Then there is a non-split exact sequence
\[ \leftsub{A_n^d}{M}_{i_0, \ldots, i_d} \mono[30] E_d \to[30] \cdots \to[30] E_1 \epi[30] \leftsub{A_n^d}{M}_{j_0, \ldots, j_d} \]
such that
\begin{enumerate}
\item \[ E_r = \bigoplus_{\substack{X \subseteq \{0, \ldots, d\} \\ m_X((i_0, \ldots, i_d),(j_0, \ldots, j_d)) \in \Index{n+2d} \\ |X| = r}} M_{m_X((i_0, \ldots, i_d),(j_0, \ldots, j_d))}. \]
\item This sequence is a $(T \oplus M_{i_0, \ldots, i_d})$-resolution of $M_{j_0, \ldots, j_d}$, and a $(T \oplus M_{j_0, \ldots, j_d})$-coresolution of $M_{i_0, \ldots, i_d}$.
\end{enumerate}
\item For $\Ext_{A_n^d}^d(\leftsub{A_n^d}{M}_{i_0, \ldots, i_d}, \leftsub{A_n^d}{M}_{j_0, \ldots, j_d}) \neq 0$ we have a dual version of (4).
\end{enumerate}
\end{theorem}

\subsection{The \texorpdfstring{$d$}{d}-Auslander-Reiten translation}

This subsection contains background on the $d$-Auslander-Reiten translation. All results here are due to Iyama, and can be found in \cite{Iy_n-Auslander}.

\begin{definition}
For $d \in \mathbb{N}$ define the \emph{$d$-Auslander-Reiten translation} and \emph{inverse $d$-Auslander-Reiten translation} by
\begin{align*}
\tau_d & := \tau \Omega^{d-1} \colon \stabmod \Lambda \to[30] \costmod \Lambda, \qquad \text{and} \\
\tau_d^- & := \tau^- \Omega^{-(d-1)} \colon \costmod \Lambda \to[30] \stabmod \Lambda, \qquad \text{respectively.}
\end{align*}
Here $\tau$ ($\tau^-$) denotes the usual (inverse) Auslander-Reiten translation, $\Omega$ denotes the syzygy and $\Omega^{-1}$ the cosyzygy functor.
\end{definition}

The following result of Iyama tells us that, for a $d$-representation finite algebra, we can actually calculate a module $M$ as in Definition~\ref{def.d-repfin} and that $M$ is unique up to multiplicity. 

\begin{theorem}[\cite{Iy_n-Auslander}] \label{theo.ncto_from_proj}
Let $\Lambda$ be $d$-representation finite, and $M$ be a $d$-cluster tilting module. Then
\[ \add M = \add \{ \tau_d^{-i} \Lambda \mid i \geq 0\} = \add \{ \tau_d^i D \Lambda \mid i \geq 0 \}. \]
\end{theorem}

Finally we will need the following $d$-version of the classical Auslander-Reiten formula.

\begin{observation} \label{obs.ARformula}
For $M, N \in \mod \Lambda$ we have
\[ \Ext_{\Lambda}^d(M, N) = \Ext_{\Lambda}^1(\Omega^{d-1}M, N) = D \overline{\Hom}_{\Lambda}(N, \underbrace{\tau \Omega^{d-1}}_{= \tau_d} M), \]
where the second equality follows from the classical Auslander-Reiten formula. Similarly
\[ \Ext_{\Lambda}^d(M,N) = D \underline{\Hom}_{\Lambda}(\tau_d^- N, M). \]
\end{observation}

\subsection{Proof of Theorems~\texorpdfstring{\ref{theo.indexing}}{3.4} 
and~\texorpdfstring{\ref{theo.indexing.props}}{3.6}, and 
Corollary~\texorpdfstring{\ref{sumcor}}{3.7}}

We will actually prove (and need) the following refinement of Theorem~\ref{theo.indexing.props}(3):

\begin{proposition} \label{prop.basichom}
In the setup of Theorem~\ref{theo.indexing.props}, there are maps
\[ h_{i_0, \ldots, i_d}^{j_0, \ldots, j_d} \colon \leftsub{A_n^d}{M}_{i_0, \ldots, i_d} \to[30] \leftsub{A_n^d}{M}_{j_0, \ldots, j_d} \]
such that
\begin{enumerate}
\item $h_{i_0, \ldots, i_d}^{j_0, \ldots, j_d} \neq 0 \iff i_0-1 < j_0 < i_1-1 < j_1 < \cdots < i_d-1 < j_d$,
\item $\Hom_{A_n^d}(\leftsub{A_n^d}{M}_{i_0, \ldots, i_d}, \leftsub{A_n^d}{M}_{j_0, \ldots, j_d}) = k h_{i_0, \ldots, i_d}^{j_0, \ldots, j_d}$, and
\item $h_{i_0, \ldots, i_d}^{j_0, \ldots, j_d} h_{j_0, \ldots, j_d}^{\ell_0, \ldots, \ell_d} = h_{i_0, \ldots, i_d}^{\ell_0, \ldots, \ell_d}$ whenever $h_{i_0, \ldots, i_d}^{j_0, \ldots, j_d} \neq 0 \neq h_{j_0, \ldots, j_d}^{\ell_0, \ldots, \ell_d}$.
\end{enumerate}
\end{proposition}

Note that (1) and (2) are just a reformulation of Theorem~\ref{theo.indexing.props}(3). The new statement is (3), which claims that the bases can be chosen in a compatible way.

We will also need the following proposition, describing how the (inverse) $d$-Auslander-Reiten translation acts on the summands $\leftsub{A_n^d}{M}$ in terms of the indexing of Theorem~\ref{theo.indexing}.

\begin{proposition} \label{prop.tau}
In the setup of Theorem~\ref{theo.indexing.props} we have
\begin{align*}
\tau_d (\leftsub{A_n^d}{M}_{i_0, \ldots, i_d}) & = \left\{ \begin{array}{ll} 0 & \text{if } i_0 = 1, \\ \leftsub{A_n^d}{M}_{i_0-1, \ldots, i_d-1} & \text{otherwise} \end{array} \right. \text{, and} \\
\tau_d^- (\leftsub{A_n^d}{M}_{i_0, \ldots, i_d}) & = \left\{ \begin{array}{ll} 0 & \text{if } i_d = n + 2d, \\ \leftsub{A_n^d}{M}_{i_0+1, \ldots, i_d+1} & \text{otherwise.} \end{array} \right.
\end{align*}
\end{proposition}

The entire proof of Theorems~\ref{theo.indexing} and \ref{theo.indexing.props} and Propositions~\ref{prop.basichom} and \ref{prop.tau} is built up as an induction on $d$. That is, we assume all the statements to already be known for $A_n^{d-1}$. In particular, the modules $\leftsub{A_n^{d-1}}{M}_{i_0, \ldots, i_{d-1}} \in \mod A_n^{d-1}$, and $\leftsub{A_n^d}{P}_{i_0, \ldots, i_{d-1}}, \leftsub{A_n^d}{S}_{i_0, \ldots, i_{d-1}}, \leftsub{A_n^d}{I}_{i_0, \ldots, i_{d-1}} \in \mod A_n^d$ are assumed to be constructed, and we assume that we have maps $h_{i_0, \ldots, i_{d-1}}^{j_0, \ldots, j_{d-1}}$.

We start by defining candidates for the modules $\leftsub{A_n^d}{M}_{i_0, \ldots, i_d}$. For $i_0 = 1$ we use Theorem~\ref{theo.indexing.props}(1) as definition for $\leftsub{A_n^d}{M}_{1, i_1, \ldots, i_d}$, that is we set $\leftsub{A_n^d}{M}_{1, i_1, \ldots, i_d} := \leftsub{A_n^d}{P}_{i_1-2, \ldots, i_d-2}$. For $i_0 \neq 1$ we define $\leftsub{A_n^d}{M}_{i_0, \ldots, i_d}$ to be the cokernel of the map
\begin{align*}
& \Hom_{A_n^{d-1}}(\leftsub{A_n^{d-1}}{M}_n^{d-1}, h_{i_0-1, i_2-2, \ldots, i_d-2}^{i_1-2, \ldots, i_d-2}) \colon \\ 
& \quad \underbrace{\Hom_{A_n^{d-1}}(\leftsub{A_n^{d-1}}{M}_n^{d-1}, \leftsub{A_n^{d-1}}{M}_{i_0-1, i_2-2, \ldots, i_d-2})}_{= \leftsub{A_n^d}{P}_{i_0-1, i_2-2, \ldots, i_d-2}} \to[30] \underbrace{\Hom_{A_n^{d-1}}(\leftsub{A_n^{d-1}}{M}_n^{d-1}, \leftsub{A_n^{d-1}}{M}_{i_1-2, \ldots, i_d-2})}_{= \leftsub{A_n^d}{P}_{i_1-2, \ldots, i_d-2}}.
\end{align*}

\begin{observation}
The modules $\leftsub{A_n^d}{M}_{i_0, \ldots, i_d}$ have an indecomposable projective cover by construction. Hence the $\leftsub{A_n^d}{M}_{i_0, \ldots, i_d}$ all have simple top. In particular they are indecomposable.
\end{observation}

\begin{lemma} \label{lemma.composition}
The composition factors of $\leftsub{A_n^d}{M}_{i_0, \ldots, i_d}$ are precisely
\[ \leftsub{A_n^d}{S}_{j_0, \ldots, j_{d-1}} \text{ with }  i_0-1 < j_0 < i_1-1 < j_1 < \cdots < i_{d-1}-1 <  j_{d-1} < i_d-1. \]
\end{lemma}

\begin{proof}
We denote by $[\leftsub{A_n^d}{M}_{i_0, \ldots, i_d} : \leftsub{A_n^d}{S}_{j_0, \ldots, j_{d-1}}]$ the multiplicity of  $\leftsub{A^d_n}{S}_{j_0, \ldots, j_{d-1}}$ as a composition factor of $\leftsub{A_n^d}{M}_{i_0, \ldots, i_d}$. Then
\begin{align*}
[\leftsub{A_n^d}{M}_{i_0, \ldots, i_d} : \leftsub{A_n^d}{S}_{j_0, \ldots, j_{d-1}}] & = \dim \Hom_{A_n^d}(\leftsub{A_n^d}{P}_{j_0, \ldots, j_{d-1}}, \leftsub{A_n^d}{M}_{i_0, \ldots, i_d}). \\
\intertext{For $i_0 = 1$ we have }
[\leftsub{A_n^d}{M}_{i_0, \ldots, i_d} : \leftsub{A_n^d}{S}_{j_0, \ldots, j_{d-1}}] & = \dim \Hom_{A_n^d}(\leftsub{A_n^d}{P}_{j_0, \ldots, j_{d-1}}, \leftsub{A_n^d}{P}_{i_1-2, \ldots, i_d-2}) \\
& = \dim \Hom_{A_n^{d-1}}(\leftsub{A_n^{d-1}}{M}_{j_0, \ldots, j_{d-1}}, \leftsub{A_n^{d-1}}{M}_{i_1-2, \ldots, i_d-2}) \\
& = \left\{ \begin{array}{ll} 1 & \text{if } j_0 < i_1-1 < j_1 < i_2-1 < \cdots < j_{d-1} < i_d-1 \\ 0 & \text{otherwise,} \end{array} \right. \\
\intertext{where the second equality holds by the Yoneda Lemma, and the final one by Theorem~\ref{theo.indexing.props}(3) for $A_n^{d-1}$. Thus in this case the lemma holds. For $i_0 \neq 1$ we have}
[\leftsub{A_n^d}{M}_{i_0, \ldots, i_d} : \leftsub{A_n^d}{S}_{j_0, \ldots, j_{d-1}}] & = \dim \Hom_{A_n^d}(\leftsub{A_n^d}{P}_{j_0, \ldots, j_{d-1}}, \leftsub{A_n^d}{P}_{i_1-2, \ldots, i_d-2}) \\
& \qquad - \dim \{\text{maps factoring through } \leftsub{A_n^d}{P}_{i_0-1, i_2-2, \ldots, i_d-2} \} \\
 & = \dim \Hom_{A_n^{d-1}}(\leftsub{A_n^{d-1}}{M}_{j_0, \ldots, j_{d-1}}, \leftsub{A_n^{d-1}}{M}_{i_1-2, \ldots, i_d-2}) \\
& \qquad - \dim \{\text{maps factoring through } \leftsub{A_n^{d-1}}{M}_{i_0-1, i_2-2, \ldots, i_d-2} \} \\
& = \left\{ \begin{array}{ll} 1 & \text{if } \dim \Hom_{A_n^{d-1}}(\leftsub{A_n^{d-1}}{M}_{j_0, \ldots, j_{d-1}}, \leftsub{A_n^{d-1}}{M}_{i_1-2, \ldots, i_d-2}) = 1 \\ & \text{ and } \Hom_{A_n^{d-1}}(\leftsub{A_n^{d-1}}{M}_{j_0, \ldots, j_{d-1}}, \leftsub{A_n^{d-1}}{M}_{i_0-1, i_2-2, \ldots, i_d-2}) = 0 \\ 0 & \text{otherwise.} \end{array} \right.
\end{align*}
Here the final equality holds by Proposition~\ref{prop.basichom} for $A_n^{d-1}$. As before, we use Theorem~\ref{theo.indexing.props}(3) for $A_n^{d-1}$ to obtain the claim.
\end{proof}

We now define the maps $h_{i_0, \ldots, i_d}^{j_0, \ldots, j_d}$ and prove Proposition~\ref{prop.basichom}.

We define $h_{i_0, \ldots, i_d}^{j_0, \ldots, j_d}$ to be the right vertical map in the following diagram, if such a map exists, and to be $0$ otherwise.
\[ \begin{tikzpicture}[xscale=4,yscale=2]
 \node (A) at (0,1) {$\leftsub{A_n^d}{P}_{i_1-2, \ldots, i_d-2}$};
 \node (B) at (1,1) {$\leftsub{A_n^d}{M}_{i_0, \ldots, i_d}$};
 \node (C) at (0,0) {$\leftsub{A_n^d}{P}_{j_1-2, \ldots, j_d-2}$};
 \node (D) at (1,0) {$\leftsub{A_n^d}{M}_{j_0, \ldots, j_d}$};
 \draw [->>] (A) -- (B);
 \draw [->>] (C) -- (D);
 \draw [->] (A) -- node [left] {$(h_{i_1-2, \ldots, i_d-2}^{j_1-2, \ldots, j_d-2})_*$} (C);
 \draw [->, dashed] (B) -- (D);
\end{tikzpicture} \]
Here the left vertical map is $(h_{i_1-2, \ldots, i_d-2}^{j_1-2, \ldots, j_d-2})_* = \Hom_{A_n^{d-1}}(\leftsub{A_n^{d-1}}{M}, h_{i_1-2, \ldots, i_d-2}^{j_1-2, \ldots, j_d-2})$, and the horizontal maps are the projections coming from the construction of $M_{i_0, \ldots, i_d}$ and  $M_{j_0, \ldots, j_d}$, respectively.

\begin{observation}
\begin{enumerate}
\item Any map $\leftsub{A_n^d}{M}_{i_0, \ldots, i_d} \to \leftsub{A_n^d}{M}_{j_0, \ldots, j_d}$ induces a (not necessarily unique) map between the projective covers. Since $(h_{i_1-2, \ldots, i_d-2}^{j_1-2, \ldots, j_d-2})_*$ is the unique up to scalars map between the projective covers (by the Yoneda Lemma and Proposition~\ref{prop.basichom}(2) for $A_n^{d-1}$) it follows that
\[ \Hom_{A_n^d}(\leftsub{A_n^d}{M}_{i_0, \ldots, i_d}, \leftsub{A_n^d}{M}_{j_0, \ldots, j_d}) = k h_{i_0, \ldots, i_d}^{j_0, \ldots, j_d}. \]
That is, we have shown Proposition~\ref{prop.basichom}(2) for $A_n^d$.
\item Using Proposition~\ref{prop.basichom}(3) for $A_n^{d-1}$ it also follows immediately that, if $h_{i_0, \ldots, i_d}^{j_0, \ldots, j_d}$ and $h_{j_0, \ldots, j_d}^{\ell_0, \ldots, \ell_d}$ are both non-zero, then
\[ h_{i_0, \ldots, i_d}^{j_0, \ldots, j_d} h_{j_0, \ldots, j_d}^{\ell_0, \ldots, \ell_d} = h_{i_0, \ldots, i_d}^{\ell_0, \ldots, \ell_d}. \]
So Proposition~\ref{prop.basichom}(3) holds for $A_n^d$.
\end{enumerate}
\end{observation}

We now complete the proof of Proposition~\ref{prop.basichom} for $A_n^d$ by verifying Proposition~\ref{prop.basichom}(1). We have to consider the following four cases:

{\sc Case $i_0 = 1$ and $j_0 = 1$:} In this case the claim follows immediately from the Yoneda Lemma.

{\sc Case $i_0 > 1$ and $j_0 = 1$:} In this case we have to show that $h_{i_0, \ldots, i_d}^{j_0, \ldots, j_d} = 0$. Looking at projective resolutions we obtain the following diagram:
\[ \begin{tikzpicture}[xscale=4,yscale=2]
 \node (A0) at (-1.5,1) {$\leftsub{A_n^d}{P}_{i_0-1, i_2-2, \ldots, i_d-2}$};
 \node (A) at (0,1) {$\leftsub{A_n^d}{P}_{i_1-2, \ldots, i_d-2}$};
 \node (B) at (1,1) {$\leftsub{A_n^d}{M}_{i_0, \ldots, i_d}$};
 \node (C) at (0,0) {$\leftsub{A_n^d}{P}_{j_1-2, \ldots, j_d-2}$};
 \node (D) at (1,0) {$\leftsub{A_n^d}{M}_{j_0, \ldots, j_d}$};
 \draw [->] (A0) -- node [above] {$(h_{i_0-1, i_2-2, \ldots, i_d-2}^{i_1-2, \ldots, i_d-2})_*$} (A);
 \draw [->>] (A) -- (B);
 \draw [double distance=1.5pt] (C) -- (D);
 \draw [->] (A) -- node [left] {$(h_{i_1-2, \ldots, i_d-2}^{j_1-2, \ldots, j_d-2})_*$} (C);
 \draw [->, dashed] (B) -- node [right] {$\exists$?} (D);
\end{tikzpicture} \]
Clearly, if $h_{i_1-2, \ldots, i_d-2}^{j_1-2, \ldots, j_d-2} = 0$ then the dashed map $h_{i_0, \ldots, i_d}^{j_0, \ldots, j_d}$ will also be $0$. Hence we may assume $h_{i_1-2, \ldots, i_d-2}^{j_1-2, \ldots, j_d-2} \neq 0$. By Proposition~\ref{prop.basichom} for $A_n^{d-1}$ that means
\[ i_1-1 < j_1 < i_2-1 < j_2 < \cdots < i_d-1 < j_d. \]
Now the dashed map exists if and only if the composition $h_{i_0-1, i_2-2, \ldots, i_d-2}^{i_1-2, \ldots, i_d-2} h_{i_1-2, \ldots, i_d-2}^{j_1-2, \ldots, j_d-2} = h_{i_0-1, i_2-2, \ldots, i_d-2}^{j_1-2, \ldots, j_d-2}$ vanishes. By Proposition~\ref{prop.basichom} for $A_n^{d-1}$ this is equivalent to
\[ \text{not} (i_0 < j_1 < i_2-1 < j_2 < \cdots < i_d-1 < j_d). \]
The two conditions above clearly contradict each other (since $i_0 < i_1-1$), so the dashed map in the above diagram does not exist. Hence $h_{i_0, \ldots, i_d}^{j_0, \ldots, j_d} = 0$ by definition, as claimed.

{\sc Case $i_0 = 1$ and $j_0 > 1$:} Looking at the defining projective resolutions, one sees that
\begin{align*}
h_{i_0, \ldots, i_d}^{j_0, \ldots, j_d} \neq 0 \iff & h_{i_1-2, \ldots, i_d-2}^{j_1-2, \ldots, j_d-2} \neq 0 \\ & \quad \text{and } h_{i_1-2, \ldots, i_d-2}^{j_0-1, j_2-2, \ldots, j_d-2} = 0 \\
\iff & (i_1-1 < j_1 < i_2-1 < j_2 < \cdots < i_d-1 < j_d) \\ & \quad \text{and not} (i_1-1 < j_0+1 < i_2-1 < j_2 < \cdots < i_d-1 < j_d) \\
\iff & j_0 < i_1-1 < j_1 < i_2-1 < \cdots < i_d-1 < j_d
\end{align*}
Here the second equivalence holds by Proposition~\ref{prop.basichom} for $A_n^{d-1}$. Hence also in this case Proposition~\ref{prop.basichom}(1) is proven.

{\sc Case $i_0 > 1$ and $j_0 > 1$:} As before we look at the projective resolutions.
\[ \begin{tikzpicture}[xscale=4,yscale=2]
 \node (A0) at (-1.5,1) {$\leftsub{A_n^d}{P}_{i_0-1, i_2-2, \ldots, i_d-2}$};
 \node (A) at (0,1) {$\leftsub{A_n^d}{P}_{i_1-2, \ldots, i_d-2}$};
 \node (B) at (1,1) {$\leftsub{A_n^d}{M}_{i_0, \ldots, i_d}$};
 \node (C0) at (-1.5,0) {$\leftsub{A_n^d}{P}_{j_0-1, j_2-2, \ldots, j_d-2}$};
 \node (C) at (0,0) {$\leftsub{A_n^d}{P}_{j_1-2, \ldots, j_d-2}$};
 \node (D) at (1,0) {$\leftsub{A_n^d}{M}_{j_0, \ldots, j_d}$};
 \draw [->] (A0) -- node [above] {$(h_{i_0-1, i_2-2, \ldots, i_d-2}^{i_1-2, \ldots, i_d-2})_*$} (A);
 \draw [->>] (A) -- (B);
 \draw [->] (C0) -- node [above] {$(h_{j_0-1, j_2-2, \ldots, j_d-2}^{j_1-2, \ldots, j_d-2})_*$} (C);
 \draw [->>] (C) -- (D);
 \draw [->, dashed] (A0) -- (C0);
 \draw [->] (A) -- node [left,pos=.3] {$(h_{i_1-2, \ldots, i_d-2}^{j_1-2, \ldots, j_d-2})_*$} (C);
 \draw [->, dashed] (B) -- node [right] {$\exists$?} (D);
\end{tikzpicture} \]
First consider the case $h_{i_1-2, \ldots, i_d-2}^{j_1-2, \ldots, j_d-2} = 0$. Then we see from the diagram above that also $h_{i_0, \ldots, i_d}^{j_0, \ldots, j_d} = 0$.

Now we assume $h_{i_1-2, \ldots, i_d-2}^{j_1-2, \ldots, j_d-2} \neq 0$. By the discussion in the case $i_0 > 1$, $j_0 = 1$ above we know that this implies $h_{i_0-1, i_2-2, \ldots, i_d-2}^{i_1-2, \ldots, i_d-2} h_{i_1-2, \ldots, i_d-2}^{j_1-2, \ldots, j_d-2} \neq 0$. Thus the above diagram can only be completed if
\begin{align*}
0 \neq & \Hom_{A_n^d}(\leftsub{A_n^d}{P}_{i_0-1, i_2-2, \ldots, i_d-2}, \leftsub{A_n^d}{P}_{j_0-1, j_2-2, \ldots, j_d-2}) \\
= & \Hom_{A_n^{d-1}}(\leftsub{A_n^{d-1}}{M}_{i_0-1, i_2-2, \ldots, i_d-2}, \leftsub{A_n^{d-1}}{M}_{j_0-1, j_2-2, \ldots, j_d-2}),
\end{align*}
that is if $h_{i_0-1, i_2-2, \ldots, i_d-2}^{j_0-1, j_2-2, \ldots, j_d-2} \neq 0$. Conversely, if $h_{i_0-1, i_2-2, \ldots, i_d-2}^{j_0-1, j_2-2, \ldots, j_d-2} \neq 0$ then it follows from Proposition~\ref{prop.basichom} for $A_n^{d-1}$ that $(h_{i_0-1, i_2-2, \ldots, i_d-2}^{j_0-1, j_2-2, \ldots, j_d-2})_*$ makes the left square of the above diagram commutative. Finally note that, if the left square above is commutative, it induces a non-zero map on the cokernels if and only if $h_{i_1-2, \ldots, i_d-2}^{j_1-2, \ldots, j_d-2}$ does not factor through $h_{j_0-1, j_2-2, \ldots, j_d-2}^{j_1-2, \ldots, j_d-2}$.

Summing up we have
\begin{align*}
h_{i_0, \ldots, i_d}^{j_0, \ldots, j_d} \neq 0 \iff & h_{i_1-2, \ldots, i_d-2}^{j_1-2, \ldots, j_d-2} \neq 0 \\ & \quad \text{and } h_{i_0-1, i_2-2, \ldots, i_d-2}^{j_0-1, j_2-2, \ldots, j_d-2} \neq 0 \\ & \quad \text{and } h_{i_1-2, \ldots, i_d-2}^{j_0-1, j_2-2, \ldots, j_d-2} = 0 \\
\iff & (i_1-1 < j_1 < i_2-1 < j_2 < \cdots < i_d-1 < j_d) \\ & \quad \text{and } (i_0 < j_0+1 < i_2-1 < j_2 < \cdots < i_d-1 < j_d) \\ & \quad \text{and not} (i_1-1 < j_0+1 < i_2-1 < j_2 < \cdots < i_d-1 < j_d) \\
\iff & (i_0-1 < j_0 < i_1-1 < j_1 < \cdots < i_d-1 < j_d)
\end{align*}
That completes the proof of Proposition~\ref{prop.basichom}, and hence also of Theorem~\ref{theo.indexing.props}(3) for $A_n^d$.

We now determine projective resolutions of our modules $M_{i_0, \ldots, i_d}$.

\begin{proposition} \label{prop.proj_resol}
For $i_0 \neq 1$ the projective resolution of $\leftsub{A_n^d}{M}_{i_0, \ldots, i_d}$ is
\begin{align*}
& \leftsub{A_n^d}{M}_{i_0, \ldots, i_d} \arrow[30]{<<-} \underbrace{\leftsub{A_n^d}{M}_{1, i_1, \ldots, i_d}}_{= \leftsub{A_n^d}{P}_{i_1-2, \ldots i_d-2}} \arrow[30]{<-} \underbrace{\leftsub{A_n^d}{M}_{1, i_0+1, i_2, \ldots i_d}}_{\leftsub{A_n^d}{P}_{i_0-1, i_2-2, \ldots i_d-2}} \arrow[30]{<-} \underbrace{\leftsub{A_n^d}{M}_{1, i_0+1, i_1+1,i_3, \ldots, i_d}}_{\leftsub{A_n^d}{P}_{i_0-1, i_1-1,i_3-2, \ldots, i_d-2}} \arrow[30]{<-} \cdots \\ % ->>
& \qquad \qquad \qquad \cdots \arrow[30]{<-} \underbrace{\leftsub{A_n^d}{M}_{1, i_0+1, \ldots, i_{d-2}+1, i_d}}_{\leftsub{A_n^d}{P}_{i_0-1, \ldots, i_{d-2}-1, i_d-2}} \arrow[30]{<-<} \underbrace{\leftsub{A_n^d}{M}_{1, i_0+1, \ldots, i_{d-1}+1}}_{\leftsub{A_n^d}{P}_{i_0-1, \ldots, i_{d-1}-1}}
\end{align*}
\end{proposition}

\begin{proof}
Since, by Proposition~\ref{prop.basichom} we know precisely what morphisms there are between the modules of the form $\leftsub{A_n^d}{M}_{j_0, \ldots, j_d}$ it is straightforward to calculate the resolution of $\leftsub{A_n^d}{M}_{i_0, \ldots, i_d}$ by modules of the form $\leftsub{A_n^d}{M}_{1, j_1, \ldots, j_d}$. These are the projective $A_n^d$-modules by construction.
\end{proof}

\begin{remark} \label{rem.inj_resol}
We also have the following opposite version of Proposition~\ref{prop.proj_resol}:

For $i_d \neq n+2d$ the coresolution of $\leftsub{A_n^d}{M}_{i_0, \ldots, i_d}$ by modules of the form $\leftsub{A_n^d}{M}_{j_0, \ldots j_{d-1}, n+2d}$ is
\begin{align*}
& \leftsub{A_n^d}{M}_{i_0, \ldots, i_d} \to[30] \leftsub{A_n^d}{M}_{i_0, \ldots i_{d-1}, n+2d} \to[30] \leftsub{A_n^d}{M}_{i_0, \ldots, i_{d-2}, i_d-1, n+2d} \to[30] \leftsub{A_n^d}{M}_{i_0, \ldots, i_{d-3}, i_{d-1}-1, i_d-1, n+2d} \to[30] \cdots \\
& \qquad \qquad \qquad \cdots \to[30] \leftsub{A_n^d}{M}_{i_0, i_2-1, \ldots, i_d-1, n+2d} \to[30] \leftsub{A_n^d}{M}_{i_1-1, \ldots, i_d-1, n+2d} \to 0.
\end{align*}
Once we have shown Theorem~\ref{theo.indexing.props}(2) it follows that the above is an injective resolution, and in particular that it is exact. (We remark that it is actually easy to verify directly that this sequence is exact.)
\end{remark}

We are now ready to determine the injective $A_n^d$-modules.

\begin{proof}[Proof of Theorem~\ref{theo.indexing.props}(2)]
By definition we have $\leftsub{A_n^d}{I}_{i_0, \ldots, i_{d-1}} = \nu \leftsub{A_n^d}{P}_{i_0, \ldots, i_{d-1}}$, where $\nu$ denotes the Nakayama functor, that is the functor taking every projective module to the corresponding injective module. Hence
\begin{align*}
\leftsub{A_n^d}{I}_{i_0, \ldots, i_{d-1}} & = \nu \Hom_{A_n^{d-1}}(\leftsub{A_n^{d-1}}{M}, \leftsub{A_n^{d-1}}{M}_{i_0, \ldots, i_{d-1}}) \\
& = D \Hom_{A_n^{d-1}}(\leftsub{A_n^{d-1}}{M}_{i_0, \ldots, i_{d-1}}, \leftsub{A_n^{d-1}}{M}).
\end{align*}
If $i_0 = 1$ we have
\begin{align*}
\leftsub{A_n^d}{I}_{i_0, \ldots, i_{d-1}} & = D \Hom_{A_n^{d-1}}(\leftsub{A_n^{d-1}}{P}_{i_1-2, \ldots, i_{d-1}-2}, \leftsub{A_n^{d-1}}{M}) \\
& = \Hom_{A_n^{d-1}}(\leftsub{A_n^{d-1}}{M}, \leftsub{A_n^{d-1}}{I}_{i_1-2, \ldots, i_{d-1}-2}).
\end{align*}
By Theorem~\ref{theo.indexing.props}(2) for $A_n^{d-1}$ we know that $\leftsub{A_n^{d-1}}{I}_{i_1-2, \ldots, i_{d-1}-2} =  \leftsub{A_n^{d-1}}{M}_{i_1-2, \ldots, i_{d-1}-2, n+2(d-1)}$, so we obtain $\leftsub{A_n^d}{I}_{1, i_1, \ldots, i_{d-1}} = \leftsub{A_n^d}{M}_{1, i_1, \ldots, i_{d-1}, n+2d}$ as claimed.

If $i_0 \neq 1$ we have
\begin{align*}
\leftsub{A_n^d}{I}_{i_0, \ldots, i_{d-1}} & = D \Hom_{A_n^{d-1}}(\leftsub{A_n^{d-1}}{M}_{i_0, \ldots, i_{d-1}}, \leftsub{A_n^{d-1}}{M}) \\
& = \Ext_{A_n^{d-1}}^{d-1}(\leftsub{A_n^{d-1}}{M}, \tau_{d-1} ( \leftsub{A_n^{d-1}}{M}_{i_0, \ldots, i_{d-1}})) && \text{(by Observation~\ref{obs.ARformula})} \\
& = \Ext_{A_n^{d-1}}^{d-1}(\leftsub{A_n^{d-1}}{M}, \leftsub{A_n^{d-1}}{M}_{i_0-1, \ldots, i_{d-1}-1}) && \text{(by Proposition~\ref{prop.tau} for $A_n^{d-1}$)}
\end{align*}
To compute this $\Ext$-space we make use of the following: By Remark~\ref{rem.inj_resol} we know that the coresolution of $\leftsub{A_n^{d-1}}{M}_{i_0-1, \ldots, i_{d-1}-1}$ by modules of the form $\leftsub{A_n^{d-1}}{M}_{j_0, \ldots, j_{d-2}, n+2(d-1)}$ is
\begin{align*}
& \leftsub{A_n^{d-1}}{M}_{i_0-1, \ldots, i_{d-1}-1} \to[30] \leftsub{A_n^{d-1}}{M}_{i_0-1, \ldots, i_{d-2}-1, n+2(d-1)} \to[30] \cdots \\
& \qquad \qquad \to[30] \leftsub{A_n^{d-1}}{M}_{i_0-1, i_2-2, \ldots, i_{d-1}-2, n+2(d-1)} \to[30] \leftsub{A_n^{d-1}}{M}_{i_1-2, \ldots, i_{d-1}-2, n+2(d-1)} \to[30] 0,
\end{align*}
and since we assume Theorem~\ref{theo.indexing.props}(2) to hold for $A_n^{d-1}$ this is an injective coresolution. Hence
\begin{align*}
& \Ext_{A_n^{d-1}}^{d-1}(\leftsub{A_n^{d-1}}{M}, \leftsub{A_n^{d-1}}{M}_{i_0-1, \ldots, i_{d-1}-1}) \\
& \quad = \Cok \left[ \Hom_{A_n^{d-1}}(\leftsub{A_n^{d-1}}{M}, \leftsub{A_n^{d-1}}{M}_{i_0-1, i_2-2, \ldots, i_{d-1}-2, n+2d-2}) \to \Hom_{A_n^{d-1}}(\leftsub{A_n^{d-1}}{M}, \leftsub{A_n^{d-1}}{M}_{i_1-2, \ldots, i_{d-1}-2, n+2d-2}) \right].
\end{align*}
Now note that $\leftsub{A_n^d}{M}_{i_0, \ldots, i_{d-1}, n+2d}$ is defined to be this cokernel. Hence we have shown
\[ \leftsub{A_n^d}{I}_{i_0, \ldots, i_{d-1}} = \leftsub{A_n^d}{M}_{i_0, \ldots, i_{d-1}, n+2d}. \qedhere \]
\end{proof}

Knowing the injective $A_n^d$-modules we can now calculate the $d$-Auslander-Reiten translation, that is, prove Proposition~\ref{prop.tau}.

\begin{proof}[Proof of Proposition~\ref{prop.tau}]
We prove the second equality; the proof of the first one is similar. If $i_d = n+2d$ then $\leftsub{A_n^d}{M}_{i_0, \ldots, i_d}$ is injective, so $\tau_d^{-} ( \leftsub{A_n^d}{M}_{i_0, \ldots, i_d} ) = 0$ as claimed. Therefore assume $i_d \neq n+2d$. The rule for calculating $\tau_d^{-} ( \leftsub{A_n^d}{M}_{i_0, \ldots, i_d} )$ is the following: Take the $d$-th map in the injective coresolution of $\leftsub{A_n^d}{M}_{i_0, \ldots, i_d}$ (not counting the inclusion of $\leftsub{A_n^d}{M}_{i_0, \ldots, i_d}$ into its injective envelope). By Remark~\ref{rem.inj_resol} this map is $h_{i_0, i_2-1, \ldots, i_d-1, n+2d}^{i_1-1, \ldots, i_d-1, n+2d}$. Replace it by the corresponding map between projective modules, that is by $h_{1, i_0+2, i_2+1, \ldots, i_d+1}^{1, i_1+1, \ldots, i_d+1}$. Now
\[ \tau_d^{-} ( \leftsub{A_n^d}{M}_{i_0, \ldots, i_d} ) = \Cok h_{1, i_0+2, i_2+1, \ldots, i_d+1}^{1, i_1+1, \ldots, i_d+1} = \leftsub{A_n^d}{M}_{i_0+1, \ldots, i_d+1}. \qedhere \]
\end{proof}

Now we are ready to complete the proofs of Theorems~\ref{theo.indexing} and \ref{theo.indexing.props}

\begin{proof}[Proof of Theorem~\ref{theo.indexing}]
We have seen in Lemma~\ref{lemma.composition} that our modules $\leftsub{A_n^d}{M}_{i_0, \ldots, i_d}$ have the desired composition factors. It remains to show that they are precisely the direct summands of $\leftsub{A_n^d}{M}$.

By construction the modules $\leftsub{A_n^d}{M}_{1, i_1, \ldots, i_d}$ are precisely the projective $A_n^d$-modules. Hence, by Theorem~\ref{theo.ncto_from_proj}, we only have to see that the modules of the form $\leftsub{A_n^d}{M}_{i_0, \ldots, i_d}$ are precisely the modules obtained by applying $\tau_d^-$-powers to modules of the form $\leftsub{A_n^d}{M}_{1, i_1, \ldots, i_d}$. This follows immediately from Proposition~\ref{prop.tau}.
\end{proof}

\begin{proof}[Proof of Theorem~\ref{theo.indexing.props}]
(1) holds by construction, and we have proven (2) above. (3) is a consequence of Proposition~\ref{prop.basichom}. It remains to deduce (4). This follows immediately from (3), Observation~\ref{obs.ARformula}, and Proposition~\ref{prop.tau}.
\end{proof}

Finally we show that this also completes the proof of Corollary~\ref{sumcor}.

\begin{proof}[Proof of Corollary~\ref{sumcor}]
The correspondence between (\ref{sumone}) and (\ref{sumtwo}) is given
by Theorems \ref{th1} and \ref{th2}.  The correspondence between 
(\ref{sumone}) and (\ref{sumthree}) is immediate from Theorem
\ref{theo.indexing.props}.
\end{proof}

\subsection{Exchanging tilting modules -- proof of 
Theorem~\texorpdfstring{\ref{theo.tilt-exchange}}{3.8}}

Throughout this subsection we work only over the algebra $A_n^d$. Hence we can omit the left indices $A_n^d$ without risking confusion. We start by proving the second part of Theorem~\ref{theo.tilt-exchange}.

\begin{proof}[Proof of Theorem~\ref{theo.tilt-exchange}(2)]
We have to exclude the cases that either $\Ext_{A_n^d}^d(M_{i_0, \ldots, i_d}, M_{j_0, \ldots, j_d})$ and $\Ext_{A_n^d}^d(M_{j_0, \ldots, j_d}, M_{i_0, \ldots, i_d})$ are both zero, or both non-zero.

It follows from Theorem~\ref{theo.indexing.props}(4) that they cannot both be non-zero. Hence assume they are both zero. Then $T \oplus M_{i_0, \ldots, i_d} \oplus M_{j_0, \ldots, j_d}$  has no self-extensions. This contradicts the fact that $T \oplus M_{i_0, \ldots, i_d}$ is a tilting module with $M_{j_0, \ldots, j_d} \not\in \add (T \oplus M_{i_0, \ldots, i_d})$.
\end{proof}

Since the two possibilities of Theorem~\ref{theo.tilt-exchange}(2) are dual, we may assume we are in the first mentioned situation, that is we have $\Ext_{A_n^d}^d(M_{j_0, \ldots, j_d}, M_{i_0, \ldots, i_d}) \neq 0$.

We now point out that Theorem~\ref{theo.tilt-exchange}(4) implies Theorem~\ref{theo.tilt-exchange}(1 and 3):

\begin{proof}[Proof of Theorem~\ref{theo.tilt-exchange}(1), assuming Theorem~\ref{theo.tilt-exchange}(4)]
By Theorem~\ref{theo.tilt-exchange}(4b) the monomorphism $M_{i_0, \ldots, i_d} \mono E_d$ is a left $T$-approximation. We denote its cokernel by $C$. By \cite{RieScho}, $T \oplus C$ is a tilting $A_n^d$-module. We again have a monomorphic left $T$-approximation $C \mono E_{d-1}$. Iterating this argument we see that $T \oplus M_{j_0, \ldots, j_d}$ is a tilting $A_n^d$-module.
\end{proof}

\begin{proof}[Proof of Theorem~\ref{theo.tilt-exchange}(3), assuming Theorem~\ref{theo.tilt-exchange}(4)]
This follows immediately from the fact that the left set in (3) is, by Theorem~\ref{theo.tilt-exchange}(4a), contained in \[ \add (M_{i_0, \ldots, i_d} \oplus E_d \oplus \cdots \oplus E_1 \oplus M_{j_0, \ldots, j_d}) \subseteq \add (T \oplus M_{i_0, \ldots, i_d} \oplus M_{j_0, \ldots, j_d}), \]
where the above inclusion follows from \ref{theo.tilt-exchange}(4b).
\end{proof}

Thus it only remains to prove Theorem~\ref{theo.tilt-exchange}(4). We start by constructing an exact sequence with the desired terms.

\begin{proposition} \label{prop.E_exact}
Let $(i_0, \ldots, i_d), (j_0, \ldots, j_d) \in \Index{n+2d}$ with $(i_0, \ldots, i_d)\wr(j_0, \ldots, j_d)$. Then there is an exact sequence
\[ \mathbb{E} \colon \underbrace{E_{d+1}}_{\mathclap{= M_{i_0, \ldots, i_d}}} \mono[30] E_d \to[30] \cdots \to[30] E_1 \epi[30] \underbrace{E_0}_{\mathclap{= M_{j_0, \ldots, j_d}}} \]
with
\[ E_r = \bigoplus_{\substack{X \subseteq \{0, \ldots, d\} \\ |X| = r}} M_{m_X}. \]
Here $m_X$ is short for $m_X((i_0, \ldots, i_d), (j_0, \ldots, j_d))$. We set $M_{m_X} = 0$ whenever $m_X \not\in \Index{n+2d}$.
\end{proposition}

\begin{proof}
We describe the exact sequence
\[ \mathbb{E} \colon E_{d+1} \mono[30] E_d \to[30] \cdots \to[30] E_1 \epi[30] E_0 \]
by specifying that the component maps are
\[ M_{m_X} \to[30] M_{m_Y} \colon \left\{ \begin{array}{ll} (-1)^{|\{x \in X \mid x < y \}|} h_{m_X}^{m_Y} &  \text{if } m_X, m_Y \in \Index{n+2d} \text{ and } X = Y \cup \{y\} \text{ for some } y \not\in Y \\ 0 & \text{otherwise.} \end{array} \right. \]

To check that the sequence $\mathbb{E}$ with the maps as defined above is exact it suffices to check that for every indecomposable projective $A_n^d$-module $P_{\ell_0, \ldots, \ell_{d-1}}$ the sequence $\Hom_{A_n^d}(P_{\ell_0, \ldots, \ell_{d-1}}, \mathbb{E})$ is exact. If $\Hom_{A_n^d}(P_{\ell_0, \ldots, \ell_{d-1}}, M_{m_X}) = 0$ for all $X \subseteq \{0, \ldots, d\}$ this is clearly true. Otherwise, writing $m_X = (m_0, \ldots, m_d)$, Theorem~\ref{theo.indexing.props}(3) implies that, if $m_X \in \Index{n+2d}$, we have
\begin{equation}\label{doubleimp} 
\Hom_{A_n^d}(P_{\ell_0, \ldots, \ell_{d-1}}, M_{m_X}) \neq 0 \iff m_0 < \ell_0+1 < m_1 < \ell_1+1 < \cdots < \ell_{d-1}+1 < m_d. \end{equation}
Since, in addition, neither side of \eqref{doubleimp} holds if $m_X$ is not
separated, \eqref{doubleimp} holds for all $X$.
Thus, if we split the set $\{0, \ldots, d\}$ into the three parts
\begin{align*} 
X_0 & = \{x \mid \ell_{x-1}+1 < i_x < \ell_x+1 \text{ and not} (\ell_{x-1}+1 < j_x < \ell_x+1) \}, \\
X_1 & = \{x \mid \ell_{x-1}+1 < i_x < \ell_x+1 \text{ and } \ell_{x-1}+1 < j_x < \ell_x+1 \} \text{, and} \\
X_2 & = \{x \mid \text{not} (\ell_{x-1}+1 < i_x < \ell_x+1) \text{ and } \ell_{x-1}+1 < j_x < \ell_x+1 \}, 
\end{align*}
(where in all cases we assume the conditions $\ell_{-1}+1 < \, ?$ and $? < \ell_d+1$ to always be true), we have that
\begin{equation} \label{eq.hom_to_make_exact} \Hom_{A_n^d}(P_{\ell_0, \ldots, \ell_{d-1}}, M_{m_X}) = \left\{ \begin{array}{ll} k & \text{if } X_0 \subseteq X \subseteq X_0 \cup X_1 \\ 0 & \text{otherwise.} \end{array} \right. \end{equation}
We now claim that $X_1 \neq \emptyset$. Assume conversely that $X_1 = \emptyset$, and hence $X_0 \cup X_2 = \{0, \ldots, d\}$. The assumption $(i_0, \ldots, i_d) \wr (j_0, \ldots, j_d)$ implies $i_0 < j_0$ and $i_d < j_d$. These inequalities imply $0 \not\in X_2$ and $d \not\in X_0$, respectively. Hence $0 \in X_0$ and $d \in X_2$. Therefore there is $x$ with $x \in X_0$ and $x+1 \in X_2$. Since $i_x < j_x$ it follows that $j_x \not< \ell_x+1$, and since $i_{x+1} < j_{x+1}$ it follows that $\ell_x+1 \not< i_{x+1}$. Hence $i_{x+1} \leq \ell_x+1 \leq j_x$, contradicting the fact that $(i_0, \ldots, i_d) \wr (j_0, \ldots, j_d)$. This proves the claim that $X_1 \neq \emptyset$.

Now it follows from \eqref{eq.hom_to_make_exact} and the signs in the definition of the component maps that $\Hom_{A_n^d}(P_{\ell_0, \ldots, \ell_{d-1}}, \mathbb{E})$ is the Koszul complex $[k \to k]^{\otimes |X_1|}$, shifted to the appropriate position. In particular it is exact.
\end{proof}

\begin{proposition} \label{prop.E_nonsplit}
Let $(i_0, \ldots, i_d), (j_0, \ldots, j_d) \in \Index{n+2d}$ with $(i_0, \ldots, i_d) \wr (j_0, \ldots, j_d)$. Then the exact sequence $\mathbb{E}$ of Proposition~\ref{prop.E_exact} represents a non-zero element (and thus a $k$-basis) of $\Ext_{A_n^d}^d(M_{j_0, \ldots, j_d}, M_{i_0, \ldots, i_d})$.
\end{proposition}

\begin{proof}
By Proposition~\ref{prop.proj_resol} we know that the projective resolution of $M_{j_0, \ldots, j_d}$ is as given in the upper row of the following diagram.
\[ \begin{tikzpicture}[xscale=3.5,yscale=1.5]
 \node (P1) at (0,1) {$M_{1, j_0+1, \ldots, j_{d-1}+1}$};
 \node (P2) at (1,1) {$M_{1, j_0+1, \ldots, j_{d-2}+1, j_d}$};
 \node (P3) at (2,1) {$\cdots$};
 \node (P4) at (3,1) {$M_{1, j_1, \ldots, j_d}$};
 \node (P5) at (4,1) {$M_{j_0, \ldots, j_d}$};
 \node (E1) at (0,0) {$E_{d+1}$};
 \node (E2) at (1,0) {$E_d$};
 \node (E3) at (2,0) {$\cdots$};
 \node (E4) at (3,0) {$E_1$};
 \node (E5) at (4,0) {$E_0$};
 \draw [>->] (P1) -- (P2);
 \draw [->] (P2) -- (P3);
 \draw [->] (P3) -- (P4);
 \draw [->>] (P4) -- (P5);
 \draw [>->] (E1) -- (E2);
 \draw [->] (E2) -- (E3);
 \draw [->] (E3) -- (E4);
 \draw [->>] (E4) -- (E5);
 \draw [->, dashed] (P1) -- node [left] {$f_{d+1}$} (E1);
 \draw [->, dashed] (P2) -- node [left] {$f_d$} (E2);
 \draw [->, dashed] (P4) -- node [left] {$f_1$} (E4);
 \draw [double distance=1.5pt] (P5) -- (E5);
\end{tikzpicture} \]
Using Proposition~\ref{prop.basichom} it is easy to verify that we may choose the maps $f_r$ on components by
\begin{align*}
& M_{1, j_0+1, \ldots, j_{r-2}+1, j_r, \ldots, j_d} \to M_{m_X} \colon \\
& \qquad \left\{ \begin{array}{ll} \pm h_{1, j_0+1, \ldots, j_{r-2}+1, j_r, \ldots j_d}^{i_0, \ldots, i_{r-1}, j_r, \ldots, j_d} & \text{if } X = \{0, \ldots, r-1\} \\ 0 & \text{otherwise. } \end{array} \right.
\end{align*}
Since $f_{d+1}$ does not factor through $h_{1, j_0+1, \ldots, j_{d-1}+1}^{1, j_0+1, \ldots, j_{d-2}+1, j_d}$ it follows that the extension $\mathbb{E}$ is non-split.
\end{proof}

The rest of this section leads to a proof of Theorem~\ref{theo.tilt-exchange}(4). Therefore we assume that we are in the setup of Theorem~\ref{theo.tilt-exchange}, that is, we are given $T$, $M_{i_0, \ldots, i_d}$ and $M_{j_0, \ldots, j_d}$ such that $T \oplus M_{i_0, \ldots, i_d}$ is a tilting $A_n^d$-module and $T \oplus M_{j_0, \ldots, j_d}$ is a different rigid $A_n^d$-module, and they both lie in $\add M$. Moreover we assume that we are in the first case of Theorem~\ref{theo.tilt-exchange}(2), that is we assume $\Ext_{A_n^d}^d(M_{j_0, \ldots, j_d}, M_{i_0, \ldots, i_d}) \neq 0$. By Theorem~\ref{theo.indexing.props}(4) that means $(i_0, \ldots, i_d) \wr (j_0, \ldots, j_d)$, and hence by Propositions~\ref{prop.E_exact} and \ref{prop.E_nonsplit} we have a non-spilt $d$-extension $\mathbb{E}$ with middle terms as claimed in Theorem~\ref{theo.tilt-exchange}(4a).

Since $\widehat{T} := T \oplus M_{i_0, \ldots, i_d}$ is a tilting $A_n^d$-module, and $\Ext_{A_n^d}^i(\widehat{T}, M_{j_0, \ldots, j_d}) = 0 \, \forall i \geq 0$ there is an exact minimal $\widehat{T}$-resolution of $M_{j_0, \ldots, j_d}$:
\[ \widehat{\mathbb{T}} \colon \widehat{T}_{d+1} \mono[30] \widehat{T}_d \to[30] \cdots \to[30] \widehat{T}_1 \epi[30] M_{j_0, \ldots, j_d} \]
Our aim now is to show that $\widehat{\mathbb{T}} \iso \mathbb{E}$.

\begin{observation} \label{obs.Td_is_Mi}
Applying $\Hom_{A_n^d}(-, \widehat{T}_{d+1})$ to the exact sequence $\widehat{\mathbb{T}}$ we obtain an exact sequence
\begin{align*}
& \Hom_{A_n^d}(M_{j_0, \ldots, j_d}, \widehat{T}_{d+1}) \mono[30] \Hom_{A_n^d}(\widehat{T}_1, \widehat{T}_{d+1}) \to[30] \cdots \\
& \qquad \cdots \to[30] \Hom_{A_n^d}(\widehat{T}_d, \widehat{T}_{d+1}) \to[30] \Hom_{A_n^d}(\widehat{T}_{d+1}, \widehat{T}_{d+1}) \epi[30] \Ext_{A_n^d}^d(M_{j_0, \ldots, j_d}, \widehat{T}_{d+1})
\end{align*}
(here we use $M_{j_0, \ldots, j_d} \oplus \widehat{T}_{d+1} \in \add M$, and hence $\Ext_{A_n^d}^i(M_{j_0, \ldots, j_d}, \widehat{T}_{d+1}) = 0$ for $1 \leq i \leq d-1$). Since, by the minimality of $\widehat{\mathbb{T}}$ no non-radical maps in $\Hom_{A_n^d}(\widehat{T}_{d+1}, \widehat{T}_{d+1})$ factor through $\widehat{T}_d$, such maps are mapped to non-zero extensions in $\Ext_{A_n^d}^d(M_{j_0, \ldots, j_d}, \widehat{T}_{d+1})$. It follows that $\widehat{T}_{d+1} = M_{i_0, \ldots, i_d}$, since this is the only summand of $\widehat{T}$ which admits non-zero extensions with $M_{j_0, \ldots, j_d}$.
\end{observation}

\begin{lemma} \label{lemma.perp-approx}
Let
\begin{align*}
\perp = \add \{ M_{\ell_0, \ldots, \ell_d} \mid & \Ext_{A_n^d}^d(M_{i_0, \ldots, i_d} \oplus M_{j_0, \ldots, j_d}, M_{\ell_0, \ldots, \ell_d}) = 0 \\
& \text{ and } \Ext_{A_n^d}^d(M_{\ell_0, \ldots, \ell_d}, M_{i_0, \ldots, i_d} \oplus M_{j_0, \ldots, j_d}) = 0 \}.
\end{align*}
Then $E_1 \epi M_{j_0, \ldots, j_d}$ is a minimal right $\perp$-approximation of $M_{j_0, \ldots, j_d}$, and $M_{i_0, \ldots, i_d} \mono E_d$ is a minimal left $\perp$-approximation of $M_{i_0, \ldots, i_d}$.
\end{lemma}

\begin{proof}
We only prove the first claim; the proof of the second is similar. By Theorem~\ref{theo.indexing.props}(3) there are only maps $M_{\ell_0, \ldots, \ell_d} \to M_{j_0, \ldots, j_d}$ if $\ell_x \leq j_x \, \forall x$. By our general assumption $\Ext_{A_n^d}^d(M_{j_0, \ldots, j_d}, M_{i_0, \ldots, i_d}) \neq 0$, so we have $(i_0, \ldots, i_d) \wr (j_0, \ldots, j_d)$, that is
\[ i_0 < j_0 < i_1 < j_1 < \cdots < i_d < j_d. \]
Since we want to approximate by those $M_{\ell_0, \ldots, \ell_d}$ that satisfy $\Ext_{A_n^d}^d(M_{\ell_0, \ldots, \ell_d}, M_{i_0, \ldots, i_d}) = 0$, that is $(i_0, \ldots, i_d) \nwr (\ell_0, \ldots, \ell_d)$, at least one of the above inequalities fails after replacing each $j_x$ with the corresponding $\ell_x$. That is, $\exists x \colon i_x \not< \ell_x$. In other words $\ell_x \leq i_x$. Hence we obtain the approximation by summing up all possibilities of replacing one index $j_x$ by the corresponding $i_x$.
\end{proof}

\begin{proposition}
With the above notation we have $\widehat{\mathbb{T}} \iso \mathbb{E}$.
\end{proposition}

\begin{proof}
We consider the following diagram.
\[ \begin{tikzpicture}[xscale=2.5,yscale=1.5]
 \node at (-.5,1) {$\widehat{\mathbb{T}}$:};
 \node (T1) at (0,1) {$\widehat{T}_{d+1}$};
 \node (T2) at (1,1) {$\widehat{T}_d$};
 \node (T3) at (2,1) {$\cdots$};
 \node (T4) at (3,1) {$\widehat{T}_2$};
 \node (T5) at (4,1) {$\widehat{T}_1$};
 \node (T6) at (5,1) {$M_{j_0, \ldots, j_d}$};
 \node at (-.5,0) {$\mathbb{E}$:};
 \node (E1) at (0,0) {$E_{d+1}$};
 \node (E2) at (1,0) {$E_d$};
 \node (E3) at (2,0) {$\cdots$};
 \node (E4) at (3,0) {$E_2$};
 \node (E5) at (4,0) {$E_1$};
 \node (E6) at (5,0) {$E_0$};
 \draw [>->] (T1) -- (T2);
 \draw [->] (T2) -- (T3);
 \draw [->] (T3) -- (T4);
 \draw [->] (T4) -- (T5);
 \draw [->>] (T5) -- (T6);
 \draw [>->] (E1) -- (E2);
 \draw [->] (E2) -- (E3);
 \draw [->] (E3) -- (E4);
 \draw [->] (E4) -- (E5);
 \draw [->>] (E5) -- (E6);
 \draw [->, dashed] (T1) -- node [left] {$f_{d+1}$} (E1);
 \draw [->, dashed] (T2) -- node [left] {$f_d$} (E2);
 \draw [->, dashed] (T4) -- node [left] {$f_2$} (E4);
 \draw [->] (T5) -- node [left] {$f_1$} (E5);
 \draw [double distance=1.5pt] (T6) -- (E6);
\end{tikzpicture} \]
Here the map $f_1$ making the square to its right commutative exists by Lemma~\ref{lemma.perp-approx}, since $\widehat{T}_1 \in \perp$. Then the dashed maps can be constructed inductively from right to left since all the objects are in $\add M$ (as in Observation~\ref{obs.Td_is_Mi} we see that $\Hom_{A_n^d}(\widehat{T}_{i+1}, E_{i+1}) \to \Hom_{A_n^d}(\widehat{T}_{i+1}, E_i) \to \Hom_{A_n^d}(\widehat{T}_{i+1}, E_{i-1})$ is exact). Since $\mathbb{E}$ is non-split and $\widehat{T}_{d+1} \iso E_{d+1} = M_{i_0, \ldots, i_d}$ we know that $f_{d+1}$ is an isomorphism. Since, by Lemma~\ref{lemma.perp-approx} the map $E_{d+1} \mono E_d$ is a left $\perp$-approximation, and $\widehat{T}_d \in \add T \subseteq \perp$ it follows that $f_d$ is an isomorphism. 
The rest of $\widehat{\mathbb{T}}$ as well as $\mathbb{E}$ are minimal $M$-coresolutions of the cokernels of their left-most maps, respectively. Thus one can see iterately from left to right that all the vertical maps are isomorphisms.
\end{proof}

This also completes the proof of Theorem~\ref{theo.tilt-exchange}(4). Part (5) of Theorem~\ref{theo.tilt-exchange} is dual to part (4). Hence we have also completed the proof of Theorem~\ref{theo.tilt-exchange}.

\section{Local moves}\label{local}
There is an operation on triangulations called a \emph{bistellar flip}.
Given $e+2$ points in $\mathbb R^e$, 
no $e+1$ lying in any hyperplane,
there are two triangulations of
their convex hull.  \
A bistellar flip of a triangulation $T$ 
is given by specifying some $e+2$ vertices of the triangulation,
no $e+1$ lying in any hyperplane, 
such that $T$ restricts
to a triangulation of the convex hull $X$ 
of those $e+2$ vertices.  That bistellar
flip of $T$ is then obtained by replacing the part of $T$ inside $X$ by
the other triangulation using those vertices.  
(It is possible to weaken the assumption that no $e+1$ of the
vertices involved in the bistellar flip lie on a hyperplane, but we shall
not need to make use of  that here.)

\begin{theorem}\label{th3} For $S,T\in S(m,2d)$, we have that $S$ and $T$ are 
related
by a bistellar flip iff
 $e(S)$ and $e(T)$ have all but one (d+1)-tuple in 
common. \end{theorem}

\begin{example}
When $d=1$ and the vertices are in convex position, 
a bistellar flip amounts to replacing one diagonal of
a quadrilateral with the other diagonal.  The vertices of $C(m,2)$ are 
always in convex position, so bistellar flips always amount to 
replacing one diagonal of a quadrilateral by the other one; clearly, 
if $S$ and $T$ are related in this way, then $e(S)$ and $e(T)$ differ
by one element, namely, the diagonal being flipped.
\end{example}

We recall the following result of Rambau:

\begin{theorem}[{\cite{R}}]\label{rambau} Any triangulation of a cyclic polytope can be transformed
into any other one by a sequence of bistellar flips.  
\end{theorem}

From this, we deduce the following central result of our paper:

\begin{theorem}\label{central} Triangulations of $C(n+2d,2d)$ correspond 
bijectively to
basic tilting modules for $A^d_n$ contained in $\leftsub{A^d_n}{M}$; two triangulations are related by
a bistellar flip iff the corresponding tilting objects are related by a single
mutation.  
\end{theorem}

If $A\in \Index{m}$ 
and 
$R$ is a non-intertwining subset of $\Index{m}$ not containing $A$, 
we say that $A$ is a complement for
$R$ if $R \cup \{A\}$ corresponds to a triangulation (that is to say,
it is non-intertwining and has cardinality $m-d-1\choose d$).  

If $A$ and $B$ are distinct complements to some $R$, they are called
\emph{exchangeable}.  

\begin{proposition}\label{exiffint} $A$ and $B$ are exchangeable iff they 
intertwine
(in some order). \end{proposition}

\begin{proposition}\label{mixes}
Let $A$ and $B$ be exchangeable.  $A$ and $B$ are complements to 
$R\subset \Index{m}$ iff
$R$ is a non-intertwining subset of $\Index{m}\setminus\{A,B\}$ 
with cardinality ${m-d-1\choose d}-1$, 
which 
contains every separated $(d+1)$-tuple from $A\cup B$ other than
$A$ and $B$.
\end{proposition}

\begin{remark} Suppose $A$ and $B$ are distinct complements to 
$R\subset \Index{n+2d}$.  
By Theorems~\ref{theo.tilt-exchange} and~\ref{central}, 
the indecomposable
$A^d_{n}$-modules $M_A$ and $M_B$ are related by an exact sequence.  
All the indecomposable summands of the terms of this exact sequence correspond
to certain $(d+1)$-tuples which are contained in $R$.  This implies a weaker version
of one direction of the previous proposition.
\end{remark}

\begin{example} In the $d=1$ case, it is clear that $A$ and $B$ are
exchangeable iff they cross in their interiors (as line segments) 
iff they intertwine in some order (as increasing ordered pairs from $\{1,\dots,m\}$).  
In this case, the triangulations in which $A$ can be exchanged
for $B$ are exactly those containing $A$ and the four edges of the quadrilateral
defined by the vertices of $A$ and $B$.
\end{example}

\subsection{Proofs for Section~\texorpdfstring{\ref{local}}{4}}

\begin
{proof}
[Proof of Proposition \ref{exiffint}]
Suppose that $A$ and $B$ are exchangeable; in other words, there
is some $R\subset \Index{m}$, such that $A$ and $B$ are complements
to $R$.

Since $\{A\}\cup \{B\} \cup R$ is too big to be a non-intertwining set of $d$-faces, but $\{A \} \cup R$ and $\{B\} \cup R$ are both non-intertwining, it must be that $A$ and 
$B$ intertwine in some order.  
This proves the first direction.  

Conversely, suppose that $A$ and $B$ intertwine.  We must show that 
there is some $R\subset \Index{m}$ such that 
$R\cup \{A\}$ and 
$R\cup \{B\}$ correspond to triangulations.

The
convex hull of $A\cup B$ is a cyclic polytope with $2d+2$ vertices.
By Lemma~\ref{dfaceclass}, it has exactly two internal $d$-simplices,
$A$ and $B$. A triangulation of $C(2d+2,2d)$ must use
exactly one of these internal $d$-simplices.  The two triangulations of
$C(2d+2,2d)$ are related by a bistellar flip.  
Therefore, it suffices to show that there is a triangulation $S$ of 
$C(m,2d)$
which restricts to a triangulation of the convex hull of $A\cup B$.  

Given a set of points $X$ in $\mathbb R^{e}$, one way to construct a 
triangulation of their convex hull $Q$ is to define a height function $f \colon X\rightarrow
\mathbb R$, and then to consider the points in $\mathbb R^{e+1}$ of 
the form $\hat x=(x,f(x))$.  Let $\widehat Q$ be the convex hull of the points
$\hat x$.  Take the lower facets of $\widehat Q$ and project them onto $Q$.  
This defines a subdivision of $Q$.  If $f$ is chosen generically, this
subdivision will be a triangulation.  Triangulations arising in this way
are called \emph{regular}.

Define a height function for the $m$ vertices of $C(m,2d)$, such that 
the vertices from $A\cup B$ are much lower than the
other vertices.  The above construction will result in a triangulation
which restricts to a triangulation of the convex hull of $A\cup B$.
\end{proof}

We recall an important property of cyclic polytopes:

\begin{theorem}[{\cite[Proposition VI.4.2]{bar}}]
\label{neighbourly} In $C(m,\delta)$, any set of $\lfloor \frac {\delta}2 \rfloor$ or
fewer vertices form the vertices of a boundary face.  \end{theorem}

\begin{proof}[Proof of Theorem \ref{th3}]
If triangulations $S$ and $T$ of $C(m,2d)$ 
are related by a bistellar flip, then they coincide except inside the
convex hull of some $2d+2$ vertices.  We may think of this convex hull as a 
copy of $C(2d+2,2d)$.  This cyclic polytope 
has exactly two internal $d$-simplices.  
$S$ uses one and $T$ uses
the other.  The first direction of the theorem follows.  

For the other direction, suppose that we have triangulations $S$ and $T$
with $e(S)=R\cup\{A\}$ and $e(T)=R\cup \{B\}$.
Let $A=(a_0,\dots,a_d)$, 
$B=(b_0,\dots,b_d)$.

Let $Q$ be the region in $C(m,2d)$ 
formed by the union of the $2d$-simplices of $S$ 
containing
$A$.  This is also the 
union of the $2d$-simplices of $T$
containing $B$, since the remaining $2d$-simplices of $S$ and $T$ coincide by Lemma~\ref{deylem}.

Let $A_i=(a_0,\dots,\widehat a_i,\dots,a_d)$, that is
to say, $A$ with $a_i$ removed.  For each $i$, note that $A_i$ lies on
the boundary of $Q$, since, by Theorem~\ref{neighbourly}, it lies on
the boundary of $C(m,2d)$.  Thus, there must be some $2d$-simplex in $T|_Q$
which 
contains $A_i$.   This simplex contains the vertices $A_i\cup B$, which
amount to $2d+1$ vertices.  Thus this is the complete list of vertices in the
simplex.  So $T|_Q$ contains the simplices $A_i\cup B$ for all $i$.

These simplices form one of the two triangulations of the convex hull of 
$A\cup B$.  Thus $Q$ is the convex hull of $A\cup B$, and all the 
boundary $d$-faces of the convex hull of $A \cup B$ which are not
boundary faces of $C(m,2d)$ are in $R \cup \{B\}$. Since $B$ is not
a boundary face of $Q$,  the boundary $d$-faces of $Q$ are contained in
$R$, and therefore are also contained in $e(S)=R\cup\{A\}$.
Thus $S|_Q$ is the other triangulation of 
the convex hull of $A\cup B$, and $S$ and $T$ are related by a bistellar
flip, as desired.  
\end{proof}

\begin{proof}[Proof of Theorem \ref{central}]
Consider $A^d_n$ as an $A^d_n$-module.  It is the direct sum of 
$n+d-1\choose d$ indecomposable modules of $A^d_n$, so it corresponds
to a $n+d-1\choose d$-subset of $\Index{n+2d}$. This subset equals $e(S)$
for some triangulation $S$ of $C(n+2d,2d)$.  Let $T$ be a triangulation obtained by applying
a bistellar flip to $S$.  By Theorem \ref{th3}, $e(T)$ is an  
$n+d-1\choose d$-subset of $\Index{n+2d}$ which coincides with $e(S)$ except
that one element of $e(S)$ has been replaced by an element of $e(T)$.  Let
$X$ be the $A^d_n$-module corresponding to $e(T)$. 

Now note that $A^d_n$ is a tilting module.    
Since $\Ext_{A_n^d}^d(X,X)=0$, and $X$ is obtained by replacing one 
indecomposable summand of $A_n^d$ by another indecomposable summand of $\leftsub{A_n^d}{M}$, Theorem~\ref{theo.tilt-exchange} applies to tell us that $X$ is also 
a tilting module.   

The same argument can be iterated to show that the module 
corresponding to any triangulation which can be obtained
by a sequence of bistellar flips starting from the triangulation
corresponding to $A^d_n$, is a tilting module.  Theorem
\ref{rambau} tells us that any triangulation can be obtained
by a sequence of bistellar flips starting from any fixed triangulation,
so we are done.  

The second statement (relating bistellar flip to mutation) is immediate from the above discussion.
\end{proof}

\begin
{proof} 
[Proof of Proposition \ref{mixes}]
Suppose first that $A$ and $B$ are complements to $R$.  This implies
that there are triangulations $S$ and $T$ with $e(S)=R\cup \{A\}$ and 
$e(T)=R\cup\{B\}$.  By Theorem~\ref{th3}, $S$ and $T$ are related by
a bistellar flip.  Therefore $S$ and $T$ restrict to triangulations of
the convex hull of $A\cup B$.  
In particular, this implies that $R$ must contain all the separated
$(d+1)$-tuples which correspond to $d$-faces of the convex hull of $A\cup B$.

Conversely, suppose that $R$ is a non-intertwining subset of 
$\Index{m} \setminus \{A,B\}$, with cardinality ${m-d-1\choose d}-1$, and which 
contains all the separated 
$(d+1)$-tuples from $A\cup B$ other than $A$ and $B$.  Without loss of 
generality, let $A\wr B$, and let $A=(a_0,\dots,a_d), B=(b_0,\dots,b_d)$.    

We wish to show that $R\cup \{A\}$ and $R\cup \{B\}$ are non-intertwining.  
Suppose there is some $C \in R$ such that $C\wr A$.  Then 
$C\wr (a_0,\dots,a_{d-1},b_d)\in R$, which is contrary to our assumption.  

Suppose there is some $C =(c_0,\dots,c_d)\in R$ such that $A\wr C$.  If
$c_i>b_i$ for some $i$ then $R\ni(a_0,\dots,a_{i-1}, b_i,a_{i+1},\dots,a_d)\wr C$, which is contrary to our assumption.  
(Note that
the fact that $(a_0,\dots,a_{i-1},b_i,a_{i+1},\dots,a_d)\wr C$ implies that
$(a_0,\dots,a_{i-1},b_i,a_{i+1},\dots,a_d)$ is separated, without which we would not know
that it is in $R$.)
So $a_i<c_i\leq b_i$ for all $i$.  
Note that there is therefore some index $t$ for which $c_t<b_t$, since 
$C\ne B$.  
But now 
$(a_0,a_1,\dots,a_t,b_t,a_{t+2},\dots,a_d)\wr C$, which is contrary to our assumption.  Thus $R\cup\{A\}$ is non-intertwining, as
desired.  The same result for $R\cup\{B\}$ follows similarly.
\end{proof}

\pagestyle{headings}

\section{The \texorpdfstring{$(d+2)$}{(d+2)}-angulated cluster category} \label{section.cluster_cat}

In this section we generalize the construction of a (triangulated) cluster category for a hereditary representation finite algebra $\Lambda$. More precisely, we will construct a $(d+2)$-angulated cluster category for a $d$-representation finite algebra $\Lambda$. The $(d-1)$-st higher Auslander algebra of linear oriented $A_n$ is one example of such a $d$-representation finite algebra (which will be studied in greater detail in Section~\ref{section.cluster_cat_A}). However the construction and all the results presented in this section hold for $d$-representation finite algebras in general, and hence we obtain a generalization of arbitrary cluster categories of hereditary representation finite algebras (see \cite{BMRRT}), not just a generalization of cluster categories of type $A$.

Recall (see Definition~\ref{def.d-repfin}) that an algebra is called $d$-representation finite if it has global dimension at most $d$ and its module category contains a $d$-cluster tilting module. Throughout this section we will assume $\Lambda$ to be $d$-representation finite, and we will denote the basic $d$-cluster tilting module in $\mod \Lambda$ by $M$.

It should be noted that in this case the functors $\tau_d$ and $\tau_d^-$ on $\add M$ behave very similarly to the way the usual Auslander-Reiten translations $\tau$ and $\tau^-$ behave on the module category of a hereditary representation finite algebra. For details see \cite{Iy_n-Auslander,IO}.

We want to mimic the ``usual'' construction of the cluster category, but use $\add M$ instead of all of $\mod \Lambda$. That is, on objects we want our cluster category to be
\[ \add M \vee \proj \Lambda[d].\]

\begin{remark}
In what follows we will be talking about a standard $(d+2)$-angulated category $\mathscr{O}$. For a definition of what we mean precisely 
by a standard $(d+2)$-angulated category, see Definition~\ref{def.angulated}. For now, the reader may think of a category together with
\begin{itemize}
\item an autoequivalence called the $d$-suspension, denoted by $[d]$, and
\item a collection of sequences of $d+2$ morphisms, starting in some object and ending in its $d$-suspension, called $(d+2)$-angles.
\end{itemize}
\end{remark}

The first main result of this section is the existence of a $(d+2)$-angulated cluster category with the desired properties. See \cite{K_orbit} for the triangulated case.

\begin{theorem} \label{theo.clustercat_exists}
There is a standard $(d+2)$-angulated Krull-Schmidt $k$-category $\mathscr{O}_{\Lambda}$ such that
\begin{enumerate}
\item The isomorphism classes of indecomposable objects in $\mathscr{O}_{\Lambda}$ are in bijection with the indecomposable direct summands of $M \oplus \Lambda[d]$.
\item $\mathscr{O}_{\Lambda}$ is $2d$-Calabi-Yau, that is we have a natural isomorphism
\[ \Hom_{\mathscr{O}_{\Lambda}}(X, Y) = D \Hom_{\mathscr{O}_{\Lambda}}(Y, X[2d]) \]
for $X, Y \in \mathscr{O}_{\Lambda}$. It should be noted that $[2d]$ here is the square of the $d$-suspension, and not the $2d$-th power of a $1$-suspension (in fact, there is no $1$-suspension).
\item When identifying along the bijection in (1), we have
\[ \Hom_{\mathscr{O}_{\Lambda}}(X, Y[d]) = \Hom_{D^b(\mod \Lambda)}(X, Y[d]) \oplus D \Hom_{D^b(\mod \Lambda)}(Y, X[d]) \]
for any $X, Y \in \add (M \oplus \Lambda[d])$. In particular $\Hom_{\mathscr{O}_{\Lambda}}(X, Y[d]) = 0$ if and only if both $\Hom_{D^b(\mod \Lambda)}(X, Y[d]) = 0$ and $\Hom_{D^b(\mod \Lambda)}(Y, X[d]) = 0$.
\end{enumerate}
\end{theorem}

We have the following definition of cluster tilting objects in $\mathscr{O}_{\Lambda}$, which generalizes the classical definition:

\begin{definition} \label{def.cto}
An object $T \in \mathscr{O}_{\Lambda}$ is called \emph{cluster tilting} if
\begin{enumerate}
\item $\Hom_{\mathscr{O}_{\Lambda}}(T, T[d]) = 0$, and
\item Any $X \in \mathscr{O}_{\Lambda}$ occurs in a $(d+2)$-angle
\[ X[-d] \to[30] T_d \to[30] T_{d-1} \to[30] \cdots \to[30] T_1 \to[30] T_0 \to[30] X \]
with $T_i \in \add T$.
\end{enumerate}
\end{definition}

\begin{remark}
For $d=1$ the above is not the original definition of a cluster tilting object (see \cite{BMRRT}), but our definition is equivalent to the original definition (see \cite{BMR_cta}). Also see the discussion following Proposition~\ref{prop:higherd} on why we use this definition.
\end{remark}

The following nice connection between tilting modules in $\add M$ and cluster tilting objects in $\mathscr{O}_{\Lambda}$ remains true in this setup (see \cite{BMRRT} and \cite{ABS_trivext} for the triangulated case):

\begin{theorem} \label{theo.cta_is_trivext}
Let $T \in \add M$ be a tilting $\Lambda$-module. Then $T$ is a cluster tilting object in $\mathscr{O}_{\Lambda}$.

Moreover, if we set $\Gamma := \End_{\Lambda}(T)$, then we have
\[ \End_{\mathscr{O}_{\Lambda}}(T) = \Gamma \ltimes \Ext_{\Gamma}^{2d}(D\Gamma, \Gamma). \]
\end{theorem}

Finally, we will see that we have the following connection from the $(d+2)$-angulated cluster category to the module category of a cluster tilted algebra.  The triangulated cluster category case of this result has been proven in \cite{BMR_cta}.

\begin{theorem} \label{theo.induces_ct_in_mod}
Let $T$ be a cluster tilting object in $\mathscr{O}_{\Lambda}$, and set $\Gamma := \End_{\mathscr{O}_{\Lambda}}(T)$. Then the functor
\[ \Hom_{\mathscr{O}_{\Lambda}}(T, -) \colon \mathscr{O}_{\Lambda} \to[30] \mod \Gamma \]
induces a full faithful embedding
\[ \frac{\mathscr{O}_{\Lambda}}{(T[d])} \mono[30] \mod \Gamma, \]
where $(T[d])$ denotes the ideal of all morphisms factoring through $\add T[d]$. The image of this functor (which is then equivalent to $\frac{\mathscr{O}_{\Lambda}}{(T[d])}$) is a $d$-cluster tilting subcategory of $\mod \Gamma$.

In particular $\Gamma$ is weakly $d$-representation finite in the sense of \cite{IO}, that is, it has a $d$-cluster tilting object in its module category.
\end{theorem}

Finally we have the following description of exchange $(d+2)$-angles. (See \cite{BMRRT} for the exchange triangles in the case when $d=1$.)

\begin{theorem} \label{theo.exch_angles}
Let $T$ be a basic cluster tilting object in $\mathscr{O}_{\Lambda}$, $T_0$ an indecomposable direct summand of $T$, and $R$ the sum of the remaining summands of $T$. Let $T_0 \to R^1$ be a minimal left $R$-approximation of $T_0$, and $R_1 \to T_0$ be a minimal right $R$-approximation. We consider their minimal completions to $(d+2)$-angles (see Lemma \ref{lemma.angle_completion})
\begin{align*}
& T_0 \to[30] R^1 \to[30] Y^2 \to[30] \cdots \to[30] Y^{d+1} \to[30] T_0[d] \text{, and} \\
& T_0[-d] \to[30] X_{d+1} \to[30] \cdots \to[30] X_2 \to[30] R_1 \to[30] T_0 \text{, respectively.}
\end{align*}
Then the following are equivalent:
\begin{enumerate}
\item There is an indecomposable $T_0^*$, with $T_0^* \not\iso T_0$ such that $T_0^* \oplus R$ is cluster tilting in $\mathscr{O}_{\Lambda}$.
\item There is an indecomposable $T_0^*$, with $T_0^* \not\in \add T$ such that $\Hom_{\mathscr{O}_{\Lambda}}(R, T_0^*[d]) = 0$.
\item $Y^i \in \add R \, \forall i \in \{2, \ldots, d\}$.
\item $X_i \in \add R \, \forall i \in \{2, \ldots, d\}$.
\end{enumerate}
Moreover, in this case we have $T_0^* \iso Y^{d+1} \iso X_{d+1}$.
\end{theorem}

\begin{remark}
Note that in the classical cluster category ($d=1$), if $T$ is a basic cluster tilting object and $T_0$ is an indecomposable summand of it, then there is always a $T_0^*$ satisfying the conditions of Theorem~\ref{theo.exch_angles} since (3) and (4) are vacuous.

For $d>1$ this is not the case; in other words, there are typically summands of a cluster tilting object which cannot be mutated.
\end{remark}

\subsection{Background and notation}

For triangulated categories, we will denote the suspension by $[1]$, and its powers by $[i] = [1]^i$. Moreover, for subcategories $\mathscr{A}$ and $\mathscr{B}$ we denote by $\mathscr{A} * \mathscr{B}$ the full subcategory of extensions of an object in $\mathscr{A}$ and an object in $\mathscr{B}$, that is
\[ \mathscr{A} * \mathscr{B} = \{ X \mid \exists A \to[30] X \to[30] B \to[30] A[1] \text{ with } A \in \mathscr{A}, B \in \mathscr{B} \}. \]

\begin{notation}
Let $\mathscr{T}$ be a triangulated category. If $\mathscr{T}$ has a Serre functor then it will be denoted by $\leftsub{\mathscr{T}}{\mathbb{S}}$ or just $\mathbb{S}$. Its $\delta$-th desuspension will be denoted by $\mathbb{S}_{\delta} = \mathbb{S}[-\delta]$. Finally we denote the inverses of these equivalences by $\mathbb{S}^-$ and $\mathbb{S}_{\delta}^-$, respectively.
\end{notation}

\begin{definition}
Let $\mathscr{T}$ be a triangulated category, $\mathscr{X}$ a full subcategory, and $\delta \in \mathbb{N}$. Then $\mathscr{X}$ is \emph{$\delta$-cluster tilting} if
\begin{itemize}
\item $\mathscr{X}$ is functorially finite in $\mathscr{T}$, and
\item the subcategory $\mathscr{X}$ of $\mathscr{T}$ coincides with both
\begin{align*}
& \{T \in \mathscr{T} \mid \Hom_{\mathscr{T}}(\mathscr{X}, T[i]) = 0 \, \forall i \in \{1, \ldots, \delta-1\}\}, \text{ and} \\
& \{T \in \mathscr{T} \mid \Hom_{\mathscr{T}}(T, \mathscr{X}[i]) = 0 \, \forall i \in \{1, \ldots, \delta-1\}\}.
\end{align*}
\end{itemize}
If $\add X$ is a $\delta$-cluster tilting subcategory of $\mathscr{T}$, then we call $X$ a \emph{$\delta$-cluster tilting object}.
\end{definition}

\subsubsection{\texorpdfstring{$d$}{d}-representation finiteness in the derived category}

\begin{theorem}[Iyama \cite{Iy_n-Auslander}] \label{theorem.ct_in_Db}
Let $\Lambda$ be $d$-representation finite, with $d$-cluster tilting object $M$. Then in $D^b(\mod \Lambda)$ we have
\[ \mathscr{U} = \add \{ \mathbb{S}_d^i \Lambda \mid i \in \mathbb{Z}\} = \add \{ M[id] \mid i \in \mathbb{Z} \} \]
is a $d$-cluster tilting subcategory.
\end{theorem}

\begin{remark} \label{rem.tau=S}
It should be noted that $\mathbb{S}_d$ and $\mathbb{S}_d^-$ are just the $\mathscr{U}$-versions of $\tau_d$ and $\tau_d^-$. More precisely, for $M' \in \add M$ indecomposable we have
\begin{align*}
\mathbb{S}_d M' & = \tau_d M' \quad \text{ if } \tau_d M' \neq 0 \text{, and} \\
\mathbb{S}_d^- M' & = \tau_d^- M' \quad \text{ if } \tau_d^- M' \neq 0.
\end{align*}
\end{remark}

\subsubsection{\texorpdfstring{$\delta$}{delta}-Amiot cluster categories}

\begin{construction}[Amiot \cite{CC, C_PhD}]
Let $\Lambda$ be an algebra with $\gld \Lambda \leq \delta$. Then we denote by
\[ \mathscr{C}_{\Lambda}^{\delta} = \text{triangulated hull}(\underbrace{D^b(\mod \Lambda) / (\mathbb{S}_{\delta})}_{\text{orbit category}}) \]
the $\delta$-Amiot cluster category of $\Lambda$. We do not give a definition of triangulated hull. For the purpose of this paper, this is some triangulated category containing $D^b(\mod \Lambda) / (\mathbb{S}_{\delta})$, such that the following theorem holds.
\end{construction}

\begin{theorem}[Amiot \cite{CC, C_PhD}] \phantomsection \label{theorem.amiot}
\begin{enumerate}
\item If $\gld \Lambda < \delta$, then $\mathscr{C}_{\Lambda}^{\delta}$ is $\Hom$-finite.
\item If $\mathscr{C}_{\Lambda}^{\delta}$ is $\Hom$-finite then it is $\delta$-Calabi-Yau.
\item If $\mathscr{C}_{\Lambda}^{\delta}$ is $\Hom$-finite, and $T \in D^b(\mod \Lambda)$ is a tilting complex with $\gld \End(T) \leq \delta$, then $\overline{T} \in \mathscr{C}_{\Lambda}^{\delta}$ is a $\delta$-cluster tilting object.
\end{enumerate}
\end{theorem}

\subsection{Standard \texorpdfstring{$(d+2)$}{(d+2)}-angulated categories}
${n}$-angulated categories have been introduced in \cite{4ang}. In particular \cite{4ang} gives a construction motivating our definition here.

Before we start, we need the following bit of notation: In diagrams, if an arrow is labeled by an integer $i$, then it denotes a morphism to the $i$-th suspension of the target. For instance
$\begin{tikzpicture}[baseline=(A.base)]
 \node (A) at (0,0) {$X$};
 \node (B) at (1.5,0) {$Y$};
 \draw [->] (A) -- node [pos=.4,fill=white,inner sep=1pt] {$\scriptstyle 1$} (B);
\end{tikzpicture}$
means the same thing as $X \to[30] Y[1]$.

\begin{definition} \label{def.angulated}
A $\Hom$-finite $k$-category $\mathscr{O}$, together with a $d$-suspension $[d]_{\mathscr{O}}$ and a collection of distinguished $(d+2)$-angles, is called \emph{standard $(d+2)$-angulated} if there is a full faithful $k$-embedding $\mathscr{O} \sub \mathscr{T}$ into a triangulated $k$-category $\mathscr{T}$ with Serre functor $\leftsub{\mathscr{T}}{\mathbb{S}}$ such that
\begin{enumerate}
\item $\mathscr{O}$ is a $d$-cluster tilting subcategory of $\mathscr{T}$, such that $\leftsub{\mathscr{T}}{\mathbb{S}} \mathscr{O} = \mathscr{O}$.
\item $[d]_{\mathscr{O}} = [d]_{\mathscr{T}}$.
\item A sequence of $(d+2)$ morphisms
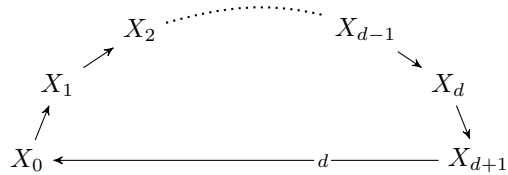
\[ \begin{tikzpicture}[xscale=1.5]
 \node (A) at (180:2) {$X_0$};
 \node (B) at (150:2) {$X_1$};
 \node (C) at (120:2) {$X_2$};
 \node (D) at (60:2) {$X_{d-1}$};
 \node (E) at (30:2) {$X_d$};
 \node (F) at (0:2) {$X_{d+1}$};
 \draw [->] (A) -- (B);
 \draw [->] (B) -- (C);
 \draw [dotted,thick,out=20,in=160] (C) to (D);
 \draw [->] (D) -- (E);
 \draw [->] (E) -- (F);
 \draw [->] (F) -- node [pos=.3,fill=white,inner sep=1pt] {$\scriptstyle d$} (A);
\end{tikzpicture} \]
in $\mathscr{O}$ is a distinguished $(d+2)$-angle if and only if this diagram can be completed to a diagram
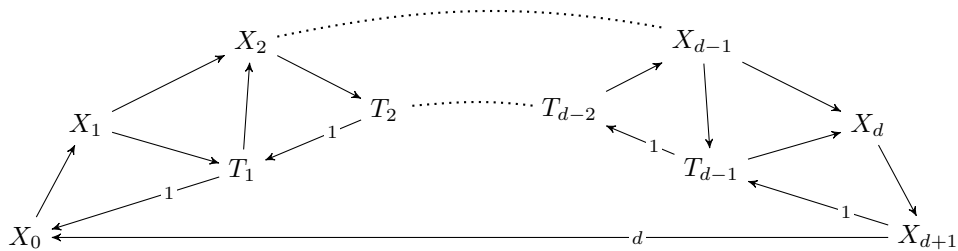
\[ \begin{tikzpicture}[xscale=3,yscale=1.5]
 \node (A) at (180:2) {$X_0$};
 \node (B) at (150:2) {$X_1$};
 \node (C) at (120:2) {$X_2$};
 \node (D) at (60:2) {$X_{d-1}$};
 \node (E) at (30:2) {$X_d$};
 \node (F) at (0:2) {$X_{d+1}$};
 \draw [->] (A) -- (B);
 \draw [->] (B) -- (C);
 \draw [dotted,thick,out=20,in=160] (C) to (D);
 \draw [->] (D) -- (E);
 \draw [->] (E) -- (F);
 \draw [->] (F) -- node [pos=.3,fill=white,inner sep=1pt] {$\scriptstyle d$} (A);
 \node (L) at (150:1.2) {$T_1$};
 \node (M) at (110:1.2) {$T_2$};
 \node (N) at (70:1.2) {$T_{d-2}$};
 \node (O) at (30:1.2) {$T_{d-1}$};
 \draw [->] (F) -- node [pos=.3,fill=white,inner sep=1pt] {$\scriptstyle 1$} (O);
 \draw [->] (O) -- node [pos=.3,fill=white,inner sep=1pt] {$\scriptstyle 1$} (N);
 \draw [dotted,thick,out=170,in=10] (N) to (M);
 \draw [->] (M) -- node [pos=.3,fill=white,inner sep=1pt] {$\scriptstyle 1$} (L);
 \draw [->] (L) -- node [pos=.3,fill=white,inner sep=1pt] {$\scriptstyle 1$} (A);
 \draw [->] (B) -- (L);
 \draw [->] (L) -- (C);
 \draw [->] (C) -- (M);
 \draw [->] (N) -- (D);
 \draw [->] (D) -- (O);
 \draw [->] (O) -- (E);
\end{tikzpicture} \]
in $\mathscr{T}$, such that all oriented triangles are distinguished triangles, and all non-oriented triangles and the lower shape commute.
\end{enumerate}
\end{definition}

\begin{remark}
In \cite{4ang} a definition of ${n}$-angulated categories in terms of axioms similar to those for triangulated categories is given. Then a standard construction is given, which is shown to always yield ${n}$-angulated categories. Standard ${n}$-angulated categories in the sense of our Definition~\ref{def.angulated} are precisely those ${n}$-angulated categories in the sense of \cite{4ang} which result from this standard construction.
\end{remark}

\begin{example} \label{example.U_angulated}
Let $\Lambda$ be a $d$-representation finite algebra. Then the category $\mathscr{U}$ of Theorem~\ref{theorem.ct_in_Db} is standard $(d+2)$-angulated. Any exact sequence
\[ M_0 \mono[30] M_1 \to[30] \cdots \to[30] M_d \epi[30] M_{d+1} \]
in $\add M$ turns into a $(d+2)$-angle
\[ M_0 \to[30] M_1 \to[30] \cdots \to[30] M_d \to[30] M_{d+1} \to[30] M_0[d] \]
in $\mathscr{U}$. (This follows from the fact that any short exact sequence in $\mod \Lambda$ turns into a triangle in $D^b(\mod \Lambda)$.)
\end{example}

\begin{lemma} \label{lemma.angle_completion}
Let $\mathscr{O}$ be a standard $(d+2)$-angulated category.
\begin{enumerate}
\item Any map $f \colon X_0 \to X_1$ in $\mathscr{O}$ can be completed to a $(d+2)$-angle $X_0 \to X_1 \to \cdots \to X_{d+1} \to X_0[d]$.
\item If we additionally require the maps $X_2 \to X_3$, $X_3 \to X_4$, $\ldots$, $X_d \to X_{d+1}$ to be radical morphisms, then the $(d+2)$-angle is determined up to (non-unique) isomorphism by the map $f$. In this case we call the $(d+2)$-angle a \emph{minimal completion} of $f$ to a $(d+2)$-angle.
\end{enumerate}
\end{lemma}

\begin{proof}
Both claims follow from the fact that, given $f \colon X_0 \to X_1$, we can determine $T_1 \in \mathcal{T}$. Then we have to resolve $T_1$ by objects in the $d$-cluster tilting subcategory $\mathscr{O}$ of $\mathscr{T}$. By general theory of cluster tilting subcategories this is possible in $d-1$ steps in an essentially unique way.
\end{proof}

\begin{lemma} \label{lemma.angle_rotation}
The $(d+2)$-angles in a standard $(d+2)$-angulated category $\mathscr{O}$ are invariant under rotation. More precisely, if
\[ X_0 \tol[30]{\scriptstyle f_0} \cdots \tol[30]{\scriptstyle f_d} X_{d+1} \tol[30]{\scriptstyle f_{d+1}} X_0[d] \]
is a $(d+2)$-angle in $\mathscr{O}$, then so is 
\[ X_1 \tol[30]{\scriptstyle f_1} \cdots \tol[30]{\scriptstyle f_d} X_{d+1} \tol[30]{\scriptstyle f_{d+1}} X_0[d] \tol[50]{\scriptstyle (-1)^d f_0[d]} X_1[d]. \]
\end{lemma}

\begin{proof}
In the notation of Definition~\ref{def.angulated} we set $H_1 = \Cone [X_1 \to T_1 \to X_2] \in \mathscr{T}$. By the octahedral axiom (for $\mathscr{T}$) we have a triangle $H_1 \to T_2 \to X_0[2] \to H_1[1]$ in $\mathscr{T}$. We set $H_2 = \Cone [H_1 \to T_2 \to X_3]$, and, by the octahedral axiom we have a triangle $H_2 \to T_3 \to X_0[3] \to H_2[1]$. Iterating this we end up with a triangle $H_{d-1} \to X_{d+1} \to X_0[d] \to H_{d+1}[1]$. Putting these triangles together we obtain the desired $(d+2)$-angle.
\end{proof}

\begin{remark}
Let $\mathscr{O}$ be a standard $(d+2)$-angulated category, and $\mathscr{T}$ as in Definition~\ref{def.angulated}. Then the Serre functor $\leftsub{\mathscr{T}}{\mathbb{S}}$ of $\mathscr{T}$ restricts to a Serre functor $\leftsub{\mathscr{O}}{\mathbb{S}}$.
\end{remark}

\begin{lemma} \label{lemma.snake_seq}
Let $\mathscr{O}$ be a standard $(d+2)$-angulated category, and $Y \in \mathscr{O}$. Then a $(d+2)$-angle $X_0 \to \cdots \to X_{d+1} \to X_0[d]$ gives rise to a long exact sequence
\[ \begin{tikzpicture}[xscale=3,yscale=-1]
 \node (C0) at (2,0) {$\cdots$};
 \node (D0) at (3,0) {$\leftsub{\mathscr{O}}(Y, X_{d+1}[-d])$};
 \node (A1) at (0,1) {$\leftsub{\mathscr{O}}(Y, X_0)$};
 \node (B1) at (1,1) {$\leftsub{\mathscr{O}}(Y, X_1)$};
 \node (C1) at (2,1) {$\cdots$};
 \node (D1) at (3,1) {$\leftsub{\mathscr{O}}(Y, X_{d+1})$};
 \node (A2) at (0,2) {$\leftsub{\mathscr{O}}(Y, X_0[d])$};
 \node (B2) at (1,2) {$\leftsub{\mathscr{O}}(Y, X_1[d])$};
 \node (C2) at (2,2) {$\cdots$};
 \draw [->] (C0) -- (D0);
 \draw [->] (D0) -- (3.5,0) arc (-90:90:.25) -- (-.5,.5) arc (270:90:.25) -- (A1);
 \draw [->] (A1) -- (B1);
 \draw [->] (B1) -- (C1);
 \draw [->] (C1) -- (D1);
 \draw [->] (D1) -- (3.5,1) arc (-90:90:.25) -- (-.5,1.5) arc (270:90:.25) -- (A2);
 \draw [->] (A2) -- (B2);
 \draw [->] (B2) -- (C2);
\end{tikzpicture} \]
and a similar long exact sequence for the contravariant functor $\Hom_{\mathscr{O}}(-, Y)$.
\end{lemma}

\begin{proof}
This is a special case of \cite[Lemma~4.3]{IO2}.
\end{proof}

\subsection{Definition of the \texorpdfstring{$(d+2)$}{(d+2)}-angulated cluster category}

By the discussion at the beginning of the section, we want our cluster category to be $\add M \vee \proj \Lambda[d]$. Hence, in $\mathscr{U}$ of Theorem~\ref{theorem.ct_in_Db} we identify objects $U$ with $\mathbb{S}_d[-d] U$.

\begin{definition} \label{def.clustercat}
Let $\Lambda$ be a $d$-representation finite algebra. The \emph{$(d+2)$-angulated cluster category} of $\Lambda$ is defined to be the orbit category
\[ \mathscr{O}_{\Lambda} = \mathscr{U} / (\mathbb{S}_d[-d]) = \mathscr{U} / (\mathbb{S}_{2d}). \]
\end{definition}

\begin{remark} \label{rem.closed_shift}
Note that $\mathscr{O}_{\Lambda}$ comes with an inclusion into $D^b(\mod \Lambda) / (\mathbb{S}_{2d}) \sub \mathscr{C}_{\Lambda}^{2d}$.

Since the subcategory $\mathscr{U}$ of $D^b(\mod \Lambda)$ is closed under $[d]$ (by Theorem~\ref{theorem.ct_in_Db}) it follows that also $\mathscr{O}_{\Lambda}$ is closed under $[d]$ in $\mathscr{C}_{\Lambda}^{2d}$.
\end{remark}

\begin{observation} \label{obs.fund_domain}
With this definition, objects on $\mathscr{O}_{\Lambda}$ have a unique preimage in $\add M \vee \proj \Lambda[d]$.
\end{observation}

The following theorem shows that our Definition~\ref{def.clustercat} makes sense. It says that $\mathscr{O}_{\Lambda}$ is indeed $(d+2)$-angulated.

\begin{theorem} \label{theo.is_d-ctsubcat}
The subcategory $\mathscr{O}_{\Lambda} \subseteq \mathscr{C}_{\Lambda}^{2d}$ is $d$-cluster tilting. It follows that the category $\mathscr{O}_{\Lambda}$ is standard $(d+2)$-angulated (see Definition~\ref{def.angulated}).
\end{theorem}

We will give a proof of this theorem in the next subsection. Here we point out that Theorem~\ref{theo.is_d-ctsubcat} implies Theorem~\ref{theo.clustercat_exists}.

\begin{proof}[Proof of Theorem~\ref{theo.clustercat_exists}, given Theorem~\ref{theo.is_d-ctsubcat}]
Theorem~\ref{theo.is_d-ctsubcat} says that the category $\mathscr{O}_{\Lambda}$ defined above is standard $(d+2)$-angulated. Now (1) is Observation~\ref{obs.fund_domain}. (2) follows from the definition of $\mathscr{O}_{\Lambda}$ (we forced $[2d]$ to become a Serre functor). For (3) assume $X, Y \in \add (M \oplus \Lambda[d])$. Then
\begin{align*}
\Hom_{\mathscr{O}_{\Lambda}}(X, Y[d]) & = \bigoplus_{i \in \mathbb{Z}} \Hom_{\mathscr{U}}(X, \mathbb{S}_{2d}^i Y[d]) && \text{(by definition of $\mathscr{O}_{\Lambda}$)} \\
& = \Hom_{\mathscr{U}}(X, Y[d]) \oplus \Hom_{\mathscr{U}}(X, \mathbb{S}_{2d} Y[d]) && \text{(since the other summands vanish)} \\
& = \Hom_{\mathscr{U}}(X, Y[d]) \oplus \Hom_{\mathscr{U}}(X, \mathbb{S} Y[-d]) && \text{(by definition of $\mathbb{S}_{2d}$)} \\
& = \Hom_{\mathscr{U}}(X, Y[d]) \oplus D \Hom_{\mathscr{U}}(Y, X[d]) && \text{(since $\mathbb{S}$ is a Serre functor)} \qedhere
\end{align*}
\end{proof}

Finally we state the following theorem, which gives us a handle on understanding cluster tilting objects in $\mathscr{O}_{\Lambda}$. It will be shown in the Subsection~\ref{subsect.ctos}.

\begin{theorem} \label{theo.ct_is_2dct}
An object $T \in \mathscr{O}_{\Lambda}$ is cluster tilting (see Definition~\ref{def.cto}), if and only if it is $2d$-cluster tilting when seen as an object in $\mathscr{C}_{\Lambda}^{2d}$.
\end{theorem}

\subsection{Well-definedness of the \texorpdfstring{$(d+2)$}{(d+2)}-angulated cluster category -- proof of 
Theorem~\texorpdfstring{\ref{theo.is_d-ctsubcat}}{5.7}} \label{subsect.welldef}

The aim of this subsection is to verify that $\mathscr{O}_{\Lambda}$ is indeed a standard $(d+2)$-angulated category. To this end we prove Theorem~\ref{theo.is_d-ctsubcat}, saying that it is a $d$-cluster tilting subcategory of $\mathscr{C}_{\Lambda}^{2d}$. Many of the results we obtain along the way will also be helpful in studying further properties of $\mathscr{O}_{\Lambda}$ in the following subsections.

We start by verifying that $\mathscr{O}_{\Lambda}$ satisfies the first property of $d$-cluster tilting subcategories:

\begin{lemma}
The subcategory $\mathscr{O}_{\Lambda} \subseteq \mathscr{C}_{\Lambda}^{2d}$ is $d$-rigid, that is $\Hom(\mathscr{O}_{\Lambda}, \mathscr{O}_{\Lambda}[i]) = 0$ for $0 < i < d$.
\end{lemma}

\begin{proof}
Let $X, Y \in \mathscr{O}_{\Lambda}$. We may assume they are the images of $\overleftarrow{X}, \overleftarrow{Y} \in \add M \vee \proj \Lambda[d]$. Then
\[ \Hom_{\mathscr{O}_{\Lambda}}(X, Y[i]) = \oplus_j \Hom_{\mathscr{U}}(\overleftarrow{X}, \mathbb{S}_{2d}^j \overleftarrow{Y} [i]) \]
For $j > 0$ we have $\mathbb{S}_{2d}^j \overleftarrow{Y} [i]$ lies in positive degree for any $i \in \{1, \ldots, d-1\}$, and hence $\Hom_{\mathscr{U}}(\overleftarrow{X}, \mathbb{S}_{2d}^j \overleftarrow{Y} [i]) = 0$.

For $j = 0$ and $i \in \{1, \ldots, d-1\}$ we have $\Hom_{\mathscr{U}}(\overleftarrow{X}, \overleftarrow{Y} [i]) = 0$ since $\mathscr{U}$ is $d$-rigid.

For $j < 0$ and $i > 0$ we have $\Hom_{\mathscr{U}}(\overleftarrow{X}, \mathbb{S}_{2d}^j \overleftarrow{Y} [i]) = \Hom_{\mathscr{U}}(\overleftarrow{X}, \mathbb{S}_d^j \overleftarrow{Y} [-dj+i])$. Let $D^{< -d}$ denote the subcategory of $D^b(\mod \Lambda)$ of complexes whose homology is concentrated in degrees $< -d$. Since $-dj+i > d$ we have $\mathbb{S}_d^j \overleftarrow{Y}[-dj+i] \in \mathbb{S}_d^j D^{< -d}$, and since $j < 0$ we have $\mathbb{S}_d^j D^{<-d} \subseteq D^{< -d}$. Therefore $\Hom_{\mathscr{U}}(\overleftarrow{X}, \mathbb{S}_{2d}^j \overleftarrow{Y} [i]) = 0$.
\end{proof}

\begin{remark}
We will speak about $(d+2)$-angles in $\mathscr{O}_{\Lambda}$, meaning sequences of morphisms as in Definition~\ref{def.angulated}(3), even though for the moment we do not know that $\mathscr{O}_{\Lambda}$ is standard $(d+2)$-angulated.

In particular it is not yet clear that any morphism can be completed to a $(d+2)$-angle.
\end{remark}

\begin{proposition} \label{prop.anglefunctor}
The functor $\mathscr{U} \to \mathscr{O}_{\Lambda}$ is a $(d+2)$-angle functor in the sense that it
\begin{itemize}
\item commutes with the $d$-suspension up to a natural isomorphism, and
\item sends $(d+2)$-angles to $(d+2)$-angles.
\end{itemize}
Moreover it commutes with the respective Serre functors.
\end{proposition}

\begin{proof}
This follows from the fact that $\mathscr{U} \to \mathscr{O}_{\Lambda}$ is the restriction of the functor $D^b(\mod \Lambda) \to \mathscr{C}_{\Lambda}^{2d}$, which, by \cite{CC, C_PhD}, is a triangle functor and commutes with the Serre functors.
\end{proof}

\begin{corollary}
Any map in the image of the functor $\mathscr{U} \to \mathscr{O}_{\Lambda}$ can be completed to a $(d+2)$-angle in $\mathscr{O}_{\Lambda}$.
\end{corollary}

\begin{proof}
Since $\mathscr{U}$ is standard $(d+2)$-angulated (Example~\ref{example.U_angulated}), a preimage of our map in $\mathscr{U}$ can be completed to a $(d+2)$-angle in $\mathscr{U}$. Now the claim follows from Proposition~\ref{prop.anglefunctor}.
\end{proof}

\begin{corollary} \label{cor.O_has_ARangle}
Let $X \in \mathscr{O}_{\Lambda}$. Then any map $\varphi \in \Soc_{\mathscr{O}_{\Lambda}} \Hom_{\mathscr{O}_{\Lambda}}(X, \leftsub{\mathscr{O}}{\mathbb{S}} X) \setminus \{0\}$ can be completed to a $(d+2)$-angle in $\mathscr{O}_{\Lambda}$. In this case we call this $(d+2)$-angle an \emph{almost split $(d+2)$-angle}, and say that $X$ has an almost split $(d+2)$-angle.
\end{corollary}

The strategy for the remainder of this subsection is as follows: We show at the same time that $\mathscr{O}_{\Lambda}$ is $d$-cluster tilting in $\mathscr{C}_{\Lambda}^{2d}$, and that its image in the module category of a cluster tilted algebra is $d$-cluster tilting (Theorem~\ref{theo.induces_ct_in_mod}). However, since Theorem~\ref{theo.ct_is_2dct} is not proven yet, we assume that we are given a fixed $T \in \mathscr{O}_{\Lambda}$ which is $2d$-cluster tilting in $\mathscr{C}_{\Lambda}^{2d}$. We set $\Gamma = \End_{\mathscr{O}_{\Lambda}}(T)$. We denote by $\mathscr{M} = \Hom_{\mathscr{O}_{\Lambda}}(T, \mathscr{O}_{\Lambda})$ the image of $\mathscr{O}_{\Lambda}$ under the functor $\Hom_{\mathscr{C}_{\Lambda}^{2d}}(T,-) \colon \mathscr{C}_{\Lambda}^{2d} \to \mod \Gamma$.

\begin{proposition} \label{prop.sink_seq}
For any $M_0 \in \mathscr{M}$ there is an exact sequence
\[ M_{d+1} \mono[30] M_d \to[30] \cdots \to[30] M_1 \to[30] M_0, \]
with $M_i \in \mathscr{M}$, where the rightmost map is a radical approximation. 
\end{proposition}

\begin{proof}
By definition of $\mathscr{M}$ we know that $M_0 = \Hom_{\mathscr{O}_{\Lambda}}(T, X_0)$ for some $X_0 \in \mathscr{O}_{\Lambda}$. Since $\Hom_{\mathscr{O}_{\Lambda}}(T, T[d]) = 0$ we may assume that $X_0$ has no direct summands in $\add T[d]$. By Corollary~\ref{cor.O_has_ARangle} $X_0$ has an almost split $(d+2)$-angle
\[ \mathbb{S}_d X_0 \to[30] X_d \to[30] \cdots \to[30] X_1 \to[30] X_0 \to[30] \mathbb{S}_d X_0[d]. \]
Applying $\Hom_{\mathscr{O}_{\Lambda}}(T,-)$ to this, by Lemma~\ref{lemma.snake_seq}, we obtain an exact sequence where the rightmost map is a radical approximation.
\[ \leftsub{\mathscr{O}_{\Lambda}}(T, X_0[-d]) \to[30] \leftsub{\mathscr{O}_{\Lambda}}(T, \mathbb{S}_dX_0) \to[30] \leftsub{\mathscr{O}_{\Lambda}}(T, X_d) \to[30] \cdots \to[30] \leftsub{\mathscr{O}_{\Lambda}}(T, X_1) \to[30] \leftsub{\mathscr{O}_{\Lambda}}(T, X_0) = M_0. \]
Now, since $X_0$ has no direct summands in $\add T[d]$, we know that $X_0[-d]$ has no direct summands in $\add T$. Hence the leftmost map above vanishes, and we have the desired sequence
\[ \leftsub{\mathscr{O}_{\Lambda}}(T, \mathbb{S}_dX_0) \mono[30] \leftsub{\mathscr{O}_{\Lambda}}(T, X_d) \to[30] \cdots \to[30] \leftsub{\mathscr{O}_{\Lambda}}(T, X_1) \to[30] \leftsub{\mathscr{O}_{\Lambda}}(T, X_0) = M_0. \qedhere \]
\end{proof}

\begin{notation}
We now set
\[ \widetilde{\mathscr{O}}_{\Lambda} = \{X \in \mathscr{C}_{\Lambda}^{2d} \mid \Hom_{\mathscr{C}_{\Lambda}^{2d}}(\mathscr{O}_{\Lambda}, X[i]) = 0 \, \forall i \in \{1, \ldots, d-1\} \}, \]
and $\widetilde{\mathscr{M}} = \Hom_{\mathscr{C}_{\Lambda}^{2d}}(T, \widetilde{\mathscr{O}}_{\Lambda})$.
\end{notation}

Note that since $\mathscr{O}_{\Lambda}$ is closed under $[d]$ by Remark~\ref{rem.closed_shift}, it follows that $\Hom_{\mathscr{C}_{\Lambda}^{2d}}(\mathscr{O}_{\Lambda}, \widetilde{\mathscr{O}}_{\Lambda}[i]) = 0$ for all $i$ not divisible by $d$. Moreover, since $[2d]$ is the Serre functor on $\mathscr{C}_{\Lambda}^{2d}$, for $i$ nor divisible by $d$ we also have $\Hom_{\mathscr{C}_{\Lambda}^{2d}}(\widetilde{\mathscr{O}}_{\Lambda}, \mathscr{O}_{\Lambda}[i]) = 0$.

We will show that $\mathscr{O}_{\Lambda} \subseteq \mathscr{C}_{\Lambda}^{2d}$ is $d$-cluster tilting by showing (see Corollary~\ref{cor.O_is_ct}) that $\widetilde{\mathscr{O}}_{\Lambda} = \mathscr{O}_{\Lambda}$.

\begin{proposition} \label{prop.O_equiv_M}
The functor $\Hom_{\mathscr{C}_{\Lambda}^{2d}}(T,-)\colon \widetilde{\mathscr{O}}_{\Lambda} \to \widetilde{\mathscr{M}}$ induces an equivalence $\widetilde{\mathscr{O}}_{\Lambda} / (T[d]) \to \widetilde{\mathscr{M}}$.
\end{proposition}

For the proof we need the following observation.

\begin{lemma} \label{lemma.approx_in_O}
Let $X \in \widetilde{\mathscr{O}}_{\Lambda}$. Then there are 
triangles induced from right $T$-approximations $T_i \to X_i$:
\[ \begin{tikzpicture}
 \node (A8) at (1,1) {$T_{d-1}$};
 \node (A7) at (3,1) {$T_{d-2}$};
 \node (A2) at (7,1) {$T_1$};
 \node (A1) at (9,1) {$T_0$};
 \node (B9) at (0,0) {$T_d$};
 \node (B8) at (2,0) {$X_{d-1}$};
 \node (B7) at (4,0) {$X_{d-2}$};
 \node (B1) at (8,0) {$X_1$};
 \node (B0) at (10,0) {$X_0$};
 \node (B0+) at (11,0) {$X$};
 \node at (5,1) {$\cdots$};
 \node at (6,0) {$\cdots$};
 \draw [->] (A8) -- (A7);
 \draw [->] (A2) -- (A1);
 \draw [->] (B9) -- (A8);
 \draw [->] (A8) -- (B8);
 \draw [->] (B8) -- (A7);
 \draw [->] (A7) -- (B7);
 \draw [->] (A2) -- (B1);
 \draw [->] (B1) -- (A1);
 \draw [->] (A1) -- (B0);
 \draw [->] (B8) -- node [pos=.4,fill=white,inner sep=1pt] {$\scriptstyle 1$} (B9);
 \draw [->] (B7) -- node [pos=.4,fill=white,inner sep=1pt] {$\scriptstyle 1$} (B8);
 \draw [->] (B0) -- node [pos=.4,fill=white,inner sep=1pt] {$\scriptstyle 1$} (B1);
 \draw [double distance=1.5pt] (B0) -- (B0+);
\end{tikzpicture} \]
such that the sequence
\[ \Hom_{\mathscr{C}_{\Lambda}^{2d}}(T, T_d) \to[30] \Hom_{\mathscr{C}_{\Lambda}^{2d}}(T, T_{d-1}) \to[30] \cdots \to[30] \Hom_{\mathscr{C}_{\Lambda}^{2d}}(T, T_0) \epi[30] \Hom_{\mathscr{C}_{\Lambda}^{2d}}(T, X_0) \]
is exact.
\end{lemma}

\begin{proof}
Let $T_i \to X_i$ be a $T$-approximation, $X_{i+1}$ its cocone. We have
\begin{align*}
\Hom_{\mathscr{C}_{\Lambda}^{2d}}(T, X_0[j]) & = 0 && \forall j \in \{1, \ldots, d-1\} \\
\intertext{and hence}
\Hom_{\mathscr{C}_{\Lambda}^{2d}}(T, X_1[j]) & = 0 && \forall j \in \{1, \ldots, d\} \\
& \vdots \\
\Hom_{\mathscr{C}_{\Lambda}^{2d}}(T, X_i[j]) & = 0 && \forall j \in \{1, \ldots, d-1+i\}
\end{align*}
In particular $\Hom_{\mathscr{C}_{\Lambda}^{2d}}(T, X_d[j]) = 0 \, \forall j \in \{1, \ldots, 2d - 1\}$, so $X_d \in \add T$ and therefore we rename $X_d$ to $T_d$. Hence we have constructed the approximation triangles of the lemma. The exactness of the sequence now follows from \cite[Lemma~4.3]{IO2}.
\end{proof}

\begin{proof}[Proof of Proposition~\ref{prop.O_equiv_M}]
Since $\Hom_{\mathscr{O}_{\Lambda}}(T, T[d]) = 0$ the functor $\Hom_{\mathscr{C}_{\Lambda}^{2d}}(T,-)\colon \widetilde{\mathscr{O}}_{\Lambda} \to \widetilde{\mathscr{M}}$ factors through $\widetilde{\mathscr{O}}_{\Lambda} / (T[d])$. It is clear from the definition of $\widetilde{\mathscr{M}}$ that the induced functor $\widetilde{\mathscr{O}}_{\Lambda} / (T[d]) \to \widetilde{\mathscr{M}}$ is dense. Thus we only have to see that it is full and faithful.

Let $\varphi \colon X \to Y$ with $X, Y \in \widetilde{\mathscr{O}}_{\Lambda}$ be such that $\Hom_{\mathscr{C}_{\Lambda}^{2d}}(T, \varphi) = 0$. In the notation of Lemma~\ref{lemma.approx_in_O} this means that the composition $[T_0 \to X_0] \varphi$ vanishes. Hence $\varphi$ factors through the map $X_0 \to X_1[1]$, say via $\varphi_1\colon X_1[1] \to Y$. Since $\Hom_{\mathscr{C}_{\Lambda}^{2d}}(T, \widetilde{\mathscr{O}}_{\Lambda}[j]) \subseteq \Hom_{\mathscr{C}_{\Lambda}^{2d}}(\mathscr{O}, \widetilde{\mathscr{O}}_{\Lambda}[j]) = 0$ for $0 < j < d$, the composition $[T_1[1] \to X_1[1]] \varphi_1$ also vanishes, and hence $\varphi_1$ factors through the map $X_1[1] \to X_2[2]$, say via $\varphi_2\colon X_2[2] \to Y$. Iterating this argument we see that $\varphi$ factors through $T_d[d]$. Hence the functor $\Hom_{\mathscr{C}_{\Lambda}^{2d}}(T,-) \colon \widetilde{\mathscr{O}}_{\Lambda} / (T[d]) \to \widetilde{\mathscr{M}}$ is faithful.

Finally let $\varphi$ be a map from $\Hom_{\mathscr{C}_{\Lambda}^{2d}}(T,X)$ to $\Hom_{\mathscr{C}_{\Lambda}^{2d}}(T,Y)$. We take approximation triangles as in Lemma~\ref{lemma.approx_in_O} for $X$, and similarly (with $T_i'$ instead of $T_i$) for $Y$. Clearly the map $\varphi$ induces maps on the projective resolutions. Using Lemma~\ref{lemma.approx_in_O} and the fact that maps between modules of the form $\Hom_{\mathscr{C}_{\Lambda}^{2d}}(T,T')$ with $T' \in \add T$ are representable we find maps $t_i \colon T_i \to T_i'$ as in the following commutative diagram.
\[ \begin{tikzpicture}
 \node (A8) at (1,1) {$T_{d-1}$};
 \node (A7) at (3,1) {$T_{d-2}$};
 \node (A2) at (7,1) {$T_1$};
 \node (A1) at (9,1) {$T_0$};
 \node (B9) at (0,0) {$T_d$};
 \node (B8) at (2,0) {$X_{d-1}$};
 \node (B7) at (4,0) {$X_{d-2}$};
 \node (B1) at (8,0) {$X_1$};
 \node (B0) at (10,0) {$X_0$};
 \node (B0+) at (11,0) {$X$};
 \node at (5,1) {$\cdots$};
 \node at (6,0) {$\cdots$};
 \draw [->] (A8) -- (A7);
 \draw [->] (A2) -- (A1);
 \draw [->] (B9) -- (A8);
 \draw [->] (A8) -- (B8);
 \draw [->] (B8) -- (A7);
 \draw [->] (A7) -- (B7);
 \draw [->] (A2) -- (B1);
 \draw [->] (B1) -- (A1);
 \draw [->] (A1) -- (B0);
 \draw [->] (B8) -- node [pos=.2,fill=white,inner sep=1pt] {$\scriptstyle 1$} (B9);
 \draw [->] (B7) -- node [pos=.2,fill=white,inner sep=1pt] {$\scriptstyle 1$} (B8);
 \draw [->] (B0) -- node [pos=.2,fill=white,inner sep=1pt] {$\scriptstyle 1$} (B1);
 \draw [double distance=1.5pt] (B0) -- (B0+);
 \node (A8') at (1,-2) {$T_{d-1}'$};
 \node (A7') at (3,-2) {$T_{d-2}'$};
 \node (A2') at (7,-2) {$T_1'$};
 \node (A1') at (9,-2) {$T_0'$};
 \node (B9') at (0,-3) {$T_d'$};
 \node (B8') at (2,-3) {$Y_{d-1}$};
 \node (B7') at (4,-3) {$Y_{d-2}$};
 \node (B1') at (8,-3) {$Y_1$};
 \node (B0') at (10,-3) {$Y_0$};
 \node (B0+') at (11,-3) {$Y$};
 \node at (5,-2) {$\cdots$};
 \node at (6,-3) {$\cdots$};
 \draw [->] (A8') -- (A7');
 \draw [->] (A2') -- (A1');
 \draw [->] (B9') -- (A8');
 \draw [->] (A8') -- (B8');
 \draw [->] (B8') -- (A7');
 \draw [->] (A7') -- (B7');
 \draw [->] (A2') -- (B1');
 \draw [->] (B1') -- (A1');
 \draw [->] (A1') -- (B0');
 \draw [->] (B8') -- node [pos=.4,fill=white,inner sep=1pt] {$\scriptstyle 1$} (B9');
 \draw [->] (B7') -- node [pos=.4,fill=white,inner sep=1pt] {$\scriptstyle 1$} (B8');
 \draw [->] (B0') -- node [pos=.4,fill=white,inner sep=1pt] {$\scriptstyle 1$} (B1');
 \draw [double distance=1.5pt] (B0') -- (B0+');
 \draw [->] (B9) -- node [pos=.25,left] {$t_d$} (B9');
 \draw [->] (A8) -- node [pos=.7,gap] {$t_{d-1}$} (A8');
 \draw [->] (A7) -- node [pos=.7,right] {$t_{d-2}$} (A7');
 \draw [->] (A2) -- node [pos=.7,left] {$t_1$} (A2');
 \draw [->] (A1) -- node [pos=.7,left] {$t_0$} (A1');
\end{tikzpicture} \]
We now construct maps $f_i \colon X_i \to Y_i$ from left to right such that the resulting diagram is still commutative. Assume the maps $f_{d-1}, \ldots, f_{i+1}$ have already been constructed, and make all the squares they are involved in commutative. We choose $f_i$ to be a cone morphism of the triangles to its left. We only have to show that the square
\[ \begin{tikzpicture}[xscale=2,yscale=1.5]
 \node (A) at (0,1) {$X_i$};
 \node (B) at (1,1) {$T_{i-1}$};
 \node (C) at (0,0) {$Y_i$};
 \node (D) at (1,0) {$T_{i-1}'$};
 \draw [->] (A) -- (B);
 \draw [->] (C) -- (D);
 \draw [->] (A) -- node [left] {$f_i$} (C);
 \draw [->] (B) -- node [left] {$t_{i-1}$} (D);
\end{tikzpicture} \]
commutes. Since the square involving $t_i$ and $t_{i-1}$ commutes it follows that our square commutes up to a map factoring through $X_i \to X_{i+1}[1]$. But $X_{i+1}[1] \in (\add T[1]) * \cdots * (\add T[d-i])$, so for $i > 0$ there are no non-zero maps $X_{i+1}[1] \to T_{i-1}'$.

Hence we can complete the diagram. Now $\varphi$ is the image of $f_0$ by construction.
\end{proof}

\begin{lemma} \label{lemma.M_rigid}
We have $\Ext^i_{\Gamma}(\mathscr{M}, \widetilde{\mathscr{M}}\;) = 0$ for $0 < i < d$.
\end{lemma}

\begin{proof}
Let $X \in \mathscr{O}_{\Lambda}$ and $Y \in \widetilde{\mathscr{O}}_{\Lambda}$. Let $T_j, X_j$ as in Lemma~\ref{lemma.approx_in_O}. We have
\begin{align*}
& \Ext^i_{\Gamma}(\Hom_{\mathscr{C}_{\Lambda}^{2d}}(T, X), \Hom_{\mathscr{C}_{\Lambda}^{2d}}(T, Y)) \\ = & \Hom_{\Gamma}( \underbrace{\Hom_{\mathscr{C}_{\Lambda}^{2d}}(T, X_i)}_{= \Omega^i \Hom_{\mathscr{C}_{\Lambda}^{2d}} (T,X)}, \Hom_{\mathscr{C}_{\Lambda}^{2d}}(T, Y)) / (\text{maps factoring through } \Hom_{\mathscr{C}_{\Lambda}^{2d}}(T, X_i) \to \Hom_{\mathscr{C}_{\Lambda}^{2d}}(T, T_{i-1})).
\end{align*}
As in the proof of Proposition~\ref{prop.O_equiv_M} one sees that any map $\Hom_{\mathscr{C}_{\Lambda}^{2d}}(T,X_i) \to \Hom_{\mathscr{C}_{\Lambda}^{2d}}(T,Y)$ is representable, and hence $\Ext^i_{\Gamma}(\Hom_{\mathscr{C}_{\Lambda}^{2d}}(T, X), \Hom_{\mathscr{C}_{\Lambda}^{2d}}(T, Y))$ is a quotient of
\[ \Hom_{\mathscr{C}_{\Lambda}^{2d}}(X_i, Y) / (\text{maps factoring through } X_i \to T_{i-1}). \]
Using the triangle $X_i \to T_{i-1} \to X_{i-1} \to X_i[1]$ of Lemma~\ref{lemma.approx_in_O}, we see that the above space is a subspace of
\[ \Hom_{\mathscr{C}_{\Lambda}^{2d}}(X_{i-1}[-1], Y). \]
Since there are no maps from $T_{i-2}[-1]$ to $Y$, this is a subspace of $\Hom_{\mathscr{C}_{\Lambda}^{2d}}(X_{i-2}[-2], Y)$. Iterating this argument we see that $\Ext^i_{\Gamma}(\Hom_{\mathscr{C}_{\Lambda}^{2d}}(T, X), \Hom_{\mathscr{C}_{\Lambda}^{2d}}(T, Y))$ is a subquotient of $\Hom_{\mathscr{C}_{\Lambda}^{2d}}(X[-i],Y)$, and this $\Hom$-space vanishes since $X \in \mathscr{O}_{\Lambda}$ and $Y \in \widetilde{\mathscr{O}}_{\Lambda}$.
\end{proof}

\begin{theorem} \label{theorem.M_is_ct}
The subcategory $\mathscr{M} = \widetilde{\mathscr{M}}$ is $d$-cluster tilting in $\mod \Gamma$.
\end{theorem}

\begin{proof}
By Lemma~\ref{lemma.M_rigid} the subcategory $\mathscr{M}$ is $d$-rigid. Hence the exact sequence of Proposition~\ref{prop.sink_seq} is a sink sequence. Now $\mathscr{M}$ is $d$-cluster tilting in $\mod \Gamma$ by \cite[Theorem~2.2(b)]{Iy_n-Auslander}. Since, by Lemma~\ref{lemma.M_rigid}, we have
\[ \widetilde{\mathscr{M}} \subseteq \{X \in \mod \Gamma \mid \Ext^i_{\Gamma}(\mathscr{M}, X) = 0 \, \forall i \in \{1, \ldots, d-1\}\} = \mathscr{M} \]
it follows that $\widetilde{\mathscr{M}} = \mathscr{M}$.
\end{proof}

\begin{corollary} \label{cor.O_is_ct}
The subcategory $\mathscr{O}_{\Lambda} = \widetilde{\mathscr{O}}_{\Lambda}$ is $d$-cluster tilting in $\mathscr{C}_{\Lambda}^{2d}$.
\end{corollary}

\begin{proof}
The equality follows immediately from Proposition~\ref{prop.O_equiv_M} and Theorem~\ref{theorem.M_is_ct}. It then follows that $\mathscr{O}_{\Lambda}$ is $d$-cluster tilting.
\end{proof}

\begin{proof}[Proof of Theorem~\ref{theo.is_d-ctsubcat}]
The first statement is contained in Corollary~\ref{cor.O_is_ct}. For the second statement it only remains to show that $\mathscr{O}_{\Lambda}$ is closed under $\leftsub{\mathscr{C}_{\Lambda}^{2d}}{\mathbb{S}}$. This follows from the facts that $\mathscr{U}$ is closed under $\leftsub{D^b(\mod \Lambda)}{\mathbb{S}}$ by Theorem~\ref{theorem.ct_in_Db}, and the projection $D^b(\mod \Lambda) \to \mathscr{C}_{\Lambda}^{2d}$ commutes with the respective Serre functors by \cite{CC, C_PhD}.
\end{proof}

\subsection{Cluster tilting objects} \label{subsect.ctos}

The first aim of this section is to prove Theorem~\ref{theo.ct_is_2dct}, saying that cluster tilting objects in $\mathscr{O}_{\Lambda}$ are precisely the $2d$-cluster tilting objects in $\mathscr{C}_{\Lambda}^{2d}$ which lie in $\mathscr{O}_{\Lambda}$.

\begin{proof}[Proof of Theorem~\ref{theo.ct_is_2dct}]
Assume $T$ is a $2d$-cluster tilting object in $\mathscr{C}_{\Lambda}^{2d}$, such that $T \in \mathscr{O}_{\Lambda}$. It is part of the definition that $\Hom_{\mathscr{O}_{\Lambda}}(T, T[d]) = \Hom_{\mathscr{C}_{\Lambda}^{2d}}(T, T[d]) = 0$, hence we only have to check (2) of Definition~\ref{def.cto} in order to show that $T$ is cluster tilting. This however is an immediate consequence of Lemma~\ref{lemma.approx_in_O} (and the definition of $(d+2)$-angles).

Now assume conversely that $T$ is cluster tilting in $\mathscr{O}_{\Lambda}$ in the sense of Definition~\ref{def.cto}. It follows that $\Hom_{\mathscr{C}_{\Lambda}^{2d}}(T, T[i]) = 0 \, \forall i \in \{1, \ldots, 2d-1\}$: For $i = d$ this is Definition~\ref{def.cto}(1), for $i \neq d$ it follows since $T \in \mathscr{O}_{\Lambda}$. It follows that $(\add T[i]) * (\add T) = \add (T \oplus T[i])$ for $i \in \{0, \ldots, 2d-2\}$. Now note that
\begin{align*}
 \mathscr{C}_{\Lambda}^{2d} & = \mathscr{O}_{\Lambda} * \mathscr{O}_{\Lambda}[1] * \cdots * \mathscr{O}_{\Lambda}[d-1] \\
& = ((\add T) * \cdots * (\add T)[d]) * \cdots * ((\add T) * \cdots * (\add T)[d])[d-1]
\end{align*}
where the first equality comes from the fact that $\mathscr{O}_{\Lambda}$ is $d$-cluster tilting in $\mathscr{C}_{\Lambda}^{2d}$, and the second one follows from Definition~\ref{def.cto}(2). Repeatedly applying the above observation we see that this is
\[ (\add T) * \cdots * (\add T[2d-1]). \]
Hence $T$ is $2d$-cluster tilting in $\mathscr{C}_{\Lambda}^{2d}$.
\end{proof}

\begin{corollary}
Let $T \in \add M$ be a tilting $\Lambda$-module. Then $\overline{T} \in \mathscr{O}_{\Lambda}$ is cluster tilting.
\end{corollary}

\begin{proof}
Since $\gld \Lambda \leq d$ we have $\gld \End(\overline{T}) \leq 2d$. Hence it follows from Theorem~\ref{theorem.amiot} that $\overline{T} \in \mathscr{C}_{\Lambda}^{2d}$ is $2d$-cluster tilting. Now the claim follows from Theorem~\ref{theo.ct_is_2dct}.
\end{proof}

Note that the above Corollary is the first statement of Theorem~\ref{theo.cta_is_trivext}. We now complete its proof.

\begin{proof}[Proof of Theorem~\ref{theo.cta_is_trivext}]
We have just seen that $\overline{T}$ is cluster tilting in $\mathscr{O}_{\Lambda}$. Now, by \cite{CC,C_PhD} we know that $\End_{\mathscr{C}_{\Lambda}^{2d}}(T)$ is the tensor algebra of $\Ext_{\Gamma}^{2d}(D \Gamma, \Gamma)$ over $\Gamma$, and moreover that
\begin{align*}
\Ext_{\Gamma}^{2d}(D \Gamma, \Gamma)^{\otimes i} & = \Hom_{D^b(\mod \Gamma)}(\Gamma, \mathbb{S}_{2d}^{-i} \Gamma). \\
\intertext{Since $T$ is a tilting $\Lambda$-module we may identify along the equivalence $R\Hom_{\Lambda}(T,-)\colon D^b(\mod \Lambda) \to D^b(\mod \Gamma)$ and obtain}
\Ext_{\Gamma}^{2d}(D \Gamma, \Gamma)^{\otimes i} & = \Hom_{D^b(\mod \Lambda)}(T, \mathbb{S}_{2d}^{-i} T).
\end{align*}
The claim now follows since this space vanishes for $i \not\in \{0,1\}$.
\end{proof}

We conclude this subsection by summing up that we have also proven Theorem~\ref{theo.induces_ct_in_mod}.

\begin{proof}[Proof of Theorem~\ref{theo.induces_ct_in_mod}]
By Theorem~\ref{theo.ct_is_2dct} a cluster tilting object in $\mathscr{O}_{\Lambda}$ is a $2d$-cluster tilting object in $\mathscr{C}_{\Lambda}^{2d}$. Hence we are in the situation of Subsection~\ref{subsect.welldef}. Now the claim follows from Theorem~\ref{theorem.M_is_ct}.
\end{proof}

\subsection{Exchange \texorpdfstring{$(d+2)$}{(d+2)}-angles -- proof of 
Theorem~\texorpdfstring{\ref{theo.exch_angles}}{5.7}}

Throughout this section we assume $T$ to be a cluster tilting object in $\mathscr{O}_{\Lambda}$. We start by proving the easy implications of Theorem~\ref{theo.exch_angles}.

First note that the implication (1) \then{} (2) follows immediately from the definition of cluster tilting.

\begin{proof}[Proof of Theorem~\ref{theo.exch_angles}, (3) \then{} (1) or (4) \then{} (1)]
Since (1) has no preference for left or right, the proofs of (3) \then{} (1) and (4) \then{} (1) are the same. Hence we may restrict to the latter situation here. So assume we have a $(d+2)$-angle as in (4). By definition this comes from a diagram
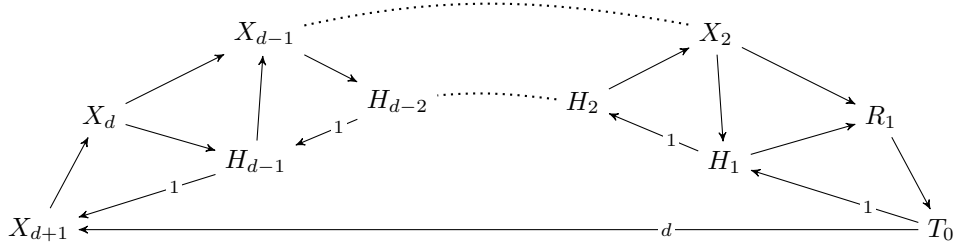
\[ \begin{tikzpicture}[xscale=3,yscale=1.5]
 \node (A) at (180:2) {$X_{d+1}$};
 \node (B) at (150:2) {$X_d$};
 \node (C) at (120:2) {$X_{d-1}$};
 \node (D) at (60:2) {$X_2$};
 \node (E) at (30:2) {$R_1$};
 \node (F) at (0:2) {$T_0$};
 \draw [->] (A) -- (B);
 \draw [->] (B) -- (C);
 \draw [dotted,thick,out=20,in=160] (C) to (D);
 \draw [->] (D) -- (E);
 \draw [->] (E) -- (F);
 \draw [->] (F) -- node [pos=.3,fill=white,inner sep=1pt] {$\scriptstyle d$} (A);
 \node (L) at (150:1.2) {$H_{d-1}$};
 \node (M) at (110:1.2) {$H_{d-2}$};
 \node (N) at (70:1.2) {$H_2$};
 \node (O) at (30:1.2) {$H_1$};
 \draw [->] (F) -- node [pos=.3,fill=white,inner sep=1pt] {$\scriptstyle 1$} (O);
 \draw [->] (O) -- node [pos=.3,fill=white,inner sep=1pt] {$\scriptstyle 1$} (N);
 \draw [dotted,thick,out=170,in=10] (N) to (M);
 \draw [->] (M) -- node [pos=.3,fill=white,inner sep=1pt] {$\scriptstyle 1$} (L);
 \draw [->] (L) -- node [pos=.3,fill=white,inner sep=1pt] {$\scriptstyle 1$} (A);
 \draw [->] (B) -- (L);
 \draw [->] (L) -- (C);
 \draw [->] (C) -- (M);
 \draw [->] (N) -- (D);
 \draw [->] (D) -- (O);
 \draw [->] (O) -- (E);
\end{tikzpicture} \]
in $\mathscr{C}_{\Lambda}^{2d}$. By Theorem~\ref{theo.ct_is_2dct} we know that $T_0 \oplus R$ is basic $2d$-cluster tilting in $\mathscr{C}_{\Lambda}^{2d}$. Then it follows from classical mutation of cluster tilting objects iteratively that $H_1 \oplus R$ is basic $2d$-cluster tilting in $\mathscr{C}_{\Lambda}^{2d}$, then that $H_2 \oplus R$ is basic $2d$-cluster tilting in $\mathscr{C}_{\Lambda}^{2d}$, and so on. We conclude that $X_{d+1} \oplus R$ is a basic $2d$-cluster tilting object in $\mathscr{C}_{\Lambda}^{2d}$. Since it lies in $\mathscr{O}_{\Lambda}$, it follows from Theorem~\ref{theo.ct_is_2dct} that $X_{d+1} \oplus R$ is cluster tilting in $\mathscr{O}_{\Lambda}$. Finally note that, since $T_0 \not\in \add R$, the map $R_1 \to T_0$ is not split epi. Hence the map $T_0 \to X_{d+1}[d]$ in the above $(d+2)$-angle does not vanish, and in particular $T_0 \not\iso X_{d+1}$.
\end{proof}

For the proof of the final implications of Theorem~\ref{theo.exch_angles} we need the following observations.

\begin{lemma} \label{lemma.not_epi}
Let $\mathscr{K}$ be a category, $X, Y \in \mathscr{K}$ such that 
$\Hom_{\mathscr{K}}(X,Y)\ne 0$ and $\Rad^n \End_{\mathscr{K}}(Y) = 0$ for some $n$. Let $r \in \mathbb{N}$ and $f_1, \ldots, f_r \in \Rad \End_{\mathscr{K}}(Y)$. Then the map
\[ ((f_1)_*, \ldots, (f_r)_*) \colon \Hom_{\mathscr{K}}(X, Y)^r \to[30] \Hom_{\mathscr{K}}(X,Y) \]
is not onto.
\end{lemma}

\begin{proof}
Assume the map of the lemma is surjective. Then, adding up such maps, we obtain a sequence of surjective maps
\[ \Hom_{\mathscr{K}}(X, Y)^{r^n} \epi[30] \Hom_{\mathscr{K}}(X, Y)^{r^{n-1}} \epi[30] \cdots \epi[30] \Hom_{\mathscr{K}}(X, Y). \]
Hence also their composition is onto, contradicting the assumption that $\Rad^n \End_{\mathscr{K}}(Y) = 0$. 
\end{proof}

\begin{lemma}\label{lemma.everyone_rigid}
Let $X \in \mathscr{O}_{\Lambda}$ be indecomposable. Then $\Hom_{\mathscr{O}_{\Lambda}}(X, X[d]) = 0$.
\end{lemma}

\begin{proof}
We first show that $\Hom_{\mathscr{U}}(X, X[d]) = 0$ for any indecomposable $X \in \mathscr{U}$. By Theorem~\ref{theorem.ct_in_Db} an indecomposable object in $\mathscr{U}$ is of the form $\mathbb{S}_d^i P$, for some indecomposable projective $\Lambda$-module $P$ and $i \in \mathbb{Z}$. Hence
\[ \Hom_{\mathscr{U}}(X, X[d]) = \Hom_{\mathscr{U}}(\mathbb{S}_d^i P, \mathbb{S}_d^i P[d]) = \Hom_{\mathscr{U}}(P, P[d]) = 0. \]
Now the claim of the lemma follows with Theorem~\ref{theo.clustercat_exists}(3).
\end{proof}

To complete the proof of Theorem~\ref{theo.exch_angles}, we still have to show (2) \then{} (3) and (2) \then{} (4). We prove (2) \then{} (3); the proof of (2) \then{} (4) is similar.

\begin{proof}[Proof of Theorem~\ref{theo.exch_angles}, (2) \then{} (3)]
Assume $T_0, T_0^*$ and $R$ are as in (2). Since $T = T_0 \oplus R$ is cluster tilting in $\mathscr{O}_{\Lambda}$ there is a $(d+2)$-angle
\[ \tag{$\star$} T_{d+1} \to[30] T_d \to[30] \cdots \to[30] T_1 \to[30] T_0^* \to[30] T_{d+1}[d] \]
with $T_{d+1}, \ldots, T_1 \in \add T$. Since $T_0^* \not\in \add T$ we may moreover assume that all maps in the $(d+2)$-angle, except for possibly the rightmost one, are radical morphisms. Applying $\Hom_{\mathscr{O}}(T_0^*, -)$ to the $(d+2)$-angle ($\star$), by Lemma~\ref{lemma.snake_seq} we obtain an exact sequence
\[ \leftsub{\mathscr{O}_{\Lambda}}(T_0^*, T_{d+1}[d]) \to[30] \leftsub{\mathscr{O}_{\Lambda}}(T_0^*, T_d[d]) \to[30] \cdots \to[30] \leftsub{\mathscr{O}_{\Lambda}}(T_0^*, T_1[d]) \to[30] \leftsub{\mathscr{O}_{\Lambda}}(T_0^*, T_0^*[d]). \]
By Lemma~\ref{lemma.everyone_rigid} the last term vanishes. For $i \in \{1, \ldots, d+1\}$ we split up $T_i = R_i \oplus T_0^{r_i}$ with $R_i \in \add R$. Since $\Hom_{\mathscr{O}_{\Lambda}}(T_0^*, R[d]) = 0$ the above sequence is isomorphic to
\[ \leftsub{\mathscr{O}_{\Lambda}}(T_0^*, T_0[d])^{r_{d+1}} \to[30] \leftsub{\mathscr{O}_{\Lambda}}(T_0^*, T_0[d])^{r_d} \to[30] \cdots \to[30] \leftsub{\mathscr{O}_{\Lambda}}(T_0^*, T_0[d])^{r_1} \to[30] 0. \]
By Lemma~\ref{lemma.not_epi} none of the maps is onto, unless its target space vanishes. Since $\Hom_{\mathscr{O}_{\Lambda}}(T_0^*, T_0[d]) \neq 0$ (otherwise $\Hom_{\mathscr{O}_{\Lambda}}(T_0^*, T[d]) = 0$, contradicting the fact that $T_0^*$ is not a summand of the cluster tilting object $T$) this implies that $r_i = 0$ for $i \neq d+1$, and hence $T_i \in \add R$ for these $i$.

Next we apply $\Hom_{\mathscr{O}_{\Lambda}}(-, R)$ to the $(d+2)$-angle ($\star$) and obtain the exact sequence
\[ \Hom_{\mathscr{O}_{\Lambda}}(T_d, R) \to[30] \Hom_{\mathscr{O}_{\Lambda}}(T_{d+1}, R) \to[30] \underbrace{\Hom_{\mathscr{O}_{\Lambda}}(T_0^*[-d], R)}_{=0}. \]
Hence the map $T_{d+1} \to T_d$ is a left $R$-approximation. Since the map $T_{d+1} \to T_d$ lies in the radical it follows that $T_{d+1} \in \add T_0$. Moreover, by the uniqueness of Lemma~\ref{lemma.angle_completion}(2) it follows that $T_{d+1}$ is indecomposable (otherwise the entire $(d+2)$-angle would decompose into a direct sum of several $(d+2)$-angles). Hence $T_{d+1} \iso T_0$. Summing up we have shown that ($\star$) is precisely the first $(d+2)$-angle in Theorem~\ref{theo.exch_angles}.
\end{proof}

We conclude this subsection by noting that the ``moreover'' part of Theorem~\ref{theo.exch_angles} also follows from the explicit description of the $(d+2)$-angle ($\star$) in the proof of (2) \then{} (3) above.

\section{\texorpdfstring{$(d+2)$}{(d+2)}-angulated cluster categories of type \texorpdfstring{$A$}{A}} \label{section.cluster_cat_A}

In this section we study the $(d+2)$-angulated cluster categories (see Section~\ref{section.cluster_cat}) of the iterated Auslander algebras of linearly oriented $A_n$ (see Section~\ref{section.higher_Aus_An}) more explicitly.

We will index the indecomposable objects in the $(d+2)$-angulated cluster category $\mathscr{O}_{A_n^d}$ by $\IndexC{n+2d+1}$ (see Definition~\ref{def:index}), as we make precise in Subsection~\ref{proofsix}. We write $O_{i_0, \ldots, i_d}$ for the indecomposable object corresponding to $(i_0, \ldots, i_d) \in \IndexC{n+2d+1}$.
With this notation we have the following.

\begin{proposition} \label{prop.cluster_ext=intertw}
Let $(i_0, \ldots, i_d), (j_0, \ldots, j_d) \in \IndexC{n+2d+1}$. Then
\[ \Hom_{\mathscr{O}_{A_n^d}}(O_{i_0, \ldots, i_d}, O_{j_0, \ldots, j_d}[d]) \neq 0 \iff \big[ (i_0, \ldots, i_d) \wr (j_0, \ldots, j_d) \text{ or } (j_0, \ldots, j_d) \wr (i_0, \ldots, i_d) \big], \]
and in this case the $\Hom$-space is one-dimensional.
\end{proposition}

Note that the condition on the right-hand side of Proposition~\ref{prop.cluster_ext=intertw} is equivalent to the $d$-simplices in the cyclic polytope $C(n+2d+1,2d)$ corresponding to $(i_0, \ldots, i_d)$ and $(j_0, \ldots, j_d)$ intersecting in their interior.

To give an explicit description of the exchange $(d+2)$-angles of Theorem~\ref{theo.exch_angles} for the specific $d$-representation finite algebras $A_n^d$, we need the following notation:

\begin{definition} \label{def.m_n}
For $(i_0, \ldots, i_d), (j_0, \ldots, j_d) \in \IndexC{n+2d+1}$ with $(i_0, \ldots, i_d) \wr (j_0, \ldots, j_d)$, and $X \subseteq \{0, \ldots, d\}$, we set
\begin{align*}
m_X((i_0, \ldots, i_d), (j_0, \ldots, j_d)) &= \sort(\{i_x \mid x \in X\}\cup \{j_x \mid x \not\in X\}) \\
n_X((i_0, \ldots, i_d), (j_0, \ldots, j_d))&= \sort(\{i_x\mid x \in X\} \cup \{j_{x-1}\mid x\not\in X \})
\end{align*}
Here we write $\sort(K)$ for the tuple consisting of the elements of the set $K$ in increasing order. In the definition of $n_X$, we interpret $j_{-1}$ as $j_d$.

Note that $m_X$ has already been introduced in Section~\ref{section.higher_Aus_An}, above Theorem~\ref{theo.tilt-exchange}.
\end{definition}

\begin{theorem} \label{theo.exch_angles_A}
Let $T \oplus O_{i_0, \ldots, i_d}$ be a cluster tilting object in $\mathscr{O}_{A_n^d}$. Assume there is some $O_{j_0, \ldots, j_d} \not\in \add T \oplus O_{i_0, \ldots, i_d}$ such that $\Hom_{\mathscr{O}_{A_n^d}}(T, O_{j_0, \ldots, j_d}[d]) = 0$. Then
\begin{enumerate}
\item $T \oplus O_{j_0, \ldots, j_d}$ is a cluster tilting object in $\mathscr{O}_{A_n^d}$.
\item Either $(i_0, \ldots, i_d) \wr (j_0, \ldots, j_d)$ or $(j_0, \ldots, j_d) \wr (i_0, \ldots, i_d)$.
\item Assume $(i_0, \ldots, i_d) \wr (j_0, \ldots, j_d)$. Then
\begin{enumerate}
\item There is up to scalars one map $O_{j_0, \ldots, j_d} \to O_{i_0, \ldots, i_d}[d]$. Its minimal completion to a $(d+2)$-angle is
\[ O_{i_0, \ldots, i_d} \to[30] E_d \to[30] \cdots \to[30] E_1 \to[30] O_{j_0, \ldots, j_d} \to[30] O_{i_0, \ldots, i_d}[d], \]
with
\[ E_r = \bigoplus_{\substack{X \subseteq \{0, \ldots, d\} \\ |X| = r \\ m_X \in \IndexC{n+2d+1}}} O_{m_X} \]
for $r \in \{1, \ldots, d\}$. (Here $m_X$ is short for $m_X((i_0, \ldots, i_d), (j_0, \ldots, j_d))$.)

Moreover $\oplus_{r=1}^d E_r \in \add T$, and the maps $O_{i_0, \ldots, i_d} \to E_d$ and $E_1 \to O_{j_0, \ldots, j_d}$ are a left $T$-approximation of $O_{i_0, \ldots, i_d}$ and a right $T$-approximation of $O_{j_0, \ldots, j_d}$, respectively.
\item There is up to scalars one map $O_{i_0, \ldots, i_d} \to O_{j_0, \ldots, j_d}[d]$. Its minimal completion to a $(d+2)$-angle is
\[ O_{j_0, \ldots, j_d} \to[30] F_1 \to[30] \cdots \to[30] F_d \to[30] O_{i_0, \ldots, i_d} \to[30] O_{j_0, \ldots, j_d}[d], \]
with
\[ F_r = \bigoplus_{\substack{X \subseteq \{0, \ldots, d\} \\ |X| = r \\ n_X \in \IndexC{n+2d+1}}} O_{n_X} \]
for $r \in \{1, \ldots, d\}$ (Here $n_X$ is short for $n_X((i_0, \ldots, i_d), (j_0, \ldots, j_d))$.)

Moreover $\oplus_{r=1}^d F_r \in \add T$, and the maps $O_{j_0, \ldots, j_d} \to F_1$ and $F_d \to O_{i_0, \ldots, i_d}$ are a left $T$-approximation of $O_{j_0, \ldots, j_d}$ and a right $T$-approximation of $O_{i_0, \ldots, i_d}$, respectively.
\end{enumerate}
\item For $(j_0, \ldots, j_d)\wr(i_0,\ldots,i_d)$ we have the same result as in (3), with $(i_0, \ldots, i_d)$ and $(j_0, \ldots, j_d)$ interchanged throughout (including in the 
definitions of $m_X$ and $n_X$).
\end{enumerate}
\end{theorem}

Together with Sections~\ref{cyclic} and \ref{local} we obtain the following classification of cluster tilting objects in $\mathscr{O}_{A_n^d}$. This is a cluster category version of Theorem~\ref{central}. The proof is obtained from the proof of Theorem~\ref{central} by replacing the reference to Theorem~\ref{theo.tilt-exchange} by a reference to Theorem~\ref{theo.exch_angles_A}.

\begin{theorem} \label{central_cluster}
Triangulations of $C(n+2d+1,2d)$ correspond bijectively to basic cluster tilting objects in $\mathscr{O}_{A^d_n}$; two triangulations are related by a bistellar flip iff the corresponding tilting objects are related by a single mutation.  
\end{theorem}

\subsection{Indexing the indecomposable objects in \texorpdfstring{$\mathscr{O}_{A_n^d}$}{O\_\{A\_n\^{}d\}} -- Proof of 
Proposition~\texorpdfstring{\ref{prop.cluster_ext=intertw}}{6.1}}\label{proofsix}

Our first aim is to make explicit the indexing of indecomposable objects in $\mathscr{O}_{A_n^d}$ by $\IndexC{n+2d+1}$, that is by interior $d$-simplices of the $2d$-dimensional cyclic polytope with $n+2d+1$ vertices.

\begin{construction} \label{const.indexing.U}
For $(i_0, \ldots, i_d) \in \mathbb{Z}^{d+1}$ such that $i_x + 2 \leq i_{x+1}$ for $0 \leq x < d$ and $i_d + 2 \leq i_0 + n + 2d + 1$ we set
\[ U_{i_0, \ldots, i_d} := \leftsub{\mathscr{U}}{\mathbb{S}}_d^{1 - i_0} \leftsub{A_n^d}{P}_{i_1 - i_0 - 1, \ldots, i_d - i_0 - 1}, \]
where $\leftsub{A_n^d}{P}_{i_1 - i_0 - 1, \ldots, i_d - i_0 - 1}$ is the projective $A_n^d$-module as defined in Theorem~\ref{theo.indexing}.
\end{construction}

\begin{lemma} \phantomsection \label{lemma.indexing.U}
\begin{enumerate}
\item The indecomposable objects in $\mathscr{U}$ are precisely
\[ \{ U_{i_0, \ldots, i_d} \mid \forall x \in \{0, \ldots, d-1\} \colon i_x + 2 \leq i_{x+1} \text{ and } i_d + 2 < i_0 + n + 2d + 1 \}. \]
\item We have
\begin{align*}
\leftsub{\mathscr{U}}{\mathbb{S}}_d U_{i_0, \ldots, i_d} & = U_{i_0-1, \ldots, i_d-1} \text{, and} \\
\leftsub{\mathscr{U}}{\mathbb{S}}_d^- U_{i_0, \ldots, i_d} & = U_{i_0+1, \ldots, i_d+1}.
\end{align*}
\item For $1 \leq i_0$ and $i_d \leq n + 2d$ the $U_{i_0, \ldots, i_d}$ defined in Construction~\ref{const.indexing.U} above coincide with the modules $M_{i_0, \ldots, i_d}$ constructed in Theorem~\ref{theo.indexing}.
\end{enumerate}
\end{lemma}

\begin{proof}
(1) follows from the definition of $\mathscr{U}$ in Theorem~\ref{theorem.ct_in_Db}. (2) is immediate from the definition. (3) follows from Theorem~\ref{theo.indexing.props}(1) for $i_0 = 1$, and then from Remark~\ref{rem.tau=S} and Proposition~\ref{prop.tau}.
\end{proof}

We next describe the functor $\leftsub{\mathscr{U}}{\mathbb{S}}_{2d}$.

\begin{lemma} \label{lemma.S2d}
\begin{align*}
\leftsub{\mathscr{U}}{\mathbb{S}}_{2d} U_{i_0, \ldots, i_d} & = U_{i_d-(n+2d+1), i_0, \ldots, i_{d-1}} \text{, and} \\
\leftsub{\mathscr{U}}{\mathbb{S}}_{2d}^- U_{i_0, \ldots, i_d} & = U_{i_1, \ldots, i_d, i_0+n+2d+1}.
\end{align*}
\end{lemma}

\begin{proof}
We only prove the second formula; the proof of the first one is similar. We have
\begin{align*}
\leftsub{\mathscr{U}}{\mathbb{S}}_{2d}^- U_{i_0, \ldots, i_d} & = \leftsub{\mathscr{U}}{\mathbb{S}}_{2d}^- \leftsub{\mathscr{U}}{\mathbb{S}}_d^{1 - i_0} P_{i_1 - i_0 - 1, \ldots, i_d - i_0 - 1} \\
& = \leftsub{\mathscr{U}}{\mathbb{S}}_d^{-(1+i_0)} \leftsub{\mathscr{U}}{\mathbb{S}} P_{i_1 - i_0 - 1, \ldots, i_d - i_0 - 1} \\
& = \leftsub{\mathscr{U}}{\mathbb{S}}_d^{-(1+i_0)} I_{i_1 - i_0 - 1, \ldots, i_d - i_0 - 1} \\
& = \leftsub{\mathscr{U}}{\mathbb{S}}_d^{-(1+i_0)} U_{i_1 - i_0 - 1, \ldots, i_d - i_0 - 1, n+2d} \\
& = U_{i_1, \ldots, i_d, i_0 + n + 2d + 1} \qedhere
\end{align*}
\end{proof}

\begin{definition}
For $(i_0, \ldots, i_d) \in \IndexC{n+2d+1}$ we denote by $O_{i_0, \ldots, i_d}$ the image of $U_{i_0, \ldots, i_d}$ in the $(d+2)$-angulated cluster category $\mathscr{O}_{A_n^d}$.
\end{definition}

Note that for $(i_0, \ldots, i_d) \in \IndexC{n+2d+1}$ the object $U_{i_0, \ldots, i_d}$ is the preimage of $O_{i_0, \ldots, i_d}$ in $\add (\leftsub{A_n^d}{M} \oplus A_n^d[d])$.

\begin{definition}
Define a permutation $\mathbb{S}_d$ of $\IndexC{n+2d+1}$ by:
$$\mathbb{S}_d(i_0,\dots,i_d)=\left\{\begin{array}{ll} (i_0-1,\dots,i_d-1) &\text{if } 
i_0>1,\\
(i_1-1,\dots,i_d-1,n+2d+1)&\text{if } i_0=1.\end{array}\right.$$
Its inverse permutation $\mathbb{S}_d^-$ is given by:
$$\mathbb{S}_d^-(i_0,\dots,i_d)=\left\{\begin{array}{ll} (i_0+1,\dots,i_d+1) &\text{if } 
i_d<n+2d+1,\\
(1,i_0+1,\dots,i_d+1)&\text{if } i_d=n+2d+1.\end{array}\right.$$
This notation is motivated by the second part of the following proposition.
\end{definition}

\begin{proposition} \phantomsection \label{prop.indexing_O}
\begin{enumerate}
\item The indecomposable objects in $\mathscr{O}_{A_n^d}$ are precisely
\[ \{ O_{i_0, \ldots, i_d} \mid (i_0, \ldots, i_d) \in \IndexC{n+2d+1} \}. \]
\item We have
$$\leftsub{\mathscr{O}}{\mathbb{S}}_d O_{i_0, \ldots, i_d} = O_{\mathbb{S}_d(i_0,\dots,i_d)}
\text{ and } 
%\left\{ \begin{array}{ll} O_{i_1-1, \ldots, i_d-1, n+2d+1} & \text{if } i_0 = 1 \\ O_{i_0 - 1, \ldots, i_d - 1} & \text{if } i_1 > 1 \end{array} \right. \text{, and} \\
\leftsub{\mathscr{O}}{\mathbb{S}}_d^- O_{i_0, \ldots, i_d}  =O_{\mathbb{S}_d^-(i_0,\dots,i_d)}.$$ 
%\left\{ \begin{array}{ll} O_{1, i_0 + 1, \ldots, i_{d-1} + 1} & \text{if } i_d = n+2d+1 \\ O_{i_0+1, \ldots, i_d + 1} & \text{if } i_{d+1} < n+2d+1 \end{array} \right. .
%\end{align*}
\end{enumerate}
\end{proposition}

\begin{proof}
Both statements follow from the corresponding statements in Lemma~\ref{lemma.indexing.U}, the fact that $\mathscr{O}_{A_n^d} = \mathscr{U} / (\mathbb{S}_{2d})$, and the explicit description of $\mathbb{S}_{2d}$ in Lemma~\ref{lemma.S2d}.
\end{proof}

We now prove Proposition~\ref{prop.cluster_ext=intertw}.

\begin{proof}[Proof of Proposition~\ref{prop.cluster_ext=intertw}]
For $(i_0, \ldots, i_d), (j_0, \ldots, j_d) \in \IndexC{n+2d+1}$ we have
\[ \Hom_{\mathscr{O}_{A_n^d}}(O_{i_0, \ldots, i_d}, O_{j_0, \ldots, j_d}[d]) = \Hom_{\mathscr{U}}(U_{i_0, \ldots, i_d}, U_{j_0, \ldots, j_d}[d]) \oplus D \Hom_{\mathscr{U}}(U_{j_0, \ldots, j_d}, U_{i_0, \ldots, i_d}[d]) \]
by Theorem~\ref{theo.clustercat_exists}(3). Since $[d] = \mathbb{S}_d \mathbb{S}_{2d}^-$ we have $U_{j_0, \ldots, j_d}[d] = U_{j_1-1, \ldots, j_d-1, j_0+n+2d}$ by Lemmas~\ref{lemma.indexing.U} and \ref{lemma.S2d}. Hence we have
\begin{align*}
& \Hom_{\mathscr{U}}(U_{i_0, \ldots, i_d}, U_{j_0, \ldots, j_d}[d]) \neq 0 \\
\iff & \Hom_{\mathscr{U}}(U_{i_0, \ldots, i_d}, U_{j_1-1, \ldots, j_d-1, j_0+n+2d}) \neq 0 \\
\iff & \Hom_{\mathscr{U}}(\mathbb{S}_d^{1-i_0} P_{i_1-i_0-1, \ldots, i_d-i_0-1}, \mathbb{S}_d^{1-i_0} U_{j_1-i_0, \ldots, j_d-i_0, j_0-i_0+n+2d+1}) \neq 0 \\
\iff & \Hom_{\mathscr{U}}(P_{i_1-i_0-1, \ldots, i_d-i_0-1}, U_{j_1-i_0, \ldots, j_d-i_0, j_0-i_0+n+2d+1}) \neq 0 \\
\iff & (j_1-i_0,\dots,j_d-i_0,j_0-i_0+n+2d+1)\in \Index{n+2d} \\
& \text{ and } \Hom_{A_n^d}(P_{i_1-i_0-1, \ldots, i_d-i_0-1}, M_{j_1-i_0, \ldots, j_d-i_0, j_0-i_0+n+2d+1}) \neq 0 \hspace{0.9cm} \text{(by Lemma \ref{lemma.indexing.U})} \\
\iff & (j_1-i_0,\dots,j_d-i_0,j_0-i_0+n+2d+1)\in \Index{n+2d} \\
 & \text{ and } M_{j_1-i_0, \ldots, j_d-i_0, j_0-i_0+n+2d+1} \text{ has } S_{i_1-i_0-1, \ldots, i_d-i_0-1} \text{ as a composition factor} \\
\iff & 0 < j_1-i_0 < i_1-i_0 < j_2-i_0 < i_2-i_0 < \cdots \\ & \qquad \cdots < j_d-i_0 < i_d-i_0 < j_0-i_0+n+2d+1 \\ & \text{ and } j_0 - i_0 + n + 2d + 1 \leq n + 2d \hspace{5.8cm} (\text{by Theorem \ref{theo.indexing}}) \\
\iff & j_0 < i_0 < j_1 < i_1 < \cdots < j_d < i_d < j_0+n+2d+1 \\
\iff & (j_0, \ldots, j_d) \wr (i_0, \ldots, i_d),
\end{align*}
where the last equivalence holds, since $i_d < j_0+n+2d+1$ holds automatically for $(i_0, \ldots, i_d), (j_0, \ldots, j_d) \in \IndexC{n+2d+1}$.

Similarly one sees that, in the case that the above equivalent statements are true, one has $\dim \Hom_{\mathscr{U}}(U_{i_0, \ldots, i_d}, U_{j_0, \ldots, j_d}[d]) = 1$. 

Summing up, and noting that the cases $(i_0, \ldots, i_d) \wr (j_0, \ldots. j_d)$ and $(j_0, \ldots, j_d) \wr (i_0, \ldots. i_d)$ are mutually exclusive, we obtain the statement of the proposition.
\end{proof}

\subsection{Proof of Theorem~\texorpdfstring{\ref{theo.exch_angles_A}}{6.3}}

We first prove Theorem~\ref{theo.exch_angles_A} except for the description of the terms of the $(d+2)$-angles in (3). This gap will be filled by Proposition~\ref{prop.target_is_angle} below.

\begin{proof}[Proof of Theorem~\ref{theo.exch_angles_A}]
By assumption, Condition~(2) of Theorem~\ref{theo.exch_angles} is satisfied. Thus all the equivalent statements of that theorem hold.

Claim (1) now is just Theorem~\ref{theo.exch_angles}(1).

(2) Clearly we cannot have $(i_0, \ldots, i_d) \wr (j_0, \ldots, j_d)$ and $(j_0, \ldots, j_d) \wr (i_0, \ldots, i_d)$. If neither $(i_0, \ldots, i_d) \wr (j_0, \ldots, j_d)$ nor $(j_0, \ldots, j_d) \wr (i_0, \ldots, i_d)$ then $T \oplus O_{i_0, \ldots, i_d} \oplus O_{j_0, \ldots, j_d}$ has no $d$-self-extensions, contradicting the fact that $T \oplus O_{i_0, \ldots, i_d}$ is a cluster tilting object.

For (3) note that the existence of an essentially unique map $O_{j_0, \ldots, j_d} \to O_{i_0, \ldots, i_d}[d]$ (for (a)) and $O_{i_0, \ldots, i_d} \to O_{j_0, \ldots, j_d}[d]$ (for (b)) follows from Proposition~\ref{prop.cluster_ext=intertw}. By Lemma~\ref{lemma.angle_completion} these maps can be minimally completed to $(d+2)$-angles in an essentially unique way. We postpone the proof that these essentially unique minimal completions have the form described in the theorem to Proposition~\ref{prop.target_is_angle}. By the final statement of Theorem~\ref{theo.exch_angles} we know that in the $(d+2)$-angles of that theorem we have $O_{j_0, \ldots, j_d} \iso Y^{d+1} \iso X_{d+1}$. Thus these $(d+2)$-angles are precisely the minimal completions of $O_{j_0, \ldots, j_d} \to O_{i_0, \ldots, i_d}[d]$ and $O_{i_0, \ldots, i_d} \to O_{j_0, \ldots, j_d}[d]$ to $(d+2)$-angles, respectively. Now by Theorem~\ref{theo.exch_angles}(3 and 4) all the middle terms $E_r$ and $F_r$ of these $(d+2)$-angles lie in $\add T$. It follows from Theorem~\ref{theo.exch_angles} that the maps $O_{i_0, \ldots, i_d} \to E_d$ and $F_d \to O_{i_0, \ldots, i_d}$ are left and right $T$-approximations, respectively. The fact that the maps $O_{j_0, \ldots, j_d} \to F_1$ and $E_1 \to O_{j_0, \ldots, j_d}$ are left and right $T$-approximations, respectively, follows by interchanging the roles of $O_{i_0, \ldots, i_d}$ and $O_{j_0, \ldots, j_d}$.
\end{proof}

\begin{proposition} \label{prop.target_is_angle}
Assume $(i_0, \ldots, i_d), (j_0, \ldots, j_d) \in \IndexC{n+2d+1}$ with $(i_0, \ldots, i_d) \wr (j_0, \ldots, j_d)$.
\begin{enumerate} \renewcommand{\labelenumi}{(\alph{enumi})}
\item There is a $(d+2)$-angle
\[ O_{i_0, \ldots, i_d} \to[30] E_d \to[30] \cdots \to[30] E_1 \to[30] O_{j_0, \ldots, j_d} \to[30] O_{i_0, \ldots, i_d}[d] \]
with
\[ E_r = \bigoplus_{\substack{X \subseteq \{0, \ldots, d\} \\ |X| = r \\ m_X((i_0, \ldots, i_d), (j_0, \ldots, j_d)) \in \IndexC{n+2d+1}}} O_{ m_X((i_0, \ldots, i_d), (j_0, \ldots, j_d))} \]
in $\mathscr{O}_{A_n^d}$.
\item There is a $(d+2)$-angle
\[ O_{j_0, \ldots, j_d} \to[30] F_1 \to[30] \cdots \to[30] F_d \to[30] O_{i_0, \ldots, i_d} \to[30] O_{j_0, \ldots, j_d}[d] \]
with
\[ F_r = \bigoplus_{\substack{X \subseteq \{0, \ldots, d\} \\ |X| = r \\ n_X((i_0, \ldots, i_d), (j_0, \ldots, j_d)) \in \IndexC{n+2d+1}}} O_{ n_X((i_0, \ldots, i_d), (j_0, \ldots, j_d))} \]
in $\mathscr{O}_{A_n^d}$.
\end{enumerate}
\end{proposition}

\begin{proof}
(a) \textsc{Special case I:} $O_{i_0, \ldots, i_d} = O_{j_0, \ldots, j_d}[-d]$. In this case any non-zero map $O_{j_0, \ldots, j_d} \to O_{i_0, \ldots, i_d}[d]$ is an isomorphism, and hence the other terms of the $(d+2)$-angle should vanish. To see this we first note that 
\[ O_{i_0, \ldots, i_d} = O_{j_0, \ldots, j_d}[-d] = \mathbb{S}_d^- O_{j_0, \ldots, j_d} = \left\{ \begin{array}{ll} O_{j_0+1, \ldots, j_d+1} & \text{if } j_d < n+2d+1 \\ O_{1, j_0+1, \ldots, j_{d-1}+1} & \text{if } j_d = n+2d+1 \end{array} \right. . \]
Since we assume $(i_0, \ldots, i_d) \wr (j_0, \ldots, j_d)$, and hence $i_0 < j_0$, we can only be in the latter case above. Now
\[ m_X((i_0, \ldots, i_d), (j_0, \ldots, j_d)) = m_X((1, j_0+1, \ldots, j_{d-1}+1), (j_0, \ldots, j_{d-1}, n+2d+1)), \]
and hence $m_X((i_0, \ldots, i_d), (j_0, \ldots, j_d)) \in \IndexC{n+2d+1}$ only for $X \in \{ \emptyset, \{0, \ldots, d\}\}$.
Thus (a) holds for $O_{i_0, \ldots, i_d} = O_{j_0, \ldots, j_d}[-d]$.

\textsc{Special case II:} $j_d \leq n+2d$ (and hence also $i_d \leq n+2d$). Then $(i_0, \ldots, i_d), (j_0, \ldots, j_d) \in \Index{n+2d}$, and by Proposition~\ref{prop.E_exact} there is an exact sequence
\[ M_{i_0, \ldots, i_d} \mono[30] \overleftarrow{E}_d \to[30] \cdots \to[30] \overleftarrow{E}_1 \epi[30] M_{j_0, \ldots, j_d} \]
in $\mod A_n^d$, with
\[ \overleftarrow{E}_r = \bigoplus_{\substack{X \subseteq \{0, \ldots, d\} \\ |X| = r \\ m_X((i_0, \ldots, i_d), (j_0, \ldots, j_d)) \in \Index{n+2d}}} M_{m_X((i_0, \ldots, i_d), (j_0, \ldots, j_d))}. \]
By Example~\ref{example.U_angulated} this turns into a $(d+2)$-angle in $\mathscr{U}$, and hence, by Proposition~\ref{prop.anglefunctor}, also into a $(d+2)$-angle in $\mathscr{O}_{A_n^d}$. Since during this transfer $M_{m_X((i_0, \ldots, i_d), (j_0, \ldots, j_d))}$ turns into $O_{m_X((i_0, \ldots, i_d), (j_0, \ldots, j_d))}$, the claim follows.

\textsc{General case:} 
If $j_d\leq n+2d$, then we can apply special case II.
So suppose otherwise, namely, that $j_d=n+2d+1$.  

If $i_0 = 1$ and $i_t = j_{t-1} + 1$ for all $1 \leq t \leq d$ then we can appy special case I. Thus we may assume that either $i_0 > 1$ or that there is $t$ such that $i_t > j_{t-1}+1$.

Assume first that $i_0>1$.  Let $(i_0',\dots,i_d')=(1,i_1-i_0+1,i_2-i_0+1,\dots,
i_d-i_0+1)$, and let $(j_0',\dots,j_d')=(j_0-i_0+1,j_1-i_0+1,\dots,
j_d-i_0+1)$.  Special case II applies to $(i_0',\dots,i_d')$ and 
$(j_0',\dots,j_d')$, resulting in a $(d+2)$-angle in $\mathscr{O}_{A^d_n}$.  
The desired $(d+2)$-angle is obtained from that one by applying 
$\leftsub{\mathscr{O}}{\mathbb{S}}_d^{-(i_0-1)}$.

Assume now that we have $t$ such that $i_t > j_{t-1}+1$.  Set 
\[ (i_0',\dots,i_d') = \mathbb{S}_d^{i_{t}-1}(i_0,\dots,i_d) \qquad \text{and} \qquad (j_0',\dots,j_d') = \mathbb{S}_d^{i_{t}-1}(j_0,\dots,j_d). \]
We have that $(i_0',\dots,i_d') \wr (j_0',\dots,j_d')$, and $j_d' = j_{t-1} - i_{t} + 1 + n + 2d + 1 < n + 2d + 1$, so special case II applies,
resulting in a $(d+2)$-angle in $\mathscr{O}_{A^d_n}$.  The desired 
$(d+2)$-angle is obtained from that one by applying 
$\leftsub{\mathscr{O}}{\mathbb{S}}_d^{-(i_{t}-1)}$.

(b) Let
\begin{align*}
(i_0',\ldots, i_d') & = \mathbb{S}_d^{i_0}(i_0,\dots,i_d) = (i_1 - i_0, \ldots, i_d - i_0, n+2d+1) \text{ and} \\ 
(j_0',\ldots, j_d') & = \mathbb{S}_d^{i_0}(j_0,\dots,j_d) = (j_0 - i_0, \ldots, j_d - i_0).
\end{align*}
Note that 
$(j_0',\dots,j_d')\wr (i_0',\dots,i_d')$.  We can therefore apply part
(a) to construct a $(d+2)$-angle 
$$O_{j_0',\dots,j_d'} \to[30] \cdots \to[30] O_{i_0',\dots,i_d'}
\to[30] O_{j_0',\dots,j_d'}[d].$$
One then applies $\leftsub{\mathscr{O}}{\mathbb S}_d^{-i_0}$, and checks
that this is the desired $(d+2)$-angle by applying the definition of
$m_X$ and $n_X$.   
\end{proof}

\section{Tropical cluster exchange relations} \label{sect.tropical}

Define a \emph{generalized lamination} to be a finite collection of 
increasing $(d+1)$-tuples from $\mathbb R\setminus \{1, \ldots, m\}$, 
such that no two intertwine.  We can also think of a generalized lamination as a collection of $d$-simplices in $\mathbb R^{2d}$ with vertices on the moment curve, which do not intersect in their interiors; the increasing $(d+1)$-tuple $(b_0,\dots,b_d)$ corresponds to the convex hull of the points 
$p_{b_i}$. We denote by $\L$ the set of all generalized laminations.

For each increasing $(d+1)$-tuple $A$ from $\{1, \ldots, m\}$ we define a function $I_A \colon \L \to \mathbb{N}$ by setting $I_A(L)$ to be the number of elements of $L$ which intertwine with $A$ (in some order).
This is also equal to the number of intersections of the simplex $A$
with the simplices defined by the lamination.
In this section we show that these functions satisfy certain tropical exchange relations which we shall define, and in which the functions $I_A$ for $A \not\in \IndexC{m}$ function as frozen variables (in other words, they cannot be mutated).

In the case that $d=1$, this was shown by Gekhtman, Shapiro, and Vainshtein \cite{GSV}. (\cite{GSV} considers more 
general situations, where the polygon is replaced by other surfaces.  
See also the work of
Fomin and Thurston \cite{FT} for another perspective and further extensions
of this.)

The next theorem gives the tropical exchange
relation between $I_A$ and $I_B$ where $A$ and $B$ are
exchangeable.

\begin{theorem}\label{th4} Let $A, B \in \IndexC{m}$ such that $A \wr B$. Then we have the following equality of functions $\L \to \mathbb{Z}$:
\begin{align}\label{tropex}
I_A &= \max \Big( \sum_{X \subsetneq \{0,\dots,d\}} 
(-1)^{|X|+d+1} I_{m_X(A,B)},\\ \notag &\qquad\qquad\qquad \sum_{X\subsetneq \{0,\dots,d\}}
 (-1)^{|X|+d+1} I_{n_X(A,B)} \Big) \end{align}
(See Definition~\ref{def.m_n} for $m_X$ and $n_X$.)
\end{theorem}

The relation (\ref{tropex}) is ``tropical'' because it uses the operations $\max(\cdot,\cdot)$ and $+$, rather than $+$ and $\times$.  In the $d=1$ case,
if one replaces $(\max,+)$ in (\ref{tropex})
with $(+,\times)$, one obtains the type $A$ cluster algebra exchange relation.
We do not know how to obtain a meaningful analogue of this for
$d>1$.  

Note that for $d>1$, 
(\ref{tropex}) is not a tropical cluster algebra relation, because of
the signs.  When $d=2$, we get, for example, the following exchange 
relation:
\begin{align*} I_{024}-I_{135} & = 
\max(I_{124}+I_{034}+I_{025} %\\ &\qquad\qquad\qquad
-I_{134}-I_{125}-I_{035},\\
 &\qquad \qquad I_{245}+I_{014}+ I_{023} %\\ &\qquad\qquad\qquad
-I_{013}-I_{145}-I_{235})\end{align*}

This is not a normal tropical cluster algebra relation because the exchanged variables appear with opposite signs on the lefthand side, and the two tropical monomials on the righthand side each include a mixture of signs.

There is a term in (\ref{tropex}) for each summand of each term of the 
exchange $(d+2)$-angles for $O_A$ and $O_B$ in $\mathscr O_{A^d_{m-2d-1}}$ (see Theorem~\ref{theo.exch_angles_A}), but
(\ref{tropex}) also includes terms corresponding to $(d+1)$-tuples which are not separated.  

The statement of Theorem~\ref{th4} was chosen for maximum uniformity.  
It follows from the proof that, if $d$ is even, then the two terms inside
the $\max(\cdot,\cdot)$ are equal, so the theorem could be stated more simply in this
case.

\subsection{Proof of Theorem~\texorpdfstring{\ref{th4}}{7.1}}

Let $\ell$ be an increasing
$(d+1)$-tuple
of non-integers, $\ell=(\ell_0,\dots,\ell_d)$.  
We will also write $\ell$ for the generalized lamination
consisting only of $\ell$.
We begin by considering (\ref{tropex}) on generalized laminations of the form $\ell$.  

\begin{proposition}\label{prop2}
Let $A$ and $B$ be exchangeable $(d+1)$-tuples such that $A\wr B$, and let $\ell$ be as above. Then  exactly one of the following happens:
\begin{enumerate}
\item $\displaystyle \sum_{X\subseteq \{0,\dots,d\}} (-1)^{|X|}I_{m_X(A,B)}(\ell)=0$, or
\item $d$ is odd, and $a_i<\ell_i<b_{i}$ for all $i$.
\end{enumerate}
\end{proposition}

\begin{proof}
Suppose first that $\ell_k<a_k$ for some $k$.  
It follows that if
$\ell$ intertwines $m_X(A,B)$ in either order, it must be that
$\ell \wr m_X(A,B)$ (rather than the reverse).  If, for any $i$ we have
$\ell_i>b_i$, then
$\ell_i>a_i$ as well, so all the terms in (1) are zero.  
Similarly, if $\ell_i<a_{i-1}$ for any $i$, all the terms in
(1) are zero.  Hence we may disregard these cases.

So, for each $i\ne k$, there are three 
possibilities: \begin{itemize}
\item $a_{i-1}<\ell_i<b_{i-1}$
\item $b_{i-1}<\ell_i<a_i$
\item $a_i<\ell_i<b_{i}$
\end{itemize}
In the first case, in order for $I_{m_X(A,B)}(\ell)$ to be nonzero, we must 
have $i-1\in X$.  In the third case, we must have $i\not\in X$.  

If $k=0$, then the above conditions are sufficient.  If $k\ne 0$, 
we also have two possibilities regarding $\ell_k$:
\begin{itemize}
\item $a_{k-1}<\ell_k<b_{k-1}$
\item $b_{k-1}<\ell_k<a_k$
\end{itemize}
In the first case, we must have $k-1\in X$.

If $X$ satisfies all the above criteria, then $\ell \wr m_X(A,B)$.  
This implies that the non-zero values for $I_{m_X(A,B)}(\ell)$ are precisely
those such that $X$ satisfies $X_{\min}\subseteq X \subseteq X_{\max}$,
for certain specific $X_{\min}$ and $X_{\max}$.  If $X_{\min}\ne 
X_{\max}$, then the sum will be zero.  So suppose otherwise.  The above
conditions must therefore have specified $X$ precisely.  We must therefore
have $k\ne 0$, and we must have $k-1\in X$. 
Consider $i=k-1$.  It must have contributed some condition, which cannot
contradict the previous condition, so it must have imposed $k-2\in X$.  
Proceeding similarly, we see that $i=1$ must impose the condition
that $0\in X$, and then there is no further (non-contradictory) 
condition which can be imposed by $i=0$.  Thus, $X_{\min}\ne X_{\max}$, and
the sum is zero.  

The case that there is some $k$ with $\ell_k>b_{k}$ is dealt with
similarly.  

The remaining case is when $a_k<\ell_k<b_k$.  In this case, the only
two nonzero terms in (1) are $I_A(\ell)$ and $I_B(\ell)$.  
If $d$ is even, they have opposite signs and cancel out;
otherwise, they do not cancel, and we are in the situation of (2).
\end{proof}

The following proposition is proved the same way:

\begin{proposition}\label{prop3}
Let $A$ and $B$ be exchangeable $(d+1)$-tuples such that $A\wr B$, and let $\ell$ be as above. Then  exactly one of the following happens:
\begin{enumerate}
\item $\displaystyle \sum_{X\subseteq \{0,\dots,d\}} (-1)^{|X|}I_{n_X(A,B)}(\ell)=0$, or
\item $d$ is odd, and $b_{i-1}<\ell_i<a_i$ for all $i$ (where the condition that
$b_{-1}<\ell_0$ is considered to be vacuously true).
\end{enumerate}
\end{proposition}

We say that $\ell$ is in $m$-special position ($n$-special position) with respect to the pair $A,B$ if it satisfies Condition~(2) of Proposition~\ref{prop2} (of Proposition~\ref{prop3}).

\begin{proof}[Proof of Theorem~\ref{th4}]
Consider (\ref{tropex}) applied on $\ell$.  
By the above two propositions, if $d$ is even, or if 
$\ell$ is neither in $m$- nor $n$-special position, 
then  
the contribution from $\ell$ to both sides of (\ref{tropex}) are equal and,
further, the two terms being maximized are also equal.  In the
remaining case ($d$ odd and $\ell$ in $m$- or $n$-special position),
one checks that the lefthand side of (\ref{tropex}) is $1$, while the
terms on the righthand side are $-1$ and $1$.  

Now we consider (\ref{tropex}) on an arbitrary generalized lamination $L$.  As already observed, the simplices in $L$ which are neither in $m$- nor $n$-special position with respect to $A,B$ give equal contributions to the left-hand side of (\ref{tropex}) and to each of the terms of the
maximum on the right-hand side, so they can be ignored.  If $d$ is even,
we are done also.  Otherwise, note that $L$ cannot have both 
elements which are in $m$-special position and elements 
which are in $n$-special position,
since these would intertwine.  Thus, only one of the two special
positions is allowed, and the contributions from all the terms of 
$L$ in special position therefore appear, with positive sign, in the same term in the 
maximum.  Thus the equality of the theorem holds.
\end{proof}

\section{Higher dimensional phenomena} \label{higherd}

In this section we report on some phenomena appearing in  
the classical $d=1$ case which do 
not persist for larger values of $d$.  In the $d=1$ case,
a maximal rigid object in $\mathscr{O}_{A_{n}^d}$ is cluster tilting.  We 
give examples showing that, for any $d\geq 3$, this statement does not always 
hold. Specifically, we show the following:

\begin{proposition}\label{prop:higherd} For $d\geq 3$, 
there exist 
maximal non-intertwining
subsets of $\IndexC{2d+3}$ which are not of the overall maximal size.
\end{proposition}
 
%\begin{remark} \label{remark.higherd}
In the setup of the cluster categories of Section~\ref{section.cluster_cat} the proposition implies that, for $d \geq 3$, there are maximal rigid objects in $\mathscr{O}_{A_2^d}$ which are not cluster tilting.

A maximal non-intertwining subset of $\Index{2d+3}$ consists of all
the elements of $\Index{2d+3}\setminus \IndexC{2d+3}$ together with
a maximal non-intertwining subset of $\IndexC{2d+3}$ (since the elements
of $\Index{2d+3}\setminus\IndexC{2d+3}$ contain both 1 and $2d+3$, and 
therefore do not intertwine any element of $\Index{2d+3}$).  It follows 
that the statement of the proposition also holds with $\Index{2d+3}$ replacing
$\IndexC{2d+3}$.  
In the representation-theoretic terms of Section~\ref{section.higher_Aus_An} 
this restatement of the proposition implies that for $d \geq 3$ there exist partial tilting modules for $A_3^d$ which cannot be extended to a tilting module in $\add M_3^d$.

Computer experiments have not detected any similar phenomena when $d=2$.

\medskip

We also consider the
simplicial complex $\Delta^d_{n}$ with vertex set $\IndexC{n+2d+1}$, whose
maximal faces correspond to the internal simplices of
triangulations of $C(n+2d+1,2d)$, or equivalently, to cluster tilting objects
in $\mathscr{O}_{A_{n}^d}$.  

Given a simplicial complex $\Delta$ on a vertex set $V$, we say that
vertices $v$ and $w$ are compatible if $\{v,w\}$ is a face of $\Delta$.  
We then say that
$\Delta$ is a \emph{clique complex} if its faces consist of all pairwise
compatible subsets of $V$.  In the classical setting, $\Delta^1_{n}$ 
is a clique complex;
this is a combinatorial expression of the statement we
have already recalled that cluster tilting objects and maximal rigid objects
coincide in the classical cluster category. 
In these terms, Proposition~\ref{prop:higherd} says that $\Delta^d_2$ 
is not a clique complex for $d\geq 3$.  

It is natural to ask about the topology of $\Delta^d_n$.  Many simplicial
complexes which arise in the context of algebraic combinatorics are
\emph{shellable}.  We recall the precise definition in Subsection~\ref{subsect.higherd_profs}; the point
is that if a simplicial complex is shellable, then its homotopy type
admits a very simple description.  
It is classical that $\Delta^1_{n}$ is shellable, 
because it can be realized as the boundary of a convex polytope,
the (simple) associahedron \cite{Lee}, 
and the boundary of a simplicial convex
polytope is shellable \cite{BrM}.

Our result in this direction is a negative one:

\begin{proposition}\label{prop:higherdtwo} 
For $d\geq 2$, the complex $\Delta^d_{2}$ 
is not shellable.
\end{proposition}

\subsection{Proofs for Section~\texorpdfstring{\ref{higherd}}{8}} \label{subsect.higherd_profs}
The elements of $\IndexC{2d+3}$ can be 
arranged in a cycle, in such a fashion that any 
$(d+1)$-tuple is compatible with any other one except the two which are 
maximally distant from it.  The overall maximal size of a 
non-intertwining collection is $d+1$; the non-intertwining collections
of that size consist of 
$d+1$ consecutive entries around the cycle.

For $d=3$, the resulting cycle is below:
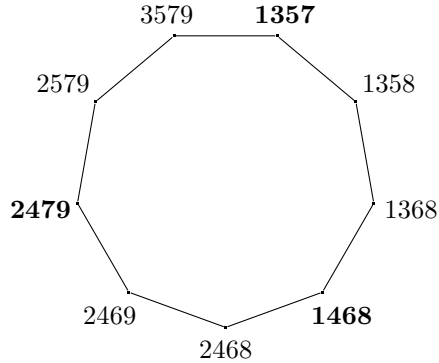
\[ \begin{tikzpicture}
 \node (1+) at (70:2.3) {\bf 1357};
 \node (2+) at (30:2.5) {1358};
 \node (3+) at (350:2.5) {1368};
 \node (4+) at (310:2.4) {\bf 1468};
 \node (5+) at (270:2.3) {2468};
 \node (6+) at (230:2.4) {2469};
 \node (7+) at (190:2.5) {\bf 2479};
 \node (8+) at (150:2.5) {2579};
 \node (9+) at (110:2.3) {3579};
 \node (1) at (70:2) [inner sep=0pt,fill=black] {};
 \node (2) at (30:2) [inner sep=0pt,fill=black] {};
 \node (3) at (350:2) [inner sep=0pt,fill=black] {};
 \node (4) at (310:2) [inner sep=0pt,fill=black] {};
 \node (5) at (270:2) [inner sep=0pt,fill=black] {};
 \node (6) at (230:2) [inner sep=0pt,fill=black] {};
 \node (7) at (190:2) [inner sep=0pt,fill=black] {};
 \node (8) at (150:2) [inner sep=0pt,fill=black] {};
 \node (9) at (110:2) [inner sep=0pt,fill=black] {};
 \draw (1) -- (2);
 \draw (2) -- (3);
 \draw (3) -- (4);
 \draw (4) -- (5);
 \draw (5) -- (6);
 \draw (6) -- (7);
 \draw (7) -- (8);
 \draw (8) -- (9);
 \draw (9) -- (1);
\end{tikzpicture} \]

\begin{proof}[Proof of Proposition~\ref{prop:higherd}] 
If $d\geq 3$, it is possible to choose three $(d+1)$-tuples 
in $\IndexC{2d+3}$ which are pairwise
non-intertwining, but which do not all lie in any consective sequence of
length $d+1$.  Therefore, this collection cannot be extended to a collection
of $d+1$ non-intertwining elements of $\IndexC{2d+3}$.

For example, for $d=3$, we could choose $\{1357,1468,2479\}$ as our starting collection; it is
impossible to increase it to a non-intertwining collection of size $d+1=4$.
\end{proof}

A simplicial complex is called $d$-dimensional if all its maximal
faces contain $d+1$ vertices.  

\begin{definition} 
For $d>0$, a $d$-dimensional simplicial complex
is called \emph{shellable} if its maximal faces admit an order
$F_1,\dots,F_p$ such that for all $i>1$, the intersection of 
$F_i$ with $\bigcup_{j<i} F_j$ is a non-empty union of codimension one faces of 
$F_i$.
\end{definition}

If a $d$-dimensional simplicial complex is shellable, then it is either
contractible or homotopic to the wedge product of some number of 
$d$-dimensional spheres,
\cite[Theorem 1.3]{Bj}.

\begin{proof}[Proof of Proposition~\ref{prop:higherdtwo}]
The simplicial complex $\Delta^d_2$ is 
$d$-dimensional.  Therefore, if $\Delta^d_2$
were  
shellable, it would necessarily either be contractible or be homotopic
to a wedge of some number of $d$-spheres.

The cycle defined above on the vertices of $\Delta^d_2$, 
viewed as a one-dimensional simplicial complex, is a subcomplex of
$\Delta^d_2$, and $\Delta^d_2$ admits a deformation retract to it.  
Thus $\Delta^d_2$ is homotopic to $S^1$.  It follows that for $d\geq 2$
it is not shellable.
\end{proof}

\end{document}